\documentclass[reqno,a4paper]{amsart}
\usepackage[letterpaper, portrait, margin=1.25in]{geometry}
\usepackage{michelstyle}
\usepackage{hyperref}

\newcommand{\dense}{\cal{S}_{c} (\R)}

\newcommand{\supp}{\mathrm{supp}}

\begin{document}
%\maketitle
\title[The Steklov problem for Krein Systems generated by an $A_2 (\R)$ weight]{The Steklov problem and Remainder Estimates for Krein Systems generated by a Muckenhoupt weight}
\author{Michel Alexis}

\address{
\begin{flushleft}
  Michel Alexis: alexism@mcmaster.ca\\\vspace{0.1cm}
	McMaster University\\  Department of Mathematics \& Statistics\\
1280 Main St West, Hamilton, ON, L8S 4L8, Canada\vspace{0.1cm}\\
\end{flushleft}
}

\thanks{I am thankful to Sergey Denisov for helpful discussions and advising on this topic.}

\begin{abstract}
We show that solutions to Krein systems, the continuous frequency analogue of orthogonal polynomials on the unit circle, generated by an $A_2 (\R)$ weight $w$ satisfying $w-1 \in L^1 (\R) + L^2 (\R)$, are uniformly bounded in $L^p _{\mathrm{loc}} (w, \R)$ for $p$ sufficiently close to $2$. This provides a positive answer to the Steklov problem for Krein systems. Furthermore, we define a ``remainder'' which measures the difference between the solution to a Krein system and a polynomial-like approximant, and we estimate these remainders in $L^p _w (\R)$ for $w \in A_2 (\R)$ satisfying some additional conditions. Such polynomial-like approximants, and hence remainder estimates, seem unique to Krein systems, with no analogue for orthogonal polynomials on the unit circle.\end{abstract}

\maketitle
\tableofcontents

\section{Introduction}

Let $d \sigma \ddd w(\lambda) \frac{d \lambda}{2 \pi}$ be a measure on $\R$ with weight $w\geq 0$ satisfying $w-1 \in L^1 (\R) + L^2 (\R)$. Then one can define a family of ``orthonormal continuous polynomials'' $\{P(r, \lambda ; \sigma )\}_{r \geq 0}$ with the following properties.
\begin{enumerate}[label=(\roman*)]
	\item %\label{Krein_property_exp_system}
	  $P(r, \lambda ; \sigma)$ is a ``continuous polynomial'' in $\e{i \lambda}$ of ``degree'' $r$, where $\e{x} \ddd e^{x}$: for each $r \geq 0$ there exists $B(r, \cdot) \in L^2([0,r])$ so that
		\[
			P(r, \lambda ; \sigma) = \e{i \lambda r} + \int\limits_0 ^r B(r, s) \e{i \lambda s} \, ds \, .
			\] 
		\item %\label{Krein_property_orth}  
 $\{P(r, \lambda; \sigma)\}_{r \geq 0}$ is an orthonormal system with respect to $\sigma$, i.e.\ for all $f,g \in L^2 (\R^+)$ we have equality of the inner-products
 \begin{align}
	 \langle f(s), g(s) \rangle_{(ds \, , \, \R_+)} &= \langle \int\limits_0 ^{\infty} f(s) P(s, \lambda;\sigma) \, ds, \int\limits_0 ^{\infty} g(s) P(s, \lambda;\sigma ) \, ds \rangle_{ (d \sigma (\lambda), \R)} \, . \label{orthonormal_system}
 \end{align}
Equivalently, the generalized Fourier transform $\mathcal{O}$ is an isometry from $L^2 (\R^+)$ to $L^2 _{\sigma} (\R)$, where
		\[
		\mathcal{O} f(\lambda) \ddd \int\limits_0 ^{\infty} f(s) P(s, \lambda;\sigma) \, ds \, . 
		\]  
	\item %\label{soln_to_system}
	  $P(r, \lambda; \sigma)$ is the solution to the Krein differential system \eqref{diff_Krein_system_1}, whose coefficients $A(r)$ uniquely determine $P(r, \lambda; \sigma)$. As such $P(r, \lambda; \sigma)$ is referred to as ``the solution to the Krein system.''
\end{enumerate}

For instance, when $d \sigma = \frac{d \lambda}{2 \pi}$, then $P(r, \lambda ; \sigma) = \e{i \lambda r}$, so $\{P(r, \lambda; \sigma) \}_{r \geq 0}$ is simply the Fourier orthonormal system $\{\e{i \lambda r}\}_{r \geq 0}$. See e.g. references \cite{Krein1954, denisov_krein} for more information on Krein systems.

\beginntn When the measure $\sigma$ is clear from context, we simply write $P(r, \lambda)$. And given a weight $w$ we write $P(r, \lambda ; w) \ddd P \left ( r, \lambda; w \frac{d \lambda}{2 \pi} \right)$.\smallskip

First introduced in 1954 by Krein in \cite{Krein1954}, solutions $P(r, \lambda)$ to Krein systems are the continuous frequency analogue of Orthogonal Polynomials on the Unit Circle (OPUC): given a finite measure $\mu$ on the unit circle $\T \subset \C$ with infinite support, the OPUC consist of an orthonormal sequence $\{\phi_n\}_{n \geq 0}$ in $L^2 _{\mu} (\T)$, where each $\phi_n$ is a degree $n$ polynomial over $\C$ with positive leading coefficient. One can obtain this sequence by, e.g., applying the Gram-Schmidt algorithm to the polynomials $\{1, z, z^2, \ldots\}$. See \cite{Simonbook} for a robust reference on OPUC.

Consider the following problem for Krein systems.
\begin{Prob}[The Steklov problem for Krein systems]%\label{Steklov_Krein_Thm}
	If $d \sigma = \frac{w (\lambda) }{2 \pi} \, d \lambda$, what conditions on $w$ guarantee there exists $p > 2$ such that $\{P(r, \lambda;\sigma)\}_{r \geq 0}$ is bounded in $L^p _{loc} (\sigma,\R)$? More precisely, when does there exist $p>2$ such that for each compact $\Delta \subset \R$,
\begin{equation}\label{Steklov_Krein}
	\sup\limits_{r\geq 0} \|P(r, \cdot ; \sigma)\|_{L^p _w (\Delta)} < \infty \, ?
\end{equation}
\end{Prob}

The Steklov problem and estimates like \eqref{Steklov_Krein} are related to the following problem: for which measures $\sigma$ is the maximal function 
\[
  M_{\sigma} (\lambda) \ddd \sup\limits_{r \geq 0} |P(r, \lambda ;\sigma)| 
      \] finite at almost every $\lambda \in \R$? Known as nonlinear Carleson's Theorem, this was recently proved in \cite{poltoratski2021pointwise} for a conjectured class of measures.

The Steklov problem was originally posed for orthogonal polynomials: given the OPUC $\{\phi_n\}_{n \geq 0}$ generated by the measure $d \mu = w \frac{d \theta}{2 \pi}$ on the unit circle $\T$, what conditions on $w$ guarantee there exists $p > 2$ for which
\begin{equation}\label{Steklov_OPUC}
\sup\limits_{n \geq 0} \|\phi_n\|_{L^p _w (\T)} < \infty ?
\end{equation} Nazarov first showed that when $w, w^{-1} \in L^{\infty} (\T)$, then \eqref{Steklov_OPUC} holds for some $p>2$ (\cite{denik1}). Then in \cite{denik2}, Denisov-Rush generalized Nazarov's result to the case when $w, w^{-1} \in \mathrm{BMO} (\T)$. And most recently in \cite{AlexisAD1}, Alexis-Aptekarev-Denisov further generalized these results to the case when $w \in A_2 (\T)$. Since Krein systems are the continuous analogue of OPUC, we adapt the most recent successes in \cite{AlexisAD1} to solutions of Krein systems generated by a weight $w \in A_2 (\R)$ with $A_2 (\R)$ as defined below (see \cite[p.194]{stein}).\smallskip

\noindent{\bf Definition.}
Let $p\in (1,\infty)$. The weight $w\in A_p(\R)$ if
\begin{equation*}%\label{sd_00}
  [w]_{A_p(\R)}\ddd\sup_{I} \, \left \langle w \right \rangle_I \, \left \langle w^{\frac{1}{1-p}} \right \rangle_I^{{p-1}}<\infty,\,
\end{equation*}
where $I$ ranges over the finite intervals in $\R$. \smallskip

Note \eqref{Steklov_Krein} will immediately follow if we can show there exists $\delta = \delta (w)> 0$ such that if $p \in [2, 2+ \delta)$, then
	\begin{equation}\label{P_uniform_bd_basic}
		\sup\limits_{r \geq 0} \| P(r, \lambda ; w)- \e{i \lambda r}\|_{L^p _{w } (\R)} <\infty \, .
	\end{equation}
	Whence the first main theorem of this paper below, for which we recall $L^{p_1} (\R) + L^{p_2} (\R) \subset L^1 (\R) + L^2 (\R)$ when $p_1, p_2 \in [1,2]$.

\begin{Thm}\label{A2_Thm_Krein}
	Suppose that $w-1 =u_1 + u_2$ where $u_1 \in L^{p_1} (\R) , u_2 \in L^{p_2} (\R)$ with $1 \leq p_1 \leq p_2 \leq 2$. If $[w]_{A_2 (\R)} \leq \gamma$, then
	\begin{enumerate} [label=(\alph*)]
	  \item \label{general_Krein_A2}	there exists $\epsilon(\gamma) \in (0, \frac{1}{2})$ with $\lim\limits_{\gamma \to 1} \epsilon (\gamma) = \frac{1}{2}$ such that for any $p \in [p_2, \infty)$ satisfying $\left |\frac{1}{2} - \frac{1}{p} \right | < \epsilon(\gamma)$, we have
			\begin{equation*}%\label{my_Steklov}
		\sup\limits_{r \geq 0 } \|P(r, \lambda; w ) - \e{i \lambda r}\|_{L^p _{w} (\R)} < \infty \, .
			\end{equation*}

		      \item \label{perturbative_Krein_A2} for any $p \in [p_2, \infty)$ there exists $\tau_0 (p) \in (0,1)$ such that whenever $\tau \ddd [w]_{A_2 (\R)} - 1 \leq \tau_0 (p)$, we have
	\begin{equation}\label{perturbative_estimate_A2}
	  \sup\limits_{r \geq 0 } \|P(r, \lambda;w ) - \e{i \lambda r}\|_{L^p _w (\R)} \lesssim_{p} \tau ^{\frac{p-p_2}{2p}} (\tau^{\frac{p_2 - p_1}{2}} \|u_1\|_{L^{p_1} (\R) } ^{p_1}+ \|u_2\|_{L^{p_2} (\R)} ^{p_2})^{\frac{1}{p}}\, .
		\end{equation}
	\end{enumerate}
	\end{Thm}

	Solutions to the Krein system generate generalized eigenfunctions for a unique Dirac operator $\cal{D}$ with spectral measure $2 \, d \sigma$. Furthermore, if $A(r)$, the coefficient of the Krein system \eqref{diff_Krein_system_1}, is real-valued and in $L^2 (\R^+)$, then when written in its canonical form, $\cal{D}^2$ is a diagonal operator whose entries are Schr\"odinger operators. Thus, Theorem \ref{A2_Thm_Krein} allows one to obtain $L^p _{w} (\R)$-information on generalized eigenfunctions for the Dirac and Schr\"odinger operators. See e.g. \cite[Sections 14-16]{denisov_krein} for more details.

	We note the following corollaries of Theorem \ref{A2_Thm_Krein}.

	\begin{corollary}\label{Cor_limits_A2}
	  If $w-1 \in L^{p_1}(\R)  + L^{p_2} (\R)$ for $1\leq p_1 \leq p_2 \leq 2$, then for any $p \in (p_2, \infty)$, we have 
	  \begin{equation*}%\label{my_Steklov_limit}
		\lim\limits_{[w]_{A_2 (\R)} \to 1} \sup\limits_{r \geq 0} \|P(r, \lambda; w) - \e{i \lambda r}\|_{L^p _w (\R)} = 0
			\end{equation*}
			where the limit is taken among weights satisfying $\|w-1\|_{L^{p_1} (\R)+L^{p_2} (\R) } \leq C$, with $C \geq 0$ an arbitrary absolute constant.
		      \end{corollary}
\begin{proof}
Take $\tau \to 0$ in \eqref{perturbative_estimate_A2}.
\end{proof}

Corollary \ref{Cor_limits_A2} states that as $w$ becomes flatter, $P(r, \lambda)$ tends to standard exponential, demonstrating some continuity of solutions to Krein systems with respect to the measure.

\begin{corollary}\label{Cor_Lp_existence}
	Suppose $w-1 \in L^{p_1} (\R) + L^{p_2} (\R)$ with $1 \leq p_1 \leq p_2 < 2$. If $[w]_{A_2 (\R)} \leq \gamma$, then there exists $\delta (\gamma, p_2) > 0$ such that
	\[
		P(r, \lambda; w) - \e{i \lambda r} \in L^p _w (\R)
		\]
		for all $r > 0$, $p \in [2 - \delta, \infty)$.
\end{corollary}
\begin{proof}
Apply Theorem \ref{A2_Thm_Krein} \ref{general_Krein_A2}.
\end{proof}

Corollary \ref{Cor_Lp_existence} is non-trivial, since as we discuss in Section \ref{Decay_P}, a priori all we can say about $P(r, \lambda; w) - \e{i \lambda r}$ is that it belongs in $L^p _w (\R)$ for $p \in [2, \infty)$.

The requirement that $p \geq p_2$ in Theorem \ref{A2_Thm_Krein} is sharp in the sense of Proposition \ref{minimum_p_prop} below, which morally says that the decay of $\lambda \mapsto P(r, \lambda ;w) - \e{i \lambda r}$ cannot exceed the decay of $\lambda \mapsto w(\lambda)-1$.
\begin{Prop}\label{minimum_p_prop}
  If $p_2 \in (1,2]$, then for every $\delta> 0$, there exists a weight $w$ such that $w -1 \in L^{p_2} (\R)$, $[w]_{A_2 (\R)} \leq 1 + \delta$ and 
  \[
    \sup\limits_{r \geq 0} \|P(r, \lambda ; w) - \e{i \lambda r}\|_{L^p _w (\R)} = \infty
  \]
  for every $p < p_2$. 
\end{Prop}

In the same vein as Theorem \ref{A2_Thm_Krein} \ref{general_Krein_A2}, one may also estimate a ``remainder'' term, which we define precisely in Section \ref{remainder_A2_section}. Let $\langle \lambda \rangle \ddd (1+\lambda^2)^{1/2}$. If $\langle \lambda \rangle^k (w-1) \in L^1(\R)$ then the remainder function $R_{k,r} (\lambda)$ exists and satisfies
\begin{equation*}%\label{Remainder_poly}
R_{k,r} (\lambda) = \lambda ^k (P(r, \lambda) - \e{i \lambda r}) - \sum\limits_{l=0}^{k-1} \lambda^{l} a_{k-l, r} (\lambda) 
\, ,
\end{equation*}
where $a_{l, r} (\lambda) \in L^{\infty} (\R)$, $l=1, \dots, k$. Note that $R_{k, r}$ measures how much $\lambda^k (P(r, \lambda) -\e{i \lambda r})$ deviates from a polynomial-like term of ``degree'' $k-1$. In Theorem \ref{Remainder_A2_Thm} below, we estimate $\|R_{1,r}\|_{L^p _w (\R)}$, which quantifies the decay of $P(r, \lambda;w )-\e{i \lambda r}$, providing information with no analogue for OPUC. Furthermore, the assumptions of Theorem \ref{Remainder_A2_Thm} below imply
\[
a_{1,r} (\lambda) = \e{i \lambda r} \alpha_{\infty} (r) + \alpha_2 (r) \, ,
\] where each $\alpha_q (r)$ does not depend on $\lambda$, $\alpha_q \in L^q (\R^+)$, and $\lim\limits_{r \to \infty} \alpha_{\infty} (r)$ exists (see Lemma \ref{a_L2}); as such, the estimate on $\|R_{1,r}\|_{L^p _w (\R)}$ below also qualifies the behavior of the ``continuous polynomials'' $P(r, \lambda;w)$ for $r$ large. We will need the following definitions for what follows. 
\begin{itemize}
  \item Let $H^2 (\Omega)$ denotes the Hardy space in the domain $\Omega \subset \C$.
  \item If $B$ is a Banach space with norm $\| \cdot \|_B$, $(\mu, X)$ is a measure space and $p \in [1, \infty]$, let 
\[
  L^p _{\mu} (X; B) \ddd \left \{ f:X \to B ~\biggr |~ \|f\|_{L^p _{\mu} (X;B)} \ddd \left \| \|f\|_{B} \right \|_{L^p _{\mu} (X)}  < \infty \right \} \, . 
\]
  \item Given two Banach spaces $B_1, B_2$ with norms $\|\cdot\|_{B_1}$ and $\|\cdot\|_{B_2}$, for each $f \in B_1 + B_2$ define 
\[
\|f\|_{B_1 + B_2} \ddd \inf\limits_{f = f_1 + f_2} \|f_1\|_{B_1} + \|f_2\|_{B_2} \, .
	\]
    \end{itemize}

\begin{Thm}\label{Remainder_A2_Thm}
	Suppose that $\langle \lambda \rangle^q (w-1)\in L^1 (\R)$ for some $q > 2$ and
		\begin{equation}\label{Szego_fn_Hardy}
			\mathrm{exp} \left ( \frac{1}{2 \pi i} \int\limits_{-\infty} ^{\infty} \frac{\log (w (t))}{t- \lambda} \, dt \right ) - 1 = \int\limits_0 ^{\infty} h(x) \e{i \lambda x} \, dx  \in H^2 (\C^+) \, .
		      \end{equation} If $[w]_{A_2 (\R)} \leq \gamma$, then there exists $\epsilon (\gamma) > 0$ such that whenever $p \in (1, q]$ satisfies $\left |\frac{1}{p} - \frac{1}{2} \right | < \epsilon (\gamma)$, then
\[
  \|R_{1, r} (\lambda)\|_{L^2 _{d r} (\R^+ ;  L^p _{w}(\R))+L^{\infty} _{d r} (\R^+ ;  L^p _{w}(\R))} < \infty \, .
	\]
\end{Thm}
 
Compare Theorem \ref{Remainder_A2_Thm} and Theorem \ref{A2_Thm_Krein} \ref{general_Krein_A2}, the latter of which can be understood as stating that the $L^{\infty} _{dr} ( \R^+ ; L^p _w (\R))$-norm of $0$th order remainder $R_{0, r}$ is finite.

We can also estimate higher order remainders when the weight $w$ is very close to $1$.
\begin{Thm}\label{Thm_remainder_Linfinity}
  Suppose $(w-1)\langle \lambda \rangle^{k} \in L^1 (\R)$ for some integer $k \geq 0$. If $\|(w-1) \langle \lambda \rangle^k \|_{L^{\infty} (\R)} \leq \delta$, for some $\delta \in [0,1)$, then
\begin{enumerate}[label=(\alph*)]
  \item \label{general_Nazarov_remainder} there exists $\epsilon(\delta) \in (0, \frac{1}{2})$, with $\lim\limits_{\delta \to 0} \epsilon (\delta) = \frac{1}{2}$, such that for all $p$ satisfying $\left | \frac{1}{2} - \frac{1}{p} \right | < \epsilon(\delta)$,
    \[
      \sup\limits_{r > 0} \|R_{k,r} (\lambda)\|_{L^p (\R)} < \infty \, .
    \]
		
	\item \label{perturbative_Nazarov_remainder}
	  for any $p \in  (1, \infty)$, there exists $\delta_0(p) > 0$ such that for all $\delta \in (0, \delta_0 (p))$, we have
	  \[
	    \sup\limits_{r > 0} \|R_{k,r} (\lambda)\|_{L^p (\R)} \lesssim_{p,k} \delta^{1-\frac{1}{p} } \left ( \|\langle \lambda \rangle ^k (w-1) \|_{L^1 (\R)}^{\frac{1}{p}} + \|\langle \lambda \rangle ^k (w-1) \|_{L^1 (\R)} ^{1+\frac{1}{p}} \right )\, .
		\]
	\end{enumerate}
\end{Thm}
In particular Theorem \ref{Thm_remainder_Linfinity} \ref{perturbative_Nazarov_remainder} has the following corollary, which shows some continuity of $R_{k ,r}$ with respect to the measure.
\begin{Cor}
  If $\|(w-1)\langle \lambda \rangle^{k}\|_{L^1 (\R)} \leq C$, where $C$ is any fixed constant and $k$ a nonnegative integer, then for any fixed $1< p < \infty$ we have
	\[
		\lim\limits_{\|(w-1) \langle \lambda \rangle^k \|_{L^{\infty} (\R)} \to 0} \sup\limits_{r > 0} \|R_{k,r} (\lambda)\|_{L^p (\R)} = 0 \, .
		\]

\end{Cor}

While examples of weights to which we can apply Theorem \ref{Thm_remainder_Linfinity} are easy to find, e.g.\ any weight satisfying $|w-1| \leq \frac{\delta}{\langle \lambda \rangle^{k+2}}$, let's mention some weights to which we can apply Theorems \ref{A2_Thm_Krein} and \ref{Remainder_A2_Thm}. Define
\[
	w_{\beta} (\lambda) \ddd \begin{cases} |\lambda|^{\beta} &\text{when } |\lambda|\leq 1 \\ 1 &\text{when } |\lambda| >1\end{cases} \, ,
		\]
		which is an $A_2 (\R)$ weight when $|\beta| < 1$. Then finite products $w(\lambda) \ddd \prod\limits_{j=1}^n w_{\beta_j} (\lambda -\lambda_j)$, where $|\beta_j| < 1$ and $\lambda_1, \ldots, \lambda_n$ are all distinct, satisfy the conditions of Theorems \ref{A2_Thm_Krein} and \ref{Remainder_A2_Thm}.

In practice, condition \eqref{Szego_fn_Hardy} is difficult to check, thus making it difficult to specify weights with infinitely many singularities to which Theorem \ref{Remainder_A2_Thm} applies. However, if we require the singularities be well spaced out, one can generate other weights that satisfy the assumptions of Theorem \ref{A2_Thm_Krein}. Consider for instance 
\[
	w(\lambda) \ddd \, \prod\limits_{j=0} ^{\infty} w_{\beta_j} \left ( \frac{\lambda -2^j}{\epsilon_j} \right ) ^{} \, ,
	\]
	where $\sup\limits_j |\beta_j| < 1$, and $\{\epsilon_j\}_{j \geq 0}$ is a sequence of positive reals such that $\sum\limits_{j} \epsilon_j  < \infty$.

	As far as methods are concerned, the proofs of Theorems \ref{A2_Thm_Krein} and \ref{Remainder_A2_Thm} are based on \cite{AlexisAD1}'s proof of a positive answer to the OPUC Steklov problem for $A_2 (\T)$ weights: this involves the theory of weighted operators, some basic spectral theory and complex interpolation. Meanwhile Theorem \ref{Thm_remainder_Linfinity} is based on the original, elegantly simply method of Nazarov used to prove \cite[Theorem 2.1]{denik1}, which involves just basic knowledge of the Hilbert transform.

	\subsection{Lingering questions: some potentially open problems}\label{future}

	Regarding lingering questions, let us first comment on our methods. Essential to our proof of Theorem \ref{A2_Thm_Krein} is the orthogonality
\[
  P(r, \lambda; w) \perp_{w} \mathrm{Range } \, \P_{[0,r]} \, ,
\] where $\P_{[0,r]}$ is Fourier projection onto the frequency band $[0,r]$. But using for instance Lemma \ref{orthogonal_inner_prod}, we end up showing	\[
		P(r, \lambda; w) \perp_{w q(\lambda)} \mathrm{Range } \, \P_{[0,r]} \, ,
	      \] where $q(\lambda)$ is any nonnegative polynomial. Thus, as far as the algebra in the proof of Theorem \ref{A2_Thm_Krein} is concerned, one is tempted to consider the case $\widetilde{w} \ddd w q(\lambda)\in A_2 (\R)$ and prove another version of this result. But issues arise: for Krein systems and their solutions to be well-defined, the weight $w$ needs to be ``centered'' near $1$, with e.g. $w-1 \in L^1 (\R) + L^2 (\R)$. This directly conflicts with $\widetilde{w} \in A_2 (\R)$. Indeed, heuristically $\widetilde{w} \in A_2 (\R)$ means $\widetilde{w}$ is more or less constant, which would then mean $w$ decays to $0$, which would contradict the required decay of $w-1$. A possible remedy here is to study the more general de Branges canonical systems, wherein the weight $w$ can deviate from $1$ in more exotic ways. I have not investigated this potential framing of the problem and thus so far it remains an open question.

	      While Proposition \ref{minimum_p_prop} shows that the requirement that $p \geq p_2$ is sharp in Theorem \ref{A2_Thm_Krein}, it is unclear to me whether the requirement that $p \leq q$ is sharp in the statement of Theorem \ref{Remainder_A2_Thm}. Furthermore, while Proposition \ref{minimum_p_prop} had a nice interpretation --- that $P(r, \lambda ;w) - \e{i \lambda r}$ cannot decay faster that $w-1$ --- the requirement that $p \leq q$ is more difficult to interpret.  in elucidating this part of Theorem \ref{Remainder_A2_Thm}.

		Finally, if $w \in A_2 (\R)$ and $w-1 \in L^1 (\R) + L^2 (\R)$, then Theorem \ref{A2_Thm_Krein} implies
\begin{equation}\label{P_bd_intermediate_end}
	\sup\limits_{r > 0} \|P(r, \lambda)-\e{i \lambda r}\|_{L^p _w (\R)} < \infty \,
\end{equation}
	for some $p > 2$. While this in turn implies
\begin{equation}\label{Poisson_estimate}
	\sup\limits_{r > 0} \|P(r, \lambda)\|_{L^p _{w(\lambda) \frac{ d \lambda}{1+\lambda^2}} (\R)} < \infty
\end{equation}
for some $p > 2$, I feel the estimate \eqref{Poisson_estimate} is more natural to try and prove directly. I wonder if there are other methods or assumptions which would more readily yield \eqref{Poisson_estimate} without going through \eqref{P_bd_intermediate_end}.

	\subsection{Organization of this Paper}

This paper is organized as follows. First, we outline the basic theory of Krein systems in Section \ref{Krein_system_Basics}. In Section \ref{Decay_P}, we discuss a priori which $L^p _w (\R)$ spaces $P(r, \lambda ; w) - \e{i \lambda r}$ belongs to. In Section \ref{Krein_ortho_lemma_section} we prove a useful orthogonality lemma. We spend Sections \ref{weight_A2_section},  \ref{remainder_A2_section} and \ref{Section_Pf_remainder_Linfinity} proving the main results of this paper, i.e.\ Theorem \ref{A2_Thm_Krein} and Proposition \ref{minimum_p_prop}, Theorem \ref{Remainder_A2_Thm}, and Theorem \ref{Thm_remainder_Linfinity}. In Sections \ref{tech_lemma_pf_section} and \ref{section_I-Qwp_invertible} we prove the technical Lemmas \ref{lemma_mu_weight_Lp} and \ref{thm_inverse_bd_final_basic} essential in the proofs of the results in this paper. And finally, in Appendix \ref{complex_analysis_section}, we prove Proposition \ref{solve_for_X_Steklov_fullp_statement}.

\subsection{Notation}

If $B$ is a Banach space, we denote its dual by $B^*$, and the space of bounded operators on $B$ by $\cal{L} (B)$.

If $X \subset \R^d$ and $d \mu = w(x) \, dx$, we define $L^p _w (X) \ddd L^p _{\mu} (X)$. If $\mu$ is Lebesgue measure, we omit $\mu$, i.e.\ $L^p (X) \ddd L^p _{dx} (X)$. If $T$ is a bounded linear operator between two Banach spaces $L^p _{\mu} (X)$, $L^q _{\nu} (Y)$, we write its norm as $\|T\|_{p,q}$.

If $p\in [1,\infty]$, the dual exponent is denoted by $p'=p/(p-1)$.

For $f \in L^1 _{\mu} (X)$, denote the average of $f$ over the measure space $(\mu, X)$ by
\[
	\langle f\rangle_{X, \mu} \ddd \frac{1}{\mu(X)}\int_X fd\mu\,.
\]
		If $X \subset \R^d$ and $\mu$ is Lebesgue measure, we will simply write $\langle f\rangle_{X}$.

In this paper we work with spatial variable $\lambda \in \R$ and frequency variable $r \in \R^+$. Thus for Banach spaces like $L^p _{\sigma} (\R)$ or $L^p _w (\R)$, the variable of integration is understood to be $\lambda$, i.e.\ $\|f(r, \lambda)\|_{L^p _w (\R)} \ddd \|f(r, \lambda)\|_{L^p _{w(\lambda) d \lambda} }$. And given a measure, weight or function over $\R$, we write them with respect to variable $\lambda$, i.e.\ $w = w(\lambda)$.

Given a measure space $(\mu, X)$, we denote the $L^2_{\mu} (X)$ inner product by
		\[
			\langle f, g \rangle_{(\mu, X)} \ddd \int\limits_X f(x) \ov{g(x)} \, d\mu (x) \, .
			\]
			If it is clear from context that $X = \R^d$, we will simply write $\langle f, g \rangle_{\mu}$.

			If $\e{x} \ddd e^x$, then the Fourier transform $\hat{f}$, or $\cal{F}f$, of a function $f$ and its inverse transform $\check{f}$, or $\cal{F}^{-1} f$ are given by
		\[
			\widehat{f} (\xi) \ddd \int\limits_{- \infty} ^{\infty} f(x) \e{- 2 \pi i x \xi} \, dx \, , \quad \widecheck{f} (x) \ddd \int\limits_{-\infty}^{\infty} f(\xi) \e{2 \pi i x \xi} d \xi \, .
					\] 

Given a set $A\subseteq X$ where $X = \R^d$, we will use notation $A^c$ for its complement, i.e., $A^c=X\backslash A$.

Let $C_c^\infty(U)$ denote the space of smooth functions compactly supported in $U$, an open subset of $\R^d$.

For two non-negative functions
$f_{1}$ and $f_2$, we write $f_1\lesssim f_2$ if  there is an absolute
constant $C$ such that
\[
f_1\le Cf_2
\]
for all values of the arguments of $f_1$ and $f_2$. If the constant depends on a parameter $\alpha$, we will write $f_1\lesssim_\alpha f_2$. We define $\gtrsim$
similarly and write $f_1\sim f_2$ if $f_1\lesssim f_2$ and
$f_2\lesssim f_1$ simultaneously.

If a constant $C$ depends on parameter $\alpha$, we will express this by writing $C(\alpha)$ or $C_{\alpha}$.

	\bigskip

\section{Krein Systems Basics}\label{Krein_system_Basics}

In this section we explain the basics of Krein systems following \cite{denisov_krein}, listing most facts needed for this paper without proof; see \cite{denisov_krein} for an accessible in-depth survey on Krein systems.

Suppose $\sigma$ is a Poisson-finite measure over $\R$, i.e.\
\[
	\int\limits_{\R} \frac{d \sigma (\lambda)}{1+\lambda^2} < \infty \, .
	\]

	\begindef A function $H: \R \to \C$ is \textit{Hermitian} if $H(-x) = -\overline{H(x)}$. 

	\begindef A Hermitian function $H \in L^2 _{loc} (\R)$ is the \textit{accelerant} associated to the measure $\sigma$ if there exists a real constant $\beta$ such that for all $x \in \R$,
\begin{equation}\label{accelerant_eqn}
	\int\limits_0 ^x (x-s) H(s) \, ds = i \beta x + \int\limits_{-\infty} ^{\infty} \left (1 +\frac{i \lambda x}{1+\lambda^2} - \e{i \lambda x} \right ) \frac{1}{\lambda^2} \, \left ( d\sigma (\lambda) - \frac{d \lambda}{2 \pi} \right ) \, .
	\end{equation} \smallskip

	Formally differentiating \eqref{accelerant_eqn} twice yields ``$H (2 \pi x) = \widecheck{( d \sigma- \frac{d \lambda}{2 \pi})} (x)$.'' Thus intuitively the accelerant captures the moments of the signed measure $d \sigma - \frac{d \lambda}{2 \pi}$.

If an accelerant $H$ exists, then for each $r > 0$, define the operator $\HH_r$ on $L^2 ([0,r])$ by
\begin{equation}\label{H_operator_def}
	\HH_r f (x) \ddd \int\limits_0 ^r H (x-y) f(y) \, dy \, .
\end{equation}
Of interest is when $I + \HH_r >0$, which occurs for a large class of measures.

\begin{Lem}\label{Gamma_L2}
  Suppose $d\sigma = w \frac{d\lambda}{2 \pi}$. If $w-1 \in L^1 (\R) + L^2 (\R)$, then 
  \[
  H (x) \ddd \frac{1}{2 \pi} (\widecheck{w-1}) \left ( \frac{x}{2 \pi} \right ) = \mathcal{F}^{-1} (w (2 \pi \cdot) -1)(x)
  \]
  is the accelerant associated to $d \sigma$, and $I + \HH_r >0$.
\end{Lem}
\begin{proof}
  Let $a \in L^1 (\R) + L^2 (\R)$ and first define $H \ddd \frac{1}{2 \pi} \widecheck{a} \left ( \frac{x}{2 \pi} \right ) \in C_0 (\R) + L^2 (\R) \subset L^2 _{loc} (\R)$ and
	\begin{equation}\label{def_beta}
		\beta \ddd \frac{1}{2 \pi} \int\limits_{\R} \frac{\lambda}{1+ \lambda^2} a (\lambda) \, d \lambda \, .
	\end{equation} Suppose we can show that
	\begin{equation}\label{accelerant_eqn_intermediate}
		\int\limits_0 ^x (x-s) H(s) \, ds = i \beta x + \frac{1}{2 \pi} \int\limits_{-\infty} ^{\infty} \left (1 +\frac{i \lambda x}{1+\lambda^2} - \e{i \lambda x} \right ) \frac{1}{\lambda^2} a(\lambda) \, d\lambda \,
	\end{equation}
	holds for all $x \in \R$. Then setting $a = w-1$, we will show that $H(x) \ddd \frac{1}{2 \pi} \widecheck{(w-1)} \left ( \frac{x}{2 \pi} \right )$ is the accelerant associated to $d \sigma = \frac{w (\lambda)}{2 \pi} d \lambda$, i.e.\ we will show \eqref{accelerant_eqn_intermediate} holds.

	We now note that $I + \mathcal{H}_r > 0$ : take $f \in L^2 ([0,r])$ and extend $f$ to a function over $\R$ by defining $f(x) = 0$ for $x \notin [0,r]$. Then Fourier inversion for distributions yields
		\begin{align*}
		  \langle (I+ \mathcal{H}_r )f, f \rangle_{(dx, \R)} = \langle f, f \rangle_{(dx, \R)} + \langle H \ast f, f \rangle_{(dx, \R)} &= \langle \hat{f} , \hat{f} \rangle_{(d\lambda, \R)} + \langle (w(2 \pi \cdot)-1) \hat{f}, \hat{f} \rangle_{(d\lambda, \R)} \\
		&= \|\hat{f}\|_{L^2 _{w(2 \pi \cdot)} (\R)} ^2 \\
			&> 0
		\end{align*}
			whenever $f \neq 0$, i.e.\ $I + \mathcal{H}_r > 0$.

			It remains to show \eqref{accelerant_eqn_intermediate} holds for $H, \beta, a$ as above. If we can show \eqref{accelerant_eqn_intermediate} holds for $a \in L^1(\R)$ and for $a \in L^2 (\R)$, then by linearity \eqref{accelerant_eqn_intermediate} holds for all $a \in L^1 (\R) + L^2 (\R)$.

			\textit{The case that $a \in L^1 (\R)$:} First note that both sides of \eqref{accelerant_eqn_intermediate} are well-defined functions for $a \in L^1$. Next note that the second derivative (with respect to variable $x$) of the left side of \eqref{accelerant_eqn_intermediate} equals the second derivative of the right side of \eqref{accelerant_eqn_intermediate}. Indeed, this follows from the dominated convergence theorem and that $a \in L^1 (\R)$ implies $H \in C(\R)$.

	Thus the first derivatives of each side of \eqref{accelerant_eqn_intermediate} will be identical so long as we can verify they agree at a point, say $x=0$. But $\beta$ is defined precisely to make this occur; one can check this using the dominated convergence theorem.

	Thus to verify \eqref{accelerant_eqn_intermediate} holds for all $x$, it suffices to verify each side of \eqref{accelerant_eqn_intermediate} agrees at a point, say $x=0$. Direct computation shows each side of \eqref{accelerant_eqn_intermediate} is $0$ at $x=0$. This completes the proof of this case.

	\textit{The case that $a \in L^2 (\R)$:} First let $a _n \in L^1 (\R)\cap L^2 (\R)$ approximate $a$ in $L^2 (\R)$, e.g.\ define
	\[
		a _n \ddd \chi_{[-n,n]} a \, .
		\]
		By the previous case, if $H_n (x) \ddd \frac{1}{2 \pi} \widecheck{a_n} \left (\frac{x}{2 \pi} \right )$, and $\beta_n = \frac{1}{2 \pi} \int\limits_{\R} \frac{\lambda}{1+\lambda^2} a_n (\lambda) \, d\lambda$, then \eqref{accelerant_eqn_intermediate} holds for $(H_n, \beta_n, a_n)$. We note that $H_n \to H$ in $L^2 (\R)$ by Plancherel's theorem; we also have $\beta_n \to \beta$ for $\beta$ given by \eqref{def_beta}. Now consider \eqref{accelerant_eqn_intermediate} for $(H_n, \beta_n, a_n)$ and take $n \to \infty$ to get it for $(H, \beta, a)$. This completes the proof.
				\end{proof}

				Assume $H \in L^2 _{loc} (\R)$ is an accelerant for which $I + \HH_r > 0$ for each $r > 0$. Then the resolvent operator
\begin{equation}\label{def_G}
	\G_r \ddd I - (I + \HH_r)^{-1}
\end{equation}
	 is an integral operator on $L^2 ([0,r])$ with kernel $\Gamma_r (s,t) \in L^2 _{ds \times dt} ( [0,r]^2)$ possessing the following properties.
\smallskip

	 \textbf{Properties of the Resolvent kernel $\Gamma$}
	\begin{enumerate}[label=(\roman*)]
		\item \textit{Resolvent Identity: } from \eqref{def_G} we have
\begin{align}
	\Gamma_r (t,s) + \int\limits_0 ^r H(t-u) \Gamma_r (u,s) \, du &= H(t-s) \, , \label{resolvent_identity_1} \; 0 \leq s,t \leq r \, . %\label{resolvent_identity_2}
\end{align}
\item \textit{Symmetries: }
	$\Gamma$ has following symmetries:
\begin{equation}\label{Gamma_symmetries}
  \Gamma_r (s,t) = \ov{\Gamma_r (t,s)} \, , \quad \Gamma_r(s,t) = \Gamma_r (r-t,r-s) \, , \quad 0 \leq s,t \leq r \, .
\end{equation}
\item \label{regularity_Gamma_property} \textit{Regularity inherited from accelerant:} if $H \in C^k (\R)$ for some $k \geq 0$, then for each $r > 0$, $\Gamma_r (\cdot, \cdot) \in C^k ([0,r]^2)$. Furthermore, $\Gamma_r (s,t)$ is continuously differentiable in $r$ and 
\begin{equation}\label{Gamma_diff}
\partial_r \Gamma_r (s,t) = -\Gamma_r (s,r) \Gamma_r (r,t) \, , \quad 0\leq s,t \leq r\, .
\end{equation}

\item \label{resolvent_cty_property} \textit{Continuity in $L^2$:} Define
	\begin{equation}\label{resolvent_continuity_L2}
		g_r (s) \ddd \begin{cases} \Gamma_r (s,0) &\text{ if } 0 \leq s \leq r \\ 0 &\text{ if } s > r \end{cases} \, .
	\end{equation} 
If $H \in L^2 _{loc} (\R)$, then the mapping $r \mapsto g_r$ is an element of $C([0,R], L^2 ([0,R]))$ for all $R>0$ (see \cite[Chapter 6]{denisov_krein}).
	\end{enumerate}
\smallskip

	Using $\Gamma_r$, one can define the ``continuous polynomials'' $\{P(r, \lambda; \sigma )\}_{r \geq 0}$. These are entire functions $P(r, \cdot \, ; \sigma )$, each of ``degree'' $r$, given by
\begin{align}
	P(r, \lambda;\sigma) &\ddd \e{i \lambda r}  - \int\limits_0 ^r \Gamma_r (r , t) \, \e{i \lambda t} \, dt  = (I + \HH_r)^{-1} (\e{i \lambda \cdot}) (r) \, . \label{def_P}
\end{align}
This definition is motivated by analogy to the OPUC: if one considers the Toeplitz matrix
\[
	\mathbf{T_n} \ddd \begin{pmatrix} c_0 & c_1 & \ldots & c_n \\ c_{-1} & c_0 & \ldots & c_{n-1} \\ \vdots & \vdots & \ddots & \vdots \\ c_{-n} & c_{-n +1} &  \ldots & c_{0}\end{pmatrix} \, ,
	\]
			where $c_j \ddd \int\limits_{\T} z^{-j} \, d \mu$ are the moments of a measure $\mu$ on $\T$, then the orthogonal polynomial $\phi_n (z)$ associated to $\mu$ is given by last row of the column-vector 
	\[
	  \sqrt{\frac{\det \mathbf {T_n}}{\det \mathbf{T_{n-1}}}} \, \,  \mathbf{T_n} ^{-1} \begin{pmatrix} 1 & z & z^2 & \ldots & z^n \end{pmatrix}^\intercal \, .
			\]
			And so similar to the OPUC, $\{P(r, \lambda ; \sigma)\}_{r \geq 0}$ is an orthonormal system with respect to $d \sigma$, i.e.\ \eqref{orthonormal_system} holds for all $f,g \in L^2 (\R^+)$. Furthermore, like how each $z^n$ may be expressed as a sum of the OPUC $\{\phi_k\}_{0 \leq k \leq n}$, each exponential $\e{i \lambda r}$ is a ``continuous sum'' of the polynomials $\{P(s, \lambda ; \sigma \}_{0 \leq s \leq r}$, i.e.\ given $r \in (0, R)$ we have
	\begin{equation}\label{P_basis}
		\e{i \lambda r} = P(r, \lambda ; \sigma) + \int\limits_0 ^r L_R(r,s) P(s, \lambda ; \sigma) \, ds  = (I+\cal{L}_R)P(\cdot, \lambda ; \sigma) (r)\, ,
	\end{equation}
	for some Volterra operator $\cal{L}_R$ bounded on $L^2 ([0,R])$.

      Similar to the role the Verblunsky coefficients $\{\alpha_n\}$ play for OPUC (see e.g.\ \cite[Theorem 1.5.2]{Simonbook}), there exists $A(r) \in L^2 _{loc} (\R^+)$ such that $\begin{pmatrix} P (r, \lambda) & P_* (r, \lambda) \end{pmatrix}^\intercal$ is the solution to the Krein (differential) system
\begin{equation}\label{diff_Krein_system_1}
	\begin{cases}
	P' &= i \lambda P - \ov{A} P_* \, , \quad P(0, \lambda) = 1  \\
	P_* ' &=-AP\, , \quad P_* (0, \lambda) = 1
	\end{cases} \, , \qquad \lambda \in \C \, , \end{equation}
where the derivative is taken with respect to $r$ and $P_* (r, \lambda; \sigma) \ddd \e{i \lambda r} \ov{P(r, \ov{\lambda} ; \sigma)}$.
	If $H$ is continuous, then $A(r)$ is given by the $0^{\text{th}}$ ``coefficient'' of $P(r, \lambda)$, i.e.
\begin{equation}\label{A_def}
	\ov{A(r)} = \Gamma_r (r,0) \, .
\end{equation}

\section{A priori estimates: when does \texorpdfstring{$P(r, \lambda)-\e{i \lambda r} \in L^p (\R)$}{P-exp in Lp}?} \label{Decay_P}

Suppose, just within the scope of this paragraph, that $w$ is very regular, e.g.\ $w,w^{-1} \in L^{\infty} (\R)$. For \eqref{P_uniform_bd_basic} to hold for a fixed $p$, it is then necessary that $P(r, \lambda) - \e{i \lambda r} \in L^p  (\R)$ for each $r>0$ to begin with. But this is not immediate for arbitrary weights. How do properties of $w$ determine which $L^p$ spaces $P(r, \lambda) - \e{i \lambda r}$ belongs to?

\begin{Prop}\label{regularity_P}
  If $w-1 \in L^1 (\R) + L^2 (\R)$, then $P(r, \lambda ; w) - \e{i \lambda r} \in L^p (\R)$ for $2 \leq p \leq \infty$, for each $r> 0$. In particular, $P(r, \lambda ; w) - \e{i \lambda r} \in L^p _w (\R)$ for $2 \leq p < \infty$, for each $r>0$. 
\end{Prop}
Thus a priori, if $w-1 \in L^1 (\R) + L^2 (\R)$ all we can say is that $P (r, \lambda) - \e{i \lambda r} \in L^p _w (\R)$ for $2 \leq p < \infty$.

We spend the rest of this section proving Proposition \ref{regularity_P}.

\begin{Lem} \label{accelerant_P_2_infinity}
If an accelerant $H \in L^2 _{loc} (\R)$ satisfies $I + \HH_r > 0$ for each $r > 0$, where $\HH_r$ is as in \eqref{H_operator_def}, then
	\begin{enumerate}[label=(\alph*)]
		\item $P(r, \lambda) - \e{i \lambda r} \in L^p (\R)$ for all $2 \leq p \leq \infty$. \label{Lemma_initial_P_reg}
		\item for each $R>0$, we have
			\begin{equation} \label{P_loc_Linfinity}
			  \sup\limits_{0 \leq r \leq R} |P(r, \lambda) - \e{i \lambda r}| \leq R^{1/2} \|g_{r}\|_{C([0,R],L^2([0,R]))} < \infty \, ,
\end{equation}
where $g_r$ is as in \eqref{resolvent_continuity_L2}.
			Furthermore, $P(r, \lambda) - \e{i \lambda r}$ is continuous in each variable. \label{Lemma_initial_P_reg_2}
	\end{enumerate}
\end{Lem}
\begin{proof}
	Let us first focus on part \ref{Lemma_initial_P_reg}. By H\"older's inequality it suffices to show $P(r, \lambda) - \e{i \lambda r}$ is an element of $L^p (\R)$ for $p=2, \infty$. By \eqref{Gamma_symmetries} we have
	\begin{equation}\label{def_P_alt}
	  P (r, \lambda) - \e{i \lambda r} = -\e{i \lambda r} \int\limits_0 ^r \Gamma_r (s,0) \e{- i \lambda s} \, ds = -\e{i \lambda r} \int\limits_0 ^r g_r (s) \e{- i \lambda s} \, ds \, ,
	\end{equation}
	where $g_r$ is as in \eqref{resolvent_continuity_L2}.
	Since $g_r $ is an element of $L^2 ([0,r])$, then applying Plancherel to \eqref{def_P_alt} yields $P(r, \lambda) - \e{i \lambda r} \in L^2( \R)$.

	Applying Cauchy-Schwarz to the right side of \eqref{def_P_alt} gives the $L^\infty (\R)$ estimate needed to complete the proof of \ref{Lemma_initial_P_reg}; in fact Cauchy-Schwarz yields the estimate \eqref{P_loc_Linfinity}. The continuity in part \ref{Lemma_initial_P_reg_2} follows from $r \mapsto g_r$ being an element of $C([0,R], L^2 ([0,R]))$ for each $R>0$ (see Property \ref{resolvent_cty_property} of the resolvent kernel).
\end{proof}

\begin{Lem}\label{Gamma_cts} If $\langle \lambda \rangle^k (w-1) \in L^1(\R)$ for $k\geq 0$ an integer, then $w \frac{d \lambda}{2 \pi}$ has accelerant $\frac{1}{2 \pi} (\widecheck{w-1})\left ( \frac{x}{2 \pi} \right ) \in C^k (\R)$, the resolvent kernel $\Gamma$ exists, and $\Gamma_r (\cdot, \cdot) \in C^k ([0,r]^2)$ for each $r > 0$.
\end{Lem}
\begin{proof}
	By Lemma \ref{Gamma_L2}, $H(x) \ddd \frac{1}{2 \pi} (\widecheck{w-1})\left ( \frac{x}{2 \pi} \right )$ is the accelerant, and $I + \HH_r > 0$. The latter conclusion $\Gamma$ exists. Since $\langle \lambda \rangle^k (w-1) \in L^1(\R)$, then $H \in C^k (\R)$ and so $\Gamma_r (\cdot, \cdot) \in C^k ([0,r]^2)$ by Property \ref{regularity_Gamma_property} of the Resolvent kernel $\Gamma$.
\end{proof}

\begin{proof}[Proof of Proposition \ref{regularity_P}]
	Lemma \ref{Gamma_L2} implies that the measure $d \sigma = w(\lambda) \frac{d \lambda}{2 \pi}$ has accelerant $H \in L^2 _{loc} (\R)$. 
		It follows from Lemma \ref{accelerant_P_2_infinity} \ref{Lemma_initial_P_reg} that
	\[
		P(r, \lambda) - \e{i \lambda r} \in L^p (\R) \text{ for } p \in [2, \infty] \, .
		\]

	To get the weighted-norm estimate, write
	\begin{align*}
		w |P(r, \lambda) - \e{i \lambda r}|^p &= (w-1) |P(r, \lambda) - \e{i \lambda r}|^p + |P(r, \lambda) - \e{i \lambda r}|^p \\
		&= u_1 |P(r, \lambda) - \e{i \lambda r}|^p + u_2 |P(r, \lambda) - \e{i \lambda r}|^p + |P(r, \lambda) - \e{i \lambda r}|^p \, ,
	\end{align*}
	where $w-1 = u_1 + u_2$, with $u_1 \in L^1 (\R)$ and $u_2 \in L^2 (\R)$. We are left with checking the sum is integrable.

	We just showed $|P(r, \lambda) - \e{i \lambda r}|^p$ is integrable for any $p \in [2, \infty]$. Then $u_1 |P(r, \lambda) - \e{i \lambda r}|^p$ is integrable since $P(r, \lambda) - \e{i \lambda r} \in L^{\infty} (\R)$ and $u_1  \in L^1 (\R)$. To see that $u_2 |P(r, \lambda) - \e{i \lambda r}|^p$ is integrable, apply the Cauchy-Schwarz inequality, and use the fact that $u_2  \in L^2 (\R)$ and $|P(r, \lambda) - \e{i \lambda r}| \in L^{2 p} (\R)$. 
\end{proof}

\section{Krein system solutions are orthogonal to lower Fourier frequencies}\label{Krein_ortho_lemma_section}
\begindef Let $\proj{a}{b}{}$ denote the Fourier projection onto the frequency band $[a,b]$, i.e.\
\[
	\proj{a}{b}{} \left ( \int\limits_{\R} f(\xi) \e{i (\cdot) \xi} \, d \xi \right ) (\lambda) = \int\limits_{a}^b f(\xi) \e{i \lambda \xi} \, d \xi
	\]
	for all $f \in L^2 (\R)$. Note $\proj{a}{b}{} = \mathcal{F}^{-1} \chi_{[\frac{a}{2 \pi}, \frac{b}{2 \pi}]} \mathcal{F}$. \smallskip

\beginrmk	Note that $\proj{0}{r}{} g :L^1 (\R) \to L^2 (\R)$. Indeed, by Plancherel's theorem, it suffices to show 
	\[
	  \chi_{[0, \frac{b}{2 \pi}]} \mathcal{F}: L^1 (\R) \to L^2 (\R),
\]
which follows from the fact that 
\[
  \chi_{[0, \frac{b}{2 \pi}]}: L^{\infty} (\R) \to L^2 (\R) \, , \quad \mathcal{F}: L^1 (\R) \to L^{\infty} (\R) \, ,
\]
where the last boundedness property follows from the Riemann-Lebesgue lemma. \smallskip

	Solutions to the Krein system satisfy the following orthogonality Lemma. 
\begin{Lem}\label{orthogonal_inner_prod}
  If $w-1 \in L^1 (\R) + L^2 (\R)$, then for each nonnegative integer $k$, we have
	\begin{equation}\label{orthogonal_inner_prod_eqn}
	  \langle P(r, \lambda; w), \lambda^k \int\limits_0 ^{r} f(s) \e{i \lambda s} \, ds \rangle_{ w(\lambda) d \lambda}= 	\langle P(r, \lambda; w) w , \lambda^k \int\limits_0 ^{r} f(s) \e{i \lambda s} \, ds \rangle_{d \lambda}=0 \,
	\end{equation}
	for all $f \in C_c ^{\infty} ((0,r))$.

	Furthermore, we also have
\begin{align}
	\langle \e{i \lambda r}, \lambda^k \int\limits_0 ^{r} f(s) \e{i \lambda s} \, ds \rangle_{d \lambda}=0 \label{orthogonal_inner_prod_vanishing} \, ,\\
	\langle 1, \lambda^k \int\limits_0 ^{r} f(s) \e{i \lambda s} \, ds \rangle_{d \lambda} = 0 \label{orthogonal_inner_prod_vanishing_2} .
	\end{align}
\end{Lem}

\beginrmk The main statement of interest in this lemma is \eqref{orthogonal_inner_prod_eqn}. Ignoring issues of integrability, if $k=0$, then this is the intuitive statement that \[
	P(r, \lambda) \perp_w \, \mathrm{Range} \, \P_{[0,r]} \, ,
	\]
	whose analogue for OPUC was used in \cite{denik2, AlexisAD1} in addressing the Steklov problem for OPUC.

For $k \geq 1$, this gives us the more surprising statement 
\[
	P(r, \lambda) \perp_{\lambda^k w} \, \mathrm{Range} \, \P_{[0,r]} \, , 
	\]
	possessing no analogue for OPUC.\smallskip

\begin{proof}
	We focus first on \eqref{orthogonal_inner_prod_eqn}. Use integration by parts to write
	\[
		\langle P(r, \lambda) w , \lambda^k \int\limits_0 ^{r} f(s) \e{i \lambda s} \, ds \rangle_{d \lambda} = i^k \langle P(r, \lambda) w , \int\limits_0 ^{r} f^{(k)}(s) \e{i \lambda s} \, ds \rangle_{d \lambda} \, .
		\]

		Since $w-1 \in L^1 (\R)+ L^2 (\R)$, then it follows from Lemma \ref{Gamma_L2} and then Lemma \ref{accelerant_P_2_infinity} that the accelerant $H$ associated to $w$ is in $L^2 _{loc} (\R)$ and $P$ is continuous in both variables. This last property allows us to write $P(r, \lambda)$ as the limit of its averages in $r$ and so we may write the inner-product as
	\[
		i^k \langle \left ( \lim\limits_{\epsilon \to 0} \frac{1}{\epsilon} \int\limits_{r}^{r + \epsilon} P(s, \lambda) \, ds \, \right )  w , \int\limits_0 ^{r} f^{(k)}(s) \e{i \lambda s} \, ds \rangle_{d \lambda} \, .
		\]
		Since $f^{(k)}(s) \in C_c ^{\infty}(0,r)$ and $w-1 \in L^1 (\R) + L^2 (\R)$, then $w(\lambda) \int\limits_0 ^{r} f^{(k)}(s) \e{i \lambda s} \, ds \in L^1 (\R)$. Apply the dominated convergence theorem, using e.g.\ \eqref{P_loc_Linfinity} as justification (which holds thanks to Lemma \ref{Gamma_L2}), to pull $\lim\limits_{\epsilon \to 0}$ outside the inner-product, yielding
	\begin{equation}\label{inner_prod_intermediate_2}
		i^k \lim\limits_{\epsilon \to 0} \langle \frac{1}{\epsilon} \int\limits_{r}^{r + \epsilon} P(s, \lambda) \, ds \,  , \int\limits_0 ^{r} f^{(k)} (s) \e{i \lambda s} \, ds \rangle_{w (\lambda) d\lambda} \, .
	\end{equation}

	Write
	\[
	  \int\limits_0 ^{r} f^{(k)} (s) \e{i \lambda s} \, ds = \langle \e{i \lambda s} , \ov{f^{(k)} (s)} \rangle_{(ds, [0,r])} \, ,
		\]
		which, by the change of basis formula \eqref{P_basis}, equals
	\[
	\langle (I+\cal{L}_r)P(\cdot, \lambda ; \sigma) (s) , \ov{f^{(k)} (s)} \rangle_{(ds, [0,r])} =\langle P(s, \lambda ; \sigma) , (I+\cal{L}_r ^*)(\ov{f^{(k)}}) (s) \rangle_{(ds, [0,r])} = \int\limits_0 ^r P(s, \lambda) g(s) \, ds \, ,
		\]
		where $g(s) \ddd \ov{(I+\cal{L}_r ^*)(\ov{f^{(k)}}) (s)} \in L^2 ([0,r])$.
Thus we can rewrite \eqref{inner_prod_intermediate_2} as
\begin{equation*}%\label{inner_prod_intermediate_3}
	i^k \lim\limits_{\epsilon \to 0} \langle \frac{1}{\epsilon} \int\limits_{r}^{r + \epsilon} P(s, \lambda) \, ds \,  , \int\limits_0 ^r P(s, \lambda) g(s) \, ds \rangle_{w(\lambda) d \lambda } \, .
	\end{equation*}
By the orthogonality of Krein system solutions \eqref{orthonormal_system} this equals $0$, thereby completing the proof of \eqref{orthogonal_inner_prod_eqn}.

By applying \eqref{orthogonal_inner_prod_eqn} with $w=1$ and using $P(r, \lambda;1) = \e{i \lambda r}$, we get \eqref{orthogonal_inner_prod_vanishing}.

	As for \eqref{orthogonal_inner_prod_vanishing_2}, use integration by parts to write
	\[
		\langle 1, \lambda^k \int\limits_0 ^{r} f(s) \e{i \lambda s} \, ds \rangle_{d \lambda} = i^k \langle 1, \int\limits_0 ^{r} f^{(k)}(s) \e{i \lambda s} \, ds \rangle_{d \lambda} \, .
		\]
Use Fourier inversion to write the inner-product as
	\[
		\langle 1, \int\limits_0 ^{r} f^{(k)}(s) \e{i \lambda s} \, ds \rangle_{d \lambda} = \ov{\int\limits_{\R} \int\limits_0 ^{r} f^{(k)}(s) \e{i \lambda s} \, ds \, d\lambda}  = \ov{f^{(k)} (0)} = 0 \, ,
\]
where the last equality follows from $\supp f \subseteq (0,r)$.
\end{proof}

The above lemma will often be paired with the one below.
\begin{lemma}\label{check_proj_0_lemma}
	Let $r > 0$. Then given $g \in L^1 (\R) + L^2 (\R)$, we have $\proj{0}{r}{} g \in L^2 (\R)$, which equals $0$ if and only if
		\[
			\langle g (\lambda) , \int\limits_0 ^r f(s) \e{i \lambda s} \, ds \rangle_{d \lambda} = 0\,
			\]
		for all $f \in C^{\infty} _c (0,r)$.
\end{lemma}

\begin{proof}
  Let $g = g_1 + g_2$ with $g_1 \in L^1 (\R) , \, g_2 \in L^2 (\R)$. By the remark at the beginning of the section and Plancherel's theorem, we have $\proj{0}{r}{} g \in L^2 (\R)$. 

  As for the if and only if, note $\proj{0}{r}{} g = 0$ if and only if $\chi_{[0,r]} \cal{F} g \in L^{2} (\R)$ if and only if for all $f \in C_{c} ^{\infty} (0,r)$, we have
	\[
\langle \cal{F} g (s) , f(s) \rangle _{ds} = 0 \, ,
		\]
which, by taking Fourier inverses of both entries in the inner product, holds if and only if 
	\[
			\langle g (\lambda) , \int\limits_0 ^r f(s) \e{i \lambda s} \, ds \rangle_{d \lambda} = 0\, .
			\]
\end{proof}

In the sections that follow, we will consider linear operators $T$ which satisfy 
	\begin{equation} \label{sd_01}
	  \|T\|_{L^p _w (\R^d), L^p _w (\R^d)} = \|w^{1/p} T w^{-1/p}\|_{p,p} \leq \cal{F} ([w]_{A_p (\R)}, p) \, ,
	\end{equation}
	where the function $\cal{F} (t, p)$ on $[1, \infty) \times (1, \infty)$ is continuous in $t$ for every fixed $p \in (1, \infty)$. In what follows, we do not need to know $\cal{F}$ explicitly. However,  $\cal{F}$ is known in many applications. For example, the Hunt-Muckenhoupt-Wheeden theorem \cite[p.205]{stein} shows that $T$ can be taken as a singular integral operator and recent breakthrough on domination of singular integrals by sparse operators provides the sharp dependence of $\cal{F}$ on $[w]_{A_p}$. In particular, for a large class of singular integral operators, one can take
$
\cal{F}(t,p)=C(p)t^{\max(1,(p-1)^{-1})},
$
(see, e.g., \cite[p.264]{NazLer}).

\begin{lemma}\label{proj_calF_ind_r}
The Fourier projections $\proj{0}{r}{}$ satisfy \eqref{sd_01} for some $\cal{F}$ independent of $r$.
\end{lemma}
\begin{proof}
  By e.g.\ \cite[p.264]{NazLer}, \eqref{sd_01} is satisfied by the Hilbert transform $\cal{H}$, which has Fourier multiplier $-i \, \mathrm{sign} (\xi)$ \cite[p.26]{stein}. Now note each $\proj{0}{r}{}$ is a linear combination of modulated Hilbert transforms, i.e. 
\begin{equation}\label{projection_as_hilbert_transforms}
  \proj{0}{r}{} = i \left ( \frac{\cal{H} - \e{i \lambda \frac{r}{2 \pi}} \cal{H} \e{- i \lambda \frac{r}{2 \pi} }}{2} \right ) \, .
      \end{equation}
	One can check this by e.g.\ looking at the Fourier multipliers of all the operators involved. The triangle inequality then yields \eqref{sd_01} for $T= \proj{0}{r}{}$ and function $\cal{F}$ independent of $r$.
      \end{proof}

      \section{The Steklov problem for an \texorpdfstring{$A_2 (\R)$}{A2} weight: proof of Theorem \ref{A2_Thm_Krein} and Proposition \ref{minimum_p_prop}}\label{weight_A2_section}
      In this section we prove Theorem \ref{A2_Thm_Krein}, and demonstrate its sharpness by Proposition \ref{minimum_p_prop}. We need the following result, which follows from e.g.\ \cite[Theorem 1, Corollary to Theorem 1]{Vasyunin}.

\begin{Lem}[Reverse H\"older inequality, open inclusion of $A_p$ weights]\label{Reverse_Holder_A2}
Suppose $[w]_{A_p (\R)} \leq \gamma$, $p \in (1,\infty)$. Then there exists $q (\gamma, p) > 1$, satisfying $\lim\limits_{\gamma \to 1}  q (\gamma, p) = + \infty$, such that for all $t \in [0,q]$, we have
\[
		\langle w^t \rangle_I \lesssim_{\gamma} \langle w \rangle_I ^t \, 
		\]
		for all intervals $I$, and $[w^t]_{A_p (\R)} \lesssim_{\gamma,p} 1 $. 
	
	Furthermore, there exists $s (\gamma, p) \in (1, p)$ such that for all $t \in [s, \infty)$, we have
	$[w]_{A_{t} (\R)} \leq \eta (\gamma, p)$, where $s$ and $\eta$ satisfy $\lim\limits_{\gamma \to 1} s (\gamma, p) = 1$ and $\lim\limits_{\gamma \to 1} \eta (\gamma, p) = 1$.
      \end{Lem}

Assume $[w]_{A_2 (\R)} \leq \gamma$ and define $\widehat{p}_{\gamma} ' = s (\gamma,2)$, where $s$ is as in Lemma \ref{Reverse_Holder_A2}. Then $\lim\limits_{\gamma \to 1} \widehat{p}_{\gamma} = \infty$, and
\begin{equation*}
%\label{low_Muckenhoupt}
[w]_{A_{{\widehat{p}_{\gamma}} '} (\R)} \leq \eta (\gamma,2) \, , \quad \widehat{p}_{\gamma} > 2\, . 
\end{equation*}
Define \begin{equation}
\label{low_Muckenhoupt_2}
\epsilon(\gamma) \ddd \frac{1}{p_{\gamma} '} - \frac{1}{2} \, .
\end{equation}
Thus, if $p \in [\widehat{p}_{\gamma} ', \widehat{p}_{\gamma}] = \{p ~:~ |\frac{1}{p} - \frac{1}{2}| \leq \epsilon (\gamma)\}$, then 
\[
  [w]_{A_{p} (\R)} , \,  [w]_{A_{p'} (\R)} \leq \eta (\gamma,2) \, ,
\] which in particular implies that for all such $p$, we have 
\begin{equation}\label{init_bds_Ap_characteristic}
  [w]_{A_p (\R)}, [w^{-p/p'}]_{A_p (\R)} \lesssim_{\gamma} 1 \, .
\end{equation}
Note $\epsilon (\gamma) \in (0, \frac{1}{2})$, with $\lim\limits_{\gamma \to 1} \epsilon (\gamma) = \frac{1}{2}$.

If we additionally assume $w-1 \in L^1 (\R) + L^2 (\R)$, then Proposition \ref{regularity_P} implies $P(r, \lambda;w) - \e{i \lambda r} \in L^2 (\R) \cap L^{\infty} (\R)$. Together, these estimates yield
\begin{multline}\label{Steklov_wP-e_in_L1L2}
	w P(r, \lambda; w) - \e{i \lambda r} = (w-1)( P(r, \lambda; w) - \e{i \lambda r})  \\
	 + (w-1) \e{i \lambda r} + (P(r, \lambda; w) - \e{i \lambda r}) \in L^1 (\R) + L^2 (\R) \, .
\end{multline}

Proposition \ref{regularity_P} and the fact that $w-1 \in L^1 (\R) + L^2 (\R)$ also imply 
\begin{equation}\label{X_Lp_Krein_A2_v2}
	X_p \ddd  w^{1/p} \left ( P(r, \lambda;w) - \e{i \lambda r} \right ) \in L^p(\R) \, , \quad 2 \leq p < \infty \, .
\end{equation}
 Since 
\[
  \|P(r, \lambda;w) - \e{i \lambda r}\|_{L^p _w (\R)} = \|X_p\|_{L^p (\R)} \, ,
	\]
	it suffices to estimate $X_p$ in $L^p (\R)$, which we'll do using functional analysis methods. Note \eqref{X_Lp_Krein_A2_v2} means we can do functional analysis on $X_p$ in the space $L^p (\R)$ for any $p \in [2, \infty)$ when $w-1 \in L^1 (\R) + L^2 (\R)$. 

	\begin{Lem}\label{functional_lemma_A2} If $[w]_{A_2 (\R)} \leq \gamma$ and $w-1 \in L^1 (\R) + L^2 (\R)$, then there exists $\epsilon (\gamma) \in (0,\frac{1}{2})$, with $\lim\limits_{\gamma \to 1} \epsilon (\gamma) = \frac{1}{2}$, such that for all $p \in [2, \infty)$ satisfying $|\frac{1}{p} - \frac{1}{2}| < \epsilon (\gamma)$, we have
\begin{align}\label{monic_muckenhoupt_2} 
	X_p &= w^{1/p} \P_{[0,r]} w^{-1/p} X_p \, , \\
\label{orthogonal_muckenhoupt}
		0&=w^{-1/p'} \P_{[0,r]} w^{1/p'} X_p + w^{-1/p'} \P_{[0,r]} w^{1/p'} (w^{1/p} - w^{-1/p'}) \e{i \lambda r} \,
\end{align}
in $L^p (\R)$. In particular, for all such values of $p$ we have
	\begin{equation}\label{functional_eqn_A2}
		(I-Q_{w,p}) X_p = - w^{-1/p'} \P_{[0,r]} w^{1/p'} (w^{1/p} - w^{-1/p'}) \e{i \lambda r} \, ,
	\end{equation}
where
	\begin{equation*}%\label{Q}
		Q_{w, p} \ddd w^{1/p} \P_{[0,r]} w^{-1/p} - w^{-1/p'} \P_{[0,r]} w^{1/p'} \, .
	\end{equation*}

\end{Lem}

\beginrmk One might wonder if $(w^{1/p} - w^{-1/p'}) \, \e{i \lambda r} \in L^p (\R)$; it is by taking $q=0$, and both $\widetilde{p}$ and $p$ equal in \eqref{integrability_Ap_weights_most_general_statement} of the following lemma. \smallskip
\begin{Lem}[Integrability of $A_p$ weights] \label{lemma_mu_weight_Lp}
	Let $d\mu (\lambda) =\langle \lambda  \rangle^{q} d \lambda$ for some $q \geq 0$. Suppose that
	\begin{itemize}
		\item $[w]_{A_{p'} (\R)} \leq \gamma$ for some $p' \in (1, \infty)$.
		\item $w-1 = u_1 + u_2$ with $u_1 \in L^{p_1} _{\mu} (\R), u_2 \in L^{p_2} _{\mu} (\R)$.
		\item $1 \leq p_1  \leq p_2 \leq p$.
	      \end{itemize} Then there exists $\epsilon (\gamma, p) > 0$, with $\lim\limits_{\gamma \to 1} \epsilon (\gamma,p) = \min\{1/p,1/p'\}$, such that for all $\widetilde{p}$ satisfying $|\frac{1}{p} - \frac{1}{\widetilde{p}}| < \epsilon (\gamma, p)$, we have
			\begin{equation}\label{integrability_Ap_weights_most_general_statement}
			  \|w^{1/\widetilde{p}} - w^{-1/\widetilde{p}'}\|_{L^p _{\mu} (\R)} < \infty \, . 
			\end{equation}
				
				If $q=0$ and $\widetilde{p} = p$, then there exists $\tau_0 (p)$ a small constant such that whenever 
				\[
				  \tau \ddd [w]_{A_{\infty} (\R)} - 1 \leq \tau_0(p) \, ,
				\] we have the perturbative estimate
		\begin{equation}\label{perturbative_est_weight}
			\|w^{1/p} - w^{-1/p'}\|_{L^p (\R)} \lesssim_{p} \tau ^{(p-p_2)/(2p)} (\tau^{\frac{p_2 - p_1}{2}} \|u_1\|_{L^{p_1} (\R) } ^{p_1}+ \|u_2\|_{L^{p_2} (\R)} ^{p_2})^{1/p}\, .
		\end{equation}
\end{Lem}

We defer the proof of Lemma \ref{lemma_mu_weight_Lp} till Section \ref{tech_lemma_pf_section}.

\begin{proof}[Proof of Lemma {\ref{functional_lemma_A2}}]
  Take $\epsilon (\gamma)$ as in \eqref{low_Muckenhoupt_2}. Then note all relevant quantities are well-defined as elements of, or operators on, $L^p (\R)$: $X_p \in L^p (\R)$ by \eqref{X_Lp_Krein_A2_v2}, $w^{-1/p'} \P_{[0,r]} w^{1/p'}$ and $w^{-1/p'} \P_{[0,r]} w^{1/p'}$ are bounded on $L^p (\R)$ by the Hunt-Muckenhoupt-Wheeden theorem and the fact that $w, w^{-p/p'} \in A_p (\R)$. And as per the previous remark, $(w^{1/p} - w^{-1/p'}) \e{i \lambda r} \in L^p (\R)$ by Lemma \ref{lemma_mu_weight_Lp}.

	We note \eqref{monic_muckenhoupt_2} is equivalent to 	\[
		P(r, \lambda) - \e{i \lambda r} = \proj{0}{r}{} (P(r, \lambda) - \e{i \lambda r}) \, .
	      \] Since $P(r, \lambda) - \e{i \lambda r} = - \int\limits_0 ^r \Gamma_r (r, t) \, \e{i \lambda t } \, dt$ with $\Gamma_r (r, \cdot) \in L^2 ([0,r])$ as in the discussion in Section \ref{Krein_system_Basics}, then this clearly holds.

	Meanwhile \eqref{orthogonal_muckenhoupt} is equivalent to
	\begin{equation*} %\label{intermediate_1}
			\P_{[0,r]} (w P(r, \lambda) - \e{i \lambda r}) =0 \, .
		\end{equation*}

	From \eqref{Steklov_wP-e_in_L1L2} it follows that $w P(r, \lambda) - \e{i \lambda r} \in L^1 (\R) + L^2 (\R)$. Thus by Lemma \ref{check_proj_0_lemma}, it suffices to show
	\[
		\langle w P(r, \lambda) - \e{i \lambda r}, \int\limits_0 ^r f(s) \e{i \lambda s} \, ds \, \rangle_{d \lambda } = 0
		\] for each $f \in C^{\infty} _c (0,r)$. This follows from Lemma \ref{orthogonal_inner_prod}.

To get \eqref{functional_eqn_A2}, subtract \eqref{orthogonal_muckenhoupt} from \eqref{monic_muckenhoupt_2} and rearrange.
\end{proof}

As it turns out, we can invert $I-Q_{w,p}$ for $p$ sufficiently close to $2$.
\begin{Lem} \label{thm_inverse_bd_final_basic} For $w \geq 0$, consider the formal operator
\begin{equation*}%\label{Q_general_basic}
  Q_{w, p} \ddd  w^{1/p} \proj{0}{r}{} w^{-1/p} - w^{-1/p'} \proj{0}{r}{} w^{1/p'} \, ,
\end{equation*}
If $[w]_{A_2 (\R)} \leq \gamma$, then there exists $\epsilon (\gamma) \in (0,\frac{1}{2})$, independent of $r$, with $\lim\limits_{\gamma \to 1} \epsilon (\gamma) = \frac{1}{2}$, such that for all $p$ satisfying $|\frac{1}{p} - \frac{1}{2}| < \epsilon (\gamma)$, $I-Q_{w,p}$ has bounded inverse on $L^{p} (\R)$ with operator bound
	\begin{equation}\label{inverse_bd_final_basic}
	  \|(I-Q_{w, p})^{-1}\|_{p,p} \lesssim 1 \, .
	\end{equation}
\end{Lem}

Let us briefly discuss the strategy for proving Theorem \ref{A2_Thm_Krein}. Using Lemma \ref{thm_inverse_bd_final_basic} and \eqref{functional_eqn_A2}, we have
\begin{equation}\label{X_solved_pgeq2}
  X_p = - (I- Q_{w,p})^{-1} w^{-1/p'} \P_{[0,r]} w^{1/p'} (w^{1/p} - w^{-1/p'}) \e{i \lambda r} \, \quad \text{ in } L^p (\R) \,  
\end{equation}
for all $p\geq 2$ sufficiently close to $2$. Then, we can estimate $\|X_p\|_{L^p (\R)}$ by estimating $\|(I-Q_{w,p})^{-1}\|_{p,p}$, $\|w^{-1/p'} \P_{[0,r]} w^{1/p'}\|_{p,p}$, and $\|(w^{1/p} - w^{-1/p'}) \e{i \lambda r}\|_{p}$. However, such a process will only give us a bound for $\|X_p\|_p$ when $p\geq 2$ since our starting point, \eqref{functional_eqn_A2}, was only valid for $p \geq 2$. We address this is by noting that all elements involving $p$ in \eqref{X_solved_pgeq2} are actually analytic in variable $\frac{1}{p}$, and the right-side of \eqref{X_solved_pgeq2} is well-defined for $p<2$. Hence equality must hold for $p < 2$. In particular, we have the following Proposition. 

\begin{Prop}\label{solve_for_X_Steklov_fullp_statement}
  Suppose $w-1 \in L^{1} (\R) + L^{2} (\R)$. If $[w]_{A_2 (\R)} \leq \gamma$, then there exists $\epsilon (\gamma) \in (0, \frac{1}{2})$, with $\lim\limits_{\gamma \to 1} \epsilon (\gamma) = \frac{1}{2}$, such that
	\begin{equation}\label{solve_for_X_Steklov_fullp}
		X_p = - (I-Q_{w,p})^{-1} w^{-1/p'} \P_{[0,r]} w^{1/p'} (w^{1/p} - w^{-1/p'}) \e{i \lambda r} \, ,
	\end{equation}
	for all $p$ satisfying $\left |\frac{1}{p} - \frac{1}{2} \right| < \epsilon (\gamma)$.
\end{Prop}

We prove Lemma \ref{thm_inverse_bd_final_basic} and Proposition \ref{solve_for_X_Steklov_fullp_statement} in Section \ref{section_I-Qwp_invertible} and Appendix \ref{complex_analysis_section} respectively.

	\begin{proof}[Proof of Theorem \ref{A2_Thm_Krein}]
By Proposition \ref{solve_for_X_Steklov_fullp_statement}, we may estimate
\[
	\|X_p\|_{L^p (\R)} \leq \|(I - Q_{w, p})^{-1}\|_{p,p} \|w^{-1/p'} \P_{[0,r]} w^{1/p'} \|_{p,p} \| w^{1/p} - w^{-1/p'}\|_{L^p (\R)} \, 
		  \]
		  for all $p$ satisfying $|\frac{1}{p} - \frac{1}{2}| < \epsilon (\gamma)$, where $\lim\limits_{\gamma \to 1} \epsilon (\gamma) = \frac{1}{2}$.

Since $\proj{0}{r}{}$ satisfies \eqref{sd_01} for some $\cal{F}$ independent of $r$, we may in fact write
	\[
		\|X_p\|_{L^p (\R)} \leq \|(I - Q_{w, p})^{-1}\|_{p,p} \cal{F} ([w^{-p/p'}]_{A_p (\R)}, p) \| w^{1/p} - w^{-1/p'}\|_{L^p (\R)} \, .
		  \]
		
		  Since $[w^{-p/p'}]_{A_p (\R)} \lesssim_{\gamma} 1$ by \eqref{init_bds_Ap_characteristic}, then $\cal{F} ([w^{-p/p'}]_{A_p (\R)}, p) \lesssim_{\gamma,p} 1 $. And by applying Lemma \ref{thm_inverse_bd_final_basic}, we get $\|(I-Q_{w,p})^{-1}\|_{p,p} \lesssim 1$ for $\left |\frac{1}{p} - \frac{1}{2} \right | < \epsilon (\gamma)$. Thus
		\[
			\|X_p\|_{L^p (\R)} \lesssim_{p, \gamma} \| w^{1/p} - w^{-1/p'}\|_{L^p (\R)}
			\]
		for all $p$ satisfying $|\frac{1}{p} -\frac{1}{2}| < \epsilon(\gamma)$. Note all of these estimates are uniform in $r$.

Part \ref{general_Krein_A2} now follows from noting $\| w^{1/p} - w^{-1/p'}\|_{L^p (\R)} < \infty$ by Lemma \ref{lemma_mu_weight_Lp} so long as $p \in [p_2, \infty)$.

		Meanwhile part \ref{perturbative_Krein_A2} follows by fixing $p$, taking $\tau_0 (p)$ small enough so that
		\begin{itemize}
			\item $|\frac{1}{p} - \frac{1}{2}| < \epsilon (\gamma)$.
			\item $[w]_{A_{\infty} (\R)} -1 \leq [w]_{A_2 (\R)} -1 = \tau$ is sufficiently small that \eqref{perturbative_est_weight} applies.
		\end{itemize}
		Then \eqref{perturbative_est_weight} implies \eqref{perturbative_estimate_A2}.
\end{proof}

We now turn our attention towards Proposition \ref{minimum_p_prop}. Let us first discuss its meaning: suppose $w-1 \in L^1 (\R) + L^{p_2} (\R)$ for $p_2 \in [1,2]$ and $w-1 \notin L^1 (\R)$; we think of $p_2$ as measuring the decay of $w-1$, with smaller $p_2$'s indicating better decay. Recall Theorem \ref{A2_Thm_Krein} \ref{general_Krein_A2}, which says that if $[w]_{A_2 (\R)} \leq \gamma$, then
\[
  \|P(r, \lambda; w) - \e{i \lambda r}\|_{L^p _w (\R)} < \infty  
\]
whenever $p \in [p_2, \infty)$ satisfies $|\frac{1}{p} - \frac{1}{2}| < \epsilon (\gamma)$, where $\lim\limits_{\gamma \to 1} \epsilon (\gamma) = \frac{1}{2}$.
So let $\gamma- 1$ be sufficiently small that $|\frac{1}{p_2} - \frac{1}{2}| < \epsilon (\gamma)$. Then the requirement that $p \geq p_2$ in Theorem \ref{A2_Thm_Krein} \ref{general_Krein_A2} is in part saying that $P(r, \lambda ;w ) - \e{i \lambda r}$ cannot hope to decay faster than $w-1$ decays. This seemingly makes sense a priori: since the accelerant satisfies $H(x) = \frac{1}{2 \pi}\widecheck{(w-1)} \left ( \frac{x}{2 \pi} \right )$, then whatever decay $w-1$ possesses gets ``converted'' into the regularity of $H$. But then the resolvent kernel $\Gamma_r (s,t)$, as a rule of thumb, is at most as regular as $H$. Since by \eqref{def_P} we have 
\[
  P(r, \lambda) - \e{i \lambda r} = - \int\limits_0 ^r \Gamma_r (r,t) \e{i \lambda t} \, dt \, ,
\]
i.e.\ $P(r, \lambda) - \e{i \lambda r}$ is essentially the inverse Fourier transform of $-\Gamma_r (r,\cdot)$ and so whatever regularity $\Gamma_r (r, \cdot)$ possesses gets ``converted'' into the decay of $P(r, \lambda) - \e{i \lambda r}$. Thus heuristically, we expect the decay of $P(r, \lambda) - \e{i \lambda r}$ to not exceed that of $w-1$, as the former ``inherits'' its decay from the latter.

\begin{proof}[Proof of Proposition \ref{minimum_p_prop}]
  By H\"older's inequality, we can assume without loss of generality that $p \in (1, p_2)$.

  Without loss of generality, assume $\delta \leq \frac{1}{2}$. Let $u$ be an even, real-valued function on $\R$ such that $0 \leq u \leq 1$ and $u \in L^q (\R)$ if and only if $q \geq p_2$. Then define $w = 1+\delta u$ so that $[w]_{A_2(\R)} \leq 1 + \delta$ and $w-1 \in L^{p_2} (\R)$. 

  Without loss of generality, assume $\delta$ sufficiently small that $|\frac{1}{p} - \frac{1}{2}| < \epsilon (1+\delta)$, where $\epsilon$ is as in Lemma \ref{thm_inverse_bd_final_basic}. Then by Lemma \ref{thm_inverse_bd_final_basic}, $\|(I-Q_{w,p})^{-1}\|_{p,p} \lesssim 1$. Thus by Proposition \ref{solve_for_X_Steklov_fullp_statement}, we have 
\[
  X_p = -(I-Q_{w,p})^{-1} w^{-1/p'} \proj{0}{r}{} w^{1/p'} (w^{1/p} - w^{-1/p'}) \e{i \lambda r} \, , 
\]
or rather
\[
  (I-Q_{w,p}) X_p = - \delta \, \e{i \lambda r} w^{-1/p'} \proj{-r}{0}{} u \, . 
\]
Estimating $L^p (\R)$ norms yields
\[
  \|I-Q_{w,p}\|_{p,p} \|X_p\|_{L^p (\R)} \gtrsim_p \|\proj{-r}{0}{} u\|_{L^p (\R)} \, .
\]
Since $w \sim 1$, then $[w]_{A_p (\R)}, [w^{-p/p'}]_{A_p (\R)} \sim_{p} 1$ and so by Lemma \ref{proj_calF_ind_r}, we have $\|I - Q_{w,p}\|_{p,p} \lesssim 1$ and so
\[
  \|X_p\|_{L^p (\R)} \gtrsim_p \|\proj{-r}{0}{} u\|_{L^p (\R)} \, .
\]
This then means
\[
  \sup\limits_{r \geq 0} \|X_p\|_{L^p (\R)} \gtrsim_p \sup\limits_{r \geq 0} \|\proj{-r}{0}{}u\|_{L^p (\R)} \, .
\]
It suffices to show
\[
  \sup\limits_{r \geq 0} \|\proj{-r}{0}{}u\|_{L^p (\R)} = \infty \, .
\]
Since $u \in L^2 (\R)$, write $u(\lambda) = \int\limits_{-\infty} ^{\infty} v(s) \e{i \lambda s} \, ds$, where $v \in L^2 (\R)$. Since $u$ is real-valued and even, then so is $v$. In particular this means $\proj{-r}{0}{}u = \ov{\proj{0}{r}{} u}$, and combined with the fact that $u$ is real-valued, we get $\sup\limits_{r \geq 0} \|\proj{-r}{0}{}u\|_{L^p (\R)} = \infty$ if $\sup\limits_{r \geq 0} \|\proj{-r}{r}{}u\|_{L^p (\R)} = \infty$.
 
In fact, $\{\proj{-n}{n}{} u\}_{n \geq 0}$ is unbounded in $L^p (\R)$. Indeed, suppose to the contrary the sequence is bounded. Then there exists some increasing sequence of integers $n_k$ and some $\widetilde{u} \in L^p (\R)$ such that $\proj{-n_k}{n_k}{} u$ converges to $\widetilde{u}$ weakly in $L^p (\R)$. In particular, for every Schwarz function $f$, we have
\[
  \langle \widetilde{u}, f \rangle = \lim\limits_{k \to \infty} \langle \proj{-n_k}{n_k}{} u, f \rangle = \langle u , f \rangle \, , 
\]
where the last equality follows from Plancherel's theorem. Thus $u = \widetilde{u} \in L^p (\R)$, which contradicts our requirement that $u \in L^q (\R)$ if and only if $q \geq p_2$. This completes the proof.
\end{proof}

\section{A mixed norm remainder estimate}\label{remainder_A2_section}
In this section, we prove Theorem \ref{Remainder_A2_Thm}.

\begindef If $\Gamma_r (\cdot, \cdot) \in C^k ([0,r]^2)$, integrate \eqref{def_P} by parts $k$ times to yield
\begin{equation*}%\label{Laurent}
	P(r, \lambda) - \e{i \lambda r} = \sum\limits_{l=1}^{k} \frac{a_{l, r} (\lambda)}{\lambda^l} + \frac{R_{k, r} (\lambda)}{\lambda^k} \, ,
\end{equation*}
where
\begin{equation}\label{def_a}
	a_{l, r} (\lambda) \ddd  \begin{cases} (i)^l \left ( \e{i\lambda r} (\partial_t)^{l-1} \Gamma_r (r,t) \Big|_{t=r} - (\partial_t)^{l-1} \Gamma_r (r,t)\Big|_{t=0} \right ) \quad &l \geq 1\\ 0 &l =0\end{cases}
\end{equation}
and the \textit{remainder term} $R_{k, r} (\lambda)$ is defined by
	\begin{equation}\label{Remainder_integral}
		R_{k,r} (\lambda) \ddd -(i)^k \int\limits_{0}^r \left ( (\partial_t)^k \Gamma_r (r,t) \right ) \e{i \lambda t} \, dt \, .
	\end{equation}\smallskip

We can also express the remainder in terms of the solution $P(r, \lambda)$ to the Krein system and the ``coefficients'' $a_{l, r} (\lambda)$, i.e.\
\begin{equation*}%\label{Remainder_poly}
	R_{k,r} (\lambda) = \lambda ^k (P(r, \lambda) - \e{i \lambda r}) - \sum\limits_{l=1}^{k} \lambda^{k-l} a_{l, r} (\lambda) \, .
\end{equation*}
Note that $R_{0,r} = P(r, \lambda ; w) - \e{i \lambda r}$.

The Steklov problem can be reformulated in terms of mixed norms: it asks when does one have the bound
\begin{equation}\label{P_Steklov_mixed_norm}
  \|P(r, \lambda)\|_{L^{\infty} _{d r} (\R^+; L^p _{w} (\Delta))} < \infty
\end{equation}
	for every compact $\Delta \subset \R$? In Theorem \ref{A2_Thm_Krein}, we estimated
\[
  \|R_{0,r}\|_{L^{\infty} _{dr} (\R^+; L^p _{w} (\R))}
	\]
	for some $p$ close to $2$, which implied \eqref{P_Steklov_mixed_norm}. However, it also makes sense to consider other mixed norms, as we do in Theorem \ref{Remainder_A2_Thm}. To estimate $R_{1, r}$ in $L^p _w (\R)$, it suffices to estimate $Y_p \ddd w^{1/p} R_{1, r} (\lambda)$ in $L^p (\R)$.

We begin with two lemmas.

\begin{Lem}\label{A_L2}
  Suppose $w\in A_2 (\R)$ and $w-1 \in L^1 (\R)$. If Condition \eqref{Szego_fn_Hardy} holds, then $A \in L^2 (\R^+)$, where $A(r)$ is as given by \eqref{A_def}.
\end{Lem}
\begin{proof}
	By \cite[Theorem 12.14]{denisov_krein}, $A(r) \in L^2 (\R^+)$ follows if both Condition \eqref{Szego_fn_Hardy} holds and $\log w \in L^1 (\R)$. Thus it suffices to check our assumptions imply $\log w \in L^1 (\R)$.

	Note that when $x \geq \frac{1}{2}$, then $|\log (x)| \lesssim |x-1|$. Whence
	\[
		\int\limits_{\{w \geq \frac{1}{2}\}} |\log w | \lesssim \int\limits_{\{w \geq \frac{1}{2}\}} |w-1| \leq \|w-1\|_{L^1 (\R)} < \infty \, .
		\]
		Since $\int\limits_{\R} |w-1| d \lambda < \infty$, then $\{ w \leq \frac{1}{2}\}$ has finite Lebesgue measure. By taking $\widetilde{p},p'=2$ and $q=0$ in Lemma \ref{lemma_mu_weight_Lp}, we get $w^{1/2} - w^{-1/2} \in L^2 (\R)$ and in particular $w^{1/2} -w^{-1/2}$ is square-integrable on $\{w \leq \frac{1}{2}\}$. Meanwhile $w$ is integrable on $\{w \leq \frac{1}{2}\}$ since $w-1 \in L^1 (\R)$. Thus the $L^2 (\R)$-triangle inequality implies $w^{-1}$ is integrable on $\{w \leq \frac{1}{2}\}$: indeed, $\int\limits_{\{w \leq \frac{1}{2}\} } w^{-1}$ equals   
		\[
	\int\limits_{\{w \leq \frac{1}{2}\} } |w^{-\frac{1}{2}}|^2 \lesssim \int\limits_{\{w \leq \frac{1}{2}\} } |w^{\frac{1}{2}} - w^{-\frac{1}{2}}|^2 + \int\limits_{\{w \leq \frac{1}{2}\} } |w^{\frac{1}{2}}|^2 = \int\limits_{\{w \leq \frac{1}{2}\} } |w^{\frac{1}{2}} - w^{-\frac{1}{2}}|^2 + \int\limits_{\{w \leq \frac{1}{2}\} } w-1 + \int\limits_{\{w \leq \frac{1}{2}\} } 1 < \infty \, .		\] And hence $\log w$ must be integrable on $\{ w \leq \frac{1}{2} \}$ as well.
\end{proof}

\begin{Lem}\label{a_L2} Suppose $w \in A_2 (\R)$, $\langle \lambda \rangle (w-1) \in L^1 (\R)$ and Condition \eqref{Szego_fn_Hardy} holds. Then
\begin{equation}\label{decomposition_a1r}
	a_{1, r} (\lambda) = \e{i \lambda r} \alpha_{\infty} (r) + \alpha_{2}(r) \, ,
	\end{equation}
	where $\alpha_2 \in L^2 (\R^+)$, $\alpha_{\infty} \in L^{\infty} (\R^+)$ and $\lim\limits_{r \to \infty} \alpha_{\infty} (r)$ exists. In particular, 
	\[
	  |a_{1,r} (\lambda)| \leq |\alpha_{\infty} (r)| + |\alpha_2 (r)| \in L^{\infty} _{dr} (\R^+) + L^2 _{dr} (\R^+)
		\]
		for all $\lambda \in \R$.
\end{Lem}
\begin{proof}
	Since $\langle \lambda \rangle (w-1) \in L^1 (\R)$, then by Lemma \ref{Gamma_cts} it follows that $\Gamma_r (\cdot, \cdot) \in C^1([0,r]^2)$. Hence $a_{1, r}$ is well-defined and is given by
 \[
	a_{1, r} (\lambda) = i [\e{i \lambda r} \Gamma_r(r,r) - \Gamma_r (r,0) ] = i [\e{i \lambda r} \Gamma_{r} (0,0) - \ov{A(r)}] \, ,
	\]
	where we used \eqref{Gamma_symmetries} and \eqref{A_def} in the last equality.

Then use \eqref{Gamma_diff} and the fundamental theorem of Calculus to write
	\[
		\Gamma_r (0,0) = \Gamma_{0} (0,0) + \int\limits_0 ^r \partial_x \Gamma_x (0,0) \, dx =\Gamma_{0} (0,0)  -\int\limits_0 ^r \Gamma_x (0,x) \Gamma_x (x,0) \, dx \, .
		\] Use \eqref{A_def} and the resolvent symmetries \eqref{Gamma_symmetries} again, to get
 \[
	 a_{1, r} (\lambda) = -i \left ( -\e{i \lambda r} \Gamma_0 (0,0) + \e{i\lambda r} \int\limits_0 ^r |A(s)|^2 \, ds + \ov{A(r)} \right ) \, .
	 \]

Define
\[
\alpha_{\infty} (r) \ddd -i \left ( -\Gamma_0 (0,0) + \int\limits_0 ^r |A(s)|^2 \, ds \right) \, , \quad \alpha_2 (r) \ddd -i \ov{A(r)} \, ,
\]
 so that \eqref{decomposition_a1r} holds. Since $A(r) \in L^2 (\R^+)$ by Lemma \ref{A_L2}, then $\alpha_2 (r) \in L^2 (\R^+)$ and
 \[
 |\alpha_{\infty} (r)| \leq \left | \Gamma_0 (0,0) \right| + \int\limits_0 ^{\infty}|A(s)|^2 \, ds  < \infty \, ,  \quad \lim\limits_{r \to \infty} \alpha_{\infty} (r) = -i \left ( -\Gamma_0 (0,0) + \int\limits_0 ^{\infty}|A(s)|^2 \, ds \right) \, .
 \]
\end{proof}

Now we can do functional analysis like in the proof of Theorem \ref{A2_Thm_Krein}.

\begin{Lem}Suppose  $\langle \lambda \rangle ^q (w-1) \in L^1 (\R)$ for some $q > 2$. If $[w]_{A_2 (\R)} \leq \gamma$, then there exists $\epsilon (\gamma) \in (0, \frac{1}{2})$, with $\lim\limits_{\gamma \to 1} \epsilon (\gamma) = \frac{1}{2}$, such that for all $ p \in (1, q]$ satisfying $|\frac{1}{p} - \frac{1}{2}| < \epsilon (\gamma)$, we have
	\begin{align}
		\label{exponential_system_remainder_muckenhoupt} Y_p &=w^{1/p} \P_{[0,r]} w^{-1/p} Y_p \, ,\\
		\label{orthogonality_remainder_muckenhoupt}	w^{-1/p'} \P_{[0,r]} w^{1/p'} Y_p &=  -w^{-1/p'} \P_{[0,r]}w^{1/p'} ( w^{1/p} - w^{-1/p'}) \left ( \lambda \e{i \lambda r} + a_{1, r} (\lambda) \right )
	      \end{align} in $L^p (\R)$. In particular, for all such values of $p$, we have
	      \begin{equation}\label{eqn:remainder_functional_eqn_1}
		(I-Q_{w, p}) Y_p = -w^{-1/p'} \P_{[0,r]}w^{1/p'} ( w^{1/p} - w^{-1/p'}) \left ( \lambda \e{i \lambda r} + a_{1, r} (\lambda) \right ) \, ,
	      \end{equation}
		where $Q_{w,p} \ddd w^{1/p} \proj{0}{r}{} w^{-1/p} - w^{-1/p'} \proj{0}{r}{} w^{1/p'}$.
\end{Lem}

\beginrmk Since $\langle \lambda \rangle ^2 (w-1) \in L^1 (\R)$, then $\Gamma_r (\cdot, \cdot) \in C^2([0,r]^2)$ by Lemma \ref{Gamma_cts}. Integrating \eqref{Remainder_integral} by parts once yields $R_{1, r} (\lambda) \in L^p (\R)$ for $1 < p \leq \infty$, and so $Y_p \in L^p (\R)$ for $1 < p \leq \infty$.

We also remark that the right-side of \eqref{orthogonality_remainder_muckenhoupt} is well-defined, i.e.\ 
\[
	( w^{1/p} - w^{-1/p'}) (\lambda \e{i \lambda r} +a_{1, r} (\lambda)) \in L^p (\R)
	\] for $1<p\leq q$, for each $r > 0$.  
	Indeed this follows from Lemma \ref{lemma_mu_weight_Lp}, \eqref{decomposition_a1r}, and that $\langle \lambda \rangle ^q (w-1) \in L^1 (\R)$. We require $q > 2$ just so that we may consider $p>$, which we will do later. \smallskip
\begin{proof}
  Take $\epsilon(\gamma)$ as in \eqref{low_Muckenhoupt_2}, which in particular implies \eqref{init_bds_Ap_characteristic} holds for all $p$ of concern, i.e. for all $p$ such that $\left |\frac{1}{p} - \frac{1}{2} \right | < \epsilon$. Then note all relevant quantities are well-defined as elements of, or operators on, $L^p (\R)$: $w^{-1/p'} \P_{[0,r]} w^{1/p'}$ and $w^{-1/p'} \P_{[0,r]} w^{1/p'}$ are bounded operators on $L^p (\R)$ by the Hunt-Muckenhoupt-Wheeden theorem and that $w, w^{-p/p'} \in A_p (\R)$. And as per the remark above, $Y_p \in L^p (\R)$, and $( w^{1/p} - w^{-1/p'}) (\lambda \e{i \lambda r} +a_{1, r} (\lambda)) \in L^p (\R)$ by Lemma \ref{lemma_mu_weight_Lp}.

  Notice \eqref{Remainder_integral} implies \eqref{exponential_system_remainder_muckenhoupt} is equivalent to
	\[
		R_{1, r} = \proj{0}{r}{} R_{1, r} \, .
		\] Since $R_{1, r} = -i \int\limits_0 ^r \partial_t \Gamma_r (r,t) \, \e{i \lambda t} \, dt$ with $\partial_t \Gamma_r (r, \cdot) \in C([0,r])$, it clearly holds.

	Concerning \eqref{orthogonality_remainder_muckenhoupt}, it is equivalent to
	\[
		\P_{[0,r]} [w(R_{1, r} (\lambda) +\lambda \e{i \lambda r} +  a_{1, r}(\lambda)) -(\lambda \e{i \lambda r} +  a_{1, r} (\lambda))] = 0 \, .
		\]

	We note that $\proj{0}{r}{}$ is acting on \[
		w(R_{1, r} (\lambda) +\lambda \e{i \lambda r} +  a_{1, r}(\lambda)) -(\lambda \e{i \lambda r} +  a_{1, r} (\lambda)) \in L^1 (\R) + L^2 (\R) \, ;
		\]
		to see this, rewrite it as
		\[
			(w-1)R_{1, r} + R_{1, r} + (w-1) (\lambda \e{i \lambda r} +  a_{1, r}) \, .
			\]
			By the previous remark, $R_{1,r} \in L^2 (\R)$ and also $R_{1,r} \in L^{\infty} (\R)$. If we combine these two estimates with the assumption that $\langle \lambda \rangle (w-1) \in L^1 (\R)$, then we get $(w-1)R_{1, r} \in L^1 (\R)$. Finally, combine assumption $\langle \lambda \rangle (w-1) \in L^1 (\R)$ with the estimate $a_{1,r} \in L^{\infty} (\R)$, which follows from \eqref{def_a}, to get $(w-1) (\lambda \e{i \lambda r} +  a_{1, r}) \in L^1 (\R)$.

			Thus by Lemma \ref{check_proj_0_lemma}, it suffices to show
		\[
			\langle w(R_{1, r} + a_{1,r} + \lambda \e{i \lambda r}) - (\lambda \e{i \lambda r} + a_{1,r}) , \int\limits_0 ^r f(s) \e{i \lambda s} \, ds \rangle_{d \lambda} = 0 \, ,
			\]
			or equivalently
\[
	\langle w \lambda P(r, \lambda)- (\lambda \e{i \lambda r} + a_{1,r}) , \int\limits_0 ^r f(s) \e{i \lambda s} \, ds \rangle_{d \lambda} = 0
			\]
		for each $f \in C_{c}^{\infty} (0,r)$. This now follows from Lemma \ref{orthogonal_inner_prod} and the identity
		\[
		a_{1, r} (\lambda) = i (\e{i \lambda r} \Gamma_r (r,r) - \Gamma_r (r,0)) \, ,
		\] as given by \eqref{def_a}.

		Add \eqref{orthogonality_remainder_muckenhoupt} and \eqref{exponential_system_remainder_muckenhoupt} and rearrange to obtain \eqref{eqn:remainder_functional_eqn_1}.
	\end{proof}

\begin{proof}[Proof of Theorem \ref{Remainder_A2_Thm}]
  It suffices to estimate $\|Y_p\|_{L^{\infty} _{dr} (\R; L^p (\R)) + L^2 _{dr} (\R; L^p (\R))}$; we first estimate $\|Y_p\|_{L^p (\R)}$ in terms of $r$.

	From \eqref{eqn:remainder_functional_eqn_1} we have
	\begin{align*}
	\|Y_p\|_{L^p (\R)} \leq \|(I - Q_{w,p})^{-1}\|_{p,p} \, \cdot \, \| w^{-1/p'} \P_{[0,r]}w^{1/p'}\|_{p,p} \, \cdot \,\|( w^{1/p} - w^{-1/p'}) ( \lambda \e{i \lambda r} + a_{1, r} (\lambda) )\|_{L^p (\R)}  \, .
		\end{align*}
		By Lemma \ref{thm_inverse_bd_final_basic}, we have $\|(I-Q_{w, p})^{-1}\|_{p,p} \lesssim 1$ whenever $|\frac{1}{p} - \frac{1}{2}| < \epsilon(\gamma)$, where $\lim\limits_{\gamma \to 1} \epsilon (\gamma) = \frac{1}{2}$. Since $\proj{0}{r}{}$ satisfies \eqref{sd_01} for some $\cal{F}$ independent of $r$, we can also estimate 
		\[
		  \| w^{-1/p'} \P_{[0,r]}w^{1/p'}\|_{p,p} \leq \cal{F}([w^{-p/p'}]_{A_p (\R)}, p) \lesssim_{\gamma, p} 1 \, .
		\]
		If additionally $p \in (1,q]$, then 
		\[
			\|( w^{1/p} - w^{-1/p'})( \lambda \e{i \lambda r} + a_{1, r} (\lambda))\|_{L^p _{d \lambda} (\R)}\lesssim \|w^{1/p} - w^{-1/p'}\|_{L^p _{\langle \lambda \rangle^q d \lambda} (\R)} (1 +  |\alpha_2 (r)| + |\alpha_{\infty} (r)|)
			\]
			where $\alpha_2 , \alpha_{\infty}$ arise from Lemma \ref{a_L2}. By Lemma \ref{lemma_mu_weight_Lp} with $\widetilde{p} = p$, we have 
			\[
			  \|w^{1/p} - w^{-1/p'}\|_{L^p _{\langle \lambda \rangle^q d \lambda} (\R)} \lesssim_{p,w,q} 1 \, .
			\] All together, we get
		\[
		\|Y_p\|_{L^p _w (\R)} \lesssim_{p,w, q} (1 +  \alpha_2 (r) + \alpha_{\infty} (r)) \, .			\]
		Since $1 +  \alpha_2 (r) + \alpha_{\infty} (r) \in L^{\infty} _{dr} (\R^+)+ L^2 _{dr} (\R^+)$ by Lemma \ref{a_L2}, then 
		\[
		\|Y_p\|_{L^{\infty} _{dr} (\R^+; L^p _w (\R)) + L^2 _{dr}(\R^+ ; L^p _w (\R) )} < \infty \, . \]
\end{proof}

\section{Higher order remainder estimates: proof of Theorem \ref{Thm_remainder_Linfinity}}\label{Section_Pf_remainder_Linfinity}
We will spend the rest of this subsection proving Theorem \ref{Thm_remainder_Linfinity}. Suppose 
\[
  (w-1)\langle \lambda \rangle^k \in L^1 (\R) \, .
\]
Then $\Gamma \in C^k (\R^+)$ by Lemma \ref{Gamma_cts}, and so $R_{k,r}$ is well-defined. The following Lemma shows $\{a_{l, r}\}$ are uniformly bounded in $(\lambda,r)$.

\begin{Lem} \label{Remainder_coeffs_bded} Suppose $L \ddd \|(w-1)\langle \lambda \rangle^k\|_{L^1 (\R)} < \infty$. If $\|\langle \lambda \rangle^k (w-1)\|_{L^{\infty} (\R)} \leq \delta < 1$, then
	\[
		\|\partial_t ^j \Gamma_r (r, t)\|_{L^{\infty} ((0,r))} \lesssim \frac{L}{1-\delta} \, , \, \, j=0, \ldots, k \, . 
		\]
		In particular $|a_{l, r} (\lambda)| \lesssim \frac{L}{1-\delta}$.
\end{Lem}
\begin{proof}
  The bound on $|a_{l, r} (\lambda)|$ follows from its definition in \eqref{def_a} and the bound on $\|\partial_t ^j \Gamma_r (r, t)\|_{L^{\infty} ((0,r))}$. To get this latter bound, by the resolvent symmetries \eqref{Gamma_symmetries} it suffices to estimate $\Gamma_r (u,0)$ and all its derivatives in $[0,r]$. The resolvent identity \eqref{resolvent_identity_1} is equivalent to 
  \[
    (I + \HH_r)\Gamma_r (\cdot, s) = H(\cdot-s) \, , \quad s \in [0,r]
  \]
  and thus
	\[
		\Gamma_r (u,0) = (I + \HH_r)^{-1} H (u) \, . 
		\]

	      By Lemma \ref{Gamma_L2} we have $H(x) = \mathcal{F}^{-1} (w (2 \pi \cdot)-1)$; we in fact claim $\|\HH_r\|_{L^2 ([0,r]) ,  L^2 ([0,r])} \leq \delta < 1$. Indeed:
	\begin{align*}
	  \|\HH_r f\|_{L^2 (\R)} = \|\int\limits_0 ^r \mathcal{F}^{-1} (w(2 \pi \cdot)-1) (x-y) f(y) dy\|_{L^2 (\R)}  &= \|(w(2 \pi \cdot)-1) \mathcal{F}(f\chi_{[0,r]})\|_{L^2 (\R)} \\
	  &\leq \|w-1\|_{L^{\infty} (\R)} \|\mathcal{F} (f \chi_{[0,r]}) \|_{L^2 (\R)} \\
	  &\leq \delta \|f\|_{L^2 ([0,r])} \, .
	      \end{align*}

		Thus by geometric sum
	\[
	  \Gamma_r (u,0) = (I + \HH_r)^{-1} H (u)= \sum\limits_{l=0}^{\infty} (-\HH_r )^l H (u) \, .
		\]
	
	We estimate $\sum\limits_{l=0}^{\infty} (-\HH_r )^l H (u)$ and all its derivatives. Differentiate \eqref{H_operator_def} to get  
	\[
	  \left | \left ( \frac{d}{d u} \right )^j \HH_r (f) (u) \right | = \left |\int\limits_0 ^r H ^{(j)} (u-v) f(v) \, dv \right | \leq \|H ^{ (j)} \|_{L^2 (\R)} \|f\|_{L^2([0,r])} \, 
		\]
		for all $0 \leq j \leq k$. Take $f = (-\HH_r)^{l-1} H $ to get 
	\begin{align*}
	  \left | \left ( \frac{d}{d u} \right )^j (-\HH_r )^l H (u) \right | &\leq \| H^{ (j)} \|_{L^2 (\R)} \|\HH_r ^{l-1} H\|_{L^2 ([0,r])} \\
	  &= \| H^{ (j)} \|_{L^2 ([0,r])} \|\HH_r\|_{L^2 ([0,r]) , L^2 (\R)}^{l-1} \| H\|_{L^2 ([0,r])} \\
		&\lesssim L \delta ^{l} \, ,
	\end{align*}
		where in the last inequality we estimated 
		\[
		   \|H^{(j)}\|_{L^2 (\R)} \lesssim \|(w-1)\lambda^j\|_{L^2 (\R)} \leq \|(w-1)\langle \lambda \rangle^k\|_{L^2 (\R)} \leq  \delta^{1/2}  L ^{1/2} \, ,
		 \] and similarly for $\|H\|_{L^2 ([0,r])}$.

			Thus the sum $\sum\limits_{l=0}^{\infty} \left ( \frac{d}{d u} \right )^j (-\HH_r )^l H (u)$ converges absolutely and so one can use the dominated convergence theorem to show we get the desired estimate on 
			\[
			  \partial_u ^j \Gamma_r (u,0) = \partial_u ^j \sum\limits_{l=0}^{\infty} (-\HH_r )^l H (u) \, .
			\]
\end{proof}

Now we proceed with our usual functional-analytic approach.

\begin{Lem}Suppose that $\langle \lambda \rangle^k (w-1) \in L^1 (\R)$. Then
	\begin{equation}\label{Remainder_Exponential_System}
		R_{k,r} = \P_{[0,r]} R_{k, r} \, 
	\end{equation}
and 	\begin{equation}\label{Remainder_Ortho}
		\P_{[0,r]} (w-1) R_{k,r} = - \P_{[0,r]} R_{k,r}  - \P_{[0,r]} (w-1) (\lambda^k \e{i\lambda r} + \sum\limits_{l=1}^k a_{l, r}  \lambda^{k-l} )\, 
	\end{equation}
	in $L^p (\R)$ for $2 \leq p < \infty$.
		If in addition $|\langle \lambda \rangle^k (w-1)| \leq \delta < 1$, then
		\begin{equation}\label{Remainder_Linfinity_functional_1}
		  R_{k,r} (\lambda) =  - (I - \P_{[0,r]} (1-w))^{-1} \P_{[0,r]} (w-1) (\lambda^k \e{i\lambda r} + \sum\limits_{l=1}^k a_{l, r}  \lambda^{k-l} )\, , 
	      \end{equation}
	      where both sides are well-defined in, e.g., $L^2 (\R)$.	
\end{Lem}

\begin{proof}
  First note that $R_{k,r} \in L^p (\R)$ for $2 \leq p \leq \infty$. Indeed, $p=2$ follows by \eqref{Remainder_integral} and Plancherel, and $p=\infty$ from \eqref{Remainder_integral} and that $\Gamma_r \in C^k ([0,r])$. From this and the fact that $\langle \lambda \rangle ^k (w-1) \in L^1 (\R)$ we have both sides of \eqref{Remainder_Exponential_System} and \eqref{Remainder_Ortho} are sensical. Trivially, \eqref{Remainder_Exponential_System} holds. Meanwhile \eqref{Remainder_Ortho} is equivalent to 
	\[
		\P_{[0,r]} \left ( (w-1) R_{k,r} + R_{k,r}  + (w-1) (\lambda^k \e{i\lambda r} + \sum\limits_{l=1}^k a_{l, r} \lambda^{k-l} ) \right ) = 0\, 
		\]
		or rather
		\[
		  \P_{[0,r]}  (w \lambda^k P(r, \lambda) - \lambda^k \e{i \lambda r} -\sum\limits_{l=1}^k a_{l, r} \lambda^{k-l})= 0\, , 
			\]
			which follows by applying Lemma \ref{orthogonal_inner_prod} and \eqref{def_a}. 

		Add \eqref{Remainder_Ortho} and \eqref{Remainder_Exponential_System} to yield
		\[
		(I - \P_{[0,r]} (1-w)) R_{k,r} (\lambda) =  - \P_{[0,r]} (w-1) (\lambda^k \e{i\lambda r} + \sum\limits_{l=1}^k a_{l, r}  \lambda^{k-l} )\, .
	      \]
	      Then on $L^2 (\R)$, we have 
	      \[
		\|\proj{0}{r}{} (1-w)\|_{2,2} \leq \left \| \proj{0}{r}{} \right \|_{2,2} \|1-w\|_{L^{\infty} (\R)} \leq 1 \cdot \delta < 1 \, .
	      \]
	      Thus the operator $(I - \P_{[0,r]} (1-w))$ has bounded inverse $\sum\limits_{k=0}^{\infty} (\P_{[0,r]} (1-w))^k$ on $L^2 (\R)$, and so \eqref{Remainder_Linfinity_functional_1} follows.
	    \end{proof}

	    We will also need the following Lemma to estimate the Fourier projections.
	    \begin{Lem}\label{Lemma_projections_bded_Linfinity}
	      If $\delta \in (0,1)$, then there exists $\epsilon (\delta) \in (0, \frac{1}{2} )$, with $\lim\limits_{\delta \to 0} \epsilon (\delta) = \frac{1}{2}$, such that for all $p$ satisfying $|\frac{1}{p} - \frac{1}{2}| < \epsilon (\delta)$ we have
	      \[
		\|\proj{0}{r}{}\|_{p,p} \leq \frac{1}{\sqrt{\delta}} \, .
	      \]
	    \end{Lem}
	    \begin{proof}

	      By \eqref{projection_as_hilbert_transforms} and self-adjointness of $\proj{0}{r}{}$, it follows that $ \|\proj{0}{r}{}\|_{p',p'} = \|\proj{0}{r}{}\|_{p,p} \leq \|\cal{H}\|_{p,p}$, where $\cal{H}$ is the Hilbert Transform. By \cite{Pichorides1972}, $1 \leq \|\cal{H}\|_{p,p} = \tan (\frac{\pi}{2 p})$ for $p \in (1, 2]$. Let $p_0  = p_0 (\delta) \in (1,2]$ be the unique element of $(1,2]$ that satisfies $\tan (\frac{\pi}{2 p_0}) = \delta^{-1/2}$, and let $\epsilon (\delta) \ddd \frac{1}{p_0} - \frac{1}{2}$. 

	      By duality and \eqref{projection_as_hilbert_transforms}, we have $\|\proj{0}{r}{}\|_{p_0 ',p_0 '},  \|\proj{0}{r}{}\|_{p_0,p_0} \leq \delta^{-1/2}$. Interpolate between both estimates to get
\[
		\|\proj{0}{r}{}\|_{p,p} \leq \frac{1}{\sqrt{\delta}} \, 
	      \]
	      for all $p$ satisfying $|\frac{1}{p} - \frac{1}{2}| < \epsilon (\delta)$.

	      To see that $\lim\limits_{\delta \to 0} \epsilon(\delta) =\frac{1}{2}$, it suffices to note that $\tan (\frac{\pi}{2 p})$ is increasing in $\frac{1}{p}$ for $p \in (1, 2]$ and has singularity at $\frac{1}{p} =1$. Thus as $\delta \to 0$, we have $\delta^{-1/2} \to \infty$, meaning that for our definition of $p_0$, we have $\frac{1}{p_0} \to 1$ and so $\epsilon(\delta) \to \frac{1}{2}$. 
	    \end{proof}

We are now in a position to prove Theorem \ref{Thm_remainder_Linfinity}.

\begin{proof}[Proof of Theorem \ref{Thm_remainder_Linfinity}]
  By \eqref{Remainder_Linfinity_functional_1} and Lemma \ref{Remainder_coeffs_bded}, it follows that 
  \[
    \|R_{k,r}\|_{L^p (\R)} \lesssim_k \|(I - \P_{[0,r]} (1-w))^{-1}\|_{p,p} \|\proj{0}{r}{}\|_{p,p} \|\langle \lambda \rangle^k (w-1)\|_{L^p (\R)} \left ( 1 + \frac{\|\langle \lambda \rangle^k (w-1)\|_{L^1 (\R)}}{1- \delta} \right ) \, .
  \]
  First note that by \eqref{projection_as_hilbert_transforms}, we have $\|\proj{0}{r}{}\|_{p,p} \lesssim \|\cal{H}\|_{p,p} \lesssim_p 1$.

  Next, by Lemma \ref{Lemma_projections_bded_Linfinity} we can choose $\epsilon (\delta) \in (0,\frac{1}{2})$, with $\lim\limits_{\delta \to 0} \epsilon (\delta) = \frac{1}{2}$, such that for all $p$ satisfying $|\frac{1}{p} - \frac{1}{2}| < \epsilon (\delta)$, we have 
	\[
		\|\P_{[0,r]}(1-w)\|_{p,p} \leq \|\P_{[0,r]}\|_{p,p} \|1-w\|_{\infty} \leq \sqrt{\delta} < 1 \, .
		\]
		Now by geometric sum, $I - \P_{[0,r]}(1-w)$ has inverse $\sum\limits_{j=0}^{\infty} (\proj{0}{r}{} (1-w))^k$ on $L^p (\R)$, with the bound
	\[
	  \|(I-\P_{[0,r]}(1-w))^{-1}\|_{p,p} \leq \sum\limits_{k=0}^{\infty} \| (\P_{[0,r]} (1-w))^k\|_{p,p} \leq \frac{1}{1-\delta^{1/2}} \sim \frac{1}{1-\delta}
		\]
		whenever $|\frac{1}{p} - \frac{1}{2}| < \epsilon (\delta)$.

  Combine these estimates with the fact that $\frac{1}{1-\delta} \geq 1$, we get 
  \begin{equation}\label{remainder_Linfinite_best_estimate}
    \|R_{k,r}\|_{L^p (\R)} \lesssim_{p,k} \frac{1}{(1- \delta)^2} \|\langle \lambda \rangle^k (w-1)\|_{L^p (\R)} (1 + \|\langle \lambda \rangle^k (w-1)\|_{L^1 (\R)}) \, 
\end{equation}
for all $p$ satisfying $|\frac{1}{p} - \frac{1}{2}| < \epsilon (\delta)$, from which part \ref{general_Nazarov_remainder} follows. As for part \ref{perturbative_Nazarov_remainder}, given a $p \in (1,\infty)$ choose $\delta_0 (p)$ small enough so that $|\frac{1}{p} - \frac{1}{2}| < \epsilon (\delta)$ for all $\delta \in (0, \delta_0)$. Then the estimate 
\[
  \|\langle \lambda \rangle^k (w-1)\|_{L^p (\R)} \leq \delta^{1/p'} \|\langle \lambda \rangle^k (w-1)\|_{L^1 (\R)} ^{1/p} \, ,
\]
combined with \eqref{remainder_Linfinite_best_estimate}, yields part \ref{perturbative_Nazarov_remainder}.
  \end{proof}

\section{Proof of Lemma \ref{lemma_mu_weight_Lp}}\label{tech_lemma_pf_section}
Before proving Lemma \ref{lemma_mu_weight_Lp}, we will need terminology and lemmas for a Calder\'on-Zygmund decomposition.

\begindef An interval $[a, a+r]$ has \textit{left neighbor} $[a-r,a]$ and \textit{right neighbor} $[a+r, a+2r]$.

Two intervals are \textit{almost disjoint} if their intersection is empty or a single point.

Given an interval $I$, let $D_I \ddd \mathrm{dist} (I, 0)$ denote its distance from the origin.

An interval $I$ in $\R$ is \textit{dyadic} if it is of the form $\{[j 2^{-n}, (j+1)2^{-n}]\}_{j, n \in \Z}$.

Two dyadic intervals $I$ and $J$ are \textit{siblings} if their union is a dyadic interval $K$ of strictly larger size, which we call the \textit{parent}.

So define a partial ordering $\preceq$ on the dyadic intervals: if $I \subset K$, we write $I \preceq K$, and say $K$ is an \textit{ancestor} of $I$.
\smallskip
Note that two dyadic intervals are either almost disjoint, or comparable via $\preceq$.

See e.g.\ \cite[Chapter 1, Section 3]{stein} or \cite[Chapter 1, Theorem 4]{stein_singular_ints} for other variants of the Calder\'on-Zygmund decomposition below.
\begin{Lem}[Calder\'on-Zygmund decomposition]\label{CZ_decomp_lemma} Let $d\mu = \langle \lambda \rangle^q d \lambda$ for some $q \geq 0$, and suppose $u \in L^1 _{\mu} (\R)$. If $\beta>0$ and $E_{\beta} \ddd \{\lambda \in \R~:~ |u| > \beta\}$, then there exists a collection of dyadic intervals $\{I_j\}$ with the following properties:
	\begin{enumerate}[label=(\roman*)]
		\item\label{smallset_contained}$
		E_{\beta} \ddd \{\lambda \in \R~:~ |u| > \beta\} \subset \bigcup\limits_{j} I_j
		$.
	\item \label{maximal_intervals} Each $I_j$ is maximal with respect to $\preceq$ among those dyadic intervals $I$ satisfying $\beta < \frac{1}{\mu(I)} \int\limits_{I} |u| d\mu$.
	\item \label{CZ_disjoint} The $\{I_j\}$ are pairwise almost disjoint.
\item We have the estimate
	\begin{equation}\label{sum_intervals_CZ_mu}
		\mu (E_{\beta}) \leq \sum\limits_j \mu(I_j) \leq \frac{1}{\beta} \|u\|_{L^1 _{\mu} (\R)} < \infty \, .
	\end{equation}
	\end{enumerate}
	This is known as the Calder\'on-Zygmund decomposition of $u$ at level $\beta$.
\end{Lem}
\begin{proof}
  Fix $\beta \in (0,1)$, and by substituting $u$ with $|u|$, assume without loss of generality that $u \geq 0$. Since $u \in L^1 _{\mu} (\R)$, then for any interval $I$ of size $|I| \geq \frac{1}{\beta} \int\limits_{\R} u \, d\mu$, we have $\langle u \rangle_{I, \mu} \leq \beta$, since $\mu(I)\geq |I|$. Thus any dyadic interval $J$ has an ancestor $I_0 \succeq J$ such that for all dyadics $I \succeq I_0$, we have $\langle u \rangle_{I, \mu} \leq \beta$. Meaning if we consider the set of dyadic intervals
	\[
	  \cal{I} \ddd \{I \text{ dyadic} ~:~	\beta < \langle u \rangle_{I, \mu} \} \, ,
		\]
		then for each $I \in \cal{I}$, there exists $I_{\max} \in \cal{I}$ such that $I \preceq I_{\max}$, and $I_{\max}$ is maximal in $\cal{I}$ with respect to $\preceq$. Define $\{I_j\}$ to be the set of maximal dyadic intervals in $(\cal{I}, \preceq)$. By definition, $\{I_j\}$ satisfies Property \ref{maximal_intervals}, which then immediately implies Property \ref{CZ_disjoint}.

		Meanwhile \eqref{sum_intervals_CZ_mu} follows from Property \ref{smallset_contained} and that $\beta < \frac{1}{\mu(I_j)} \int\limits_{I_j} u \, d\mu$: 
	\[
		\mu (E_{\beta}) \leq\sum\limits_j \mu(I_j) \leq \sum\limits_{j} \frac{1}{\beta} \int\limits_{I_j} |u| d\mu \leq \frac{1}{\beta} \|u\|_{L^1 _{\mu} (\R)} < \infty \, .
		\]

		We are left with showing Property \ref{smallset_contained}, which will follow if we can show $u \leq \beta$ for almost every $x$ in $(\bigcup\limits_j I_j)^c$. Consider the $L^1 (\R)$ function $v \ddd u \langle \lambda \rangle^q$. By the dyadic Lebesgue differentiation theorem, for almost every $x \in \R$, we have $\lim\limits_{n \to \infty} \langle v \rangle_{J_n} =v(x)$, where $J_n$ is any sequence of dyadic intervals shrinking to $x$. Fix such a point $x \in (\bigcup\limits_j I_j)^c$. Then by construction of $\{I_j\}$, for every dyadic interval $J$ containing $x$, it follows that 
	\[
	  \beta \geq \langle u \rangle_{J, \mu} = \frac{|J|}{\mu(J)} \langle v \rangle_J \, .
		\]
	Letting $J$ shrink down to $x$, the above yields
	\[
		\beta \geq \frac{1}{\langle x \rangle^q} u(x) \langle x \rangle^q  \, ,
		\]
		i.e.\ $\beta \geq u(x)$.
	Since $x$ arbitrary in $(\bigcup\limits_j I_j)^c \setminus N$ where $N$ is a set of measure $0$, we get Property \ref{smallset_contained}.
\end{proof}

\begin{Lem}[Calder\'on-Zygmund decomposition for $L^{p_1} (\R) +L^{p_2} (\R)$]\label{CZ_decomp_lemma_L12}
  Let $d\mu = \langle \lambda \rangle^q d \lambda$ for some $q \geq 0$, and suppose $u = u_1 + u_2$ with $u_1 \in L^{p_1} _{\mu} (\R)$, $u_2 \in L^{p_2} _{\mu} (\R)$, where $1 \leq p_1 \leq p_2 < \infty$. For each $\beta > 0$, define $E_{\beta} \ddd \{\lambda \in \R~:~ |u| > \beta\}$ and $E ^i _{\beta/2} \ddd \{ \lambda ~:~ |u_i| > \frac{\beta}{2} \}$, $i=1,2$. Then $E_{\beta } \subset E_{\beta/2} ^1 \cup E_{\beta/2} ^2$, all of which are of finite $\mu$-measure, and there exists a collection of dyadic intervals $\{I_j\}$ with the following properties:
\begin{enumerate}[label=(\roman*)]
		\item\label{smallset_contained_L12}$
		E_{\beta/2} ^1 \cup E_{\beta/2} ^2 \subset \bigcup\limits_{j} I_j
		$.
	%\item \label{maximal_intervals} Each $I_j$ is maximal with respect to $\preceq$ among those dyadic intervals $I$ satisfying $\beta < \frac{1}{\mu(I)} \int\limits_{I} |u| d\mu$.
	%\item \label{outside_CZ} Outside of $\{I_j\}$, we have $|w-1| \leq \beta$ $\mu$-a.e.
	\item \label{CZ_disjoint_L12} The $\{I_j\}$ are pairwise almost disjoint.
\item We have the estimate
	\begin{equation}\label{sum_intervals_CZ_L12}
	  \sum\limits_j \mu(I_j) \lesssim_{p_1, p_2} \frac{1}{\beta^{p_2}} (\beta^{p_2 -p_1} \|u_1\|_{L^{p_1} _{\mu} (\R)} ^{p_1}+ \|u_2\|_{L^{p_2} _{\mu} (\R)} ^{p_2} ) \, .
	\end{equation}
\item We have
	\begin{equation}\label{avg_w_close_1_L12}
	  \langle |u| \rangle_{K, \mu} \leq \beta \, , \quad K \text{ any ancestor of } I_j \, . 
		\end{equation}

	\end{enumerate}
\end{Lem}
\begin{proof}
	By the triangle inequality, $E_{\beta } \subset E_{\beta/2} ^1 \cup E_{\beta/2} ^2$. Thus $\mu(E_{\beta}) \leq \mu(E_{\beta/2} ^1) + (E_{\beta/2} ^2) < \infty$, since $u_i \in L^{p_i} _{\mu} (\R)$.

	For $i=1,2$, we apply Lemma \ref{CZ_decomp_lemma} to $|u_i|^{p_i}$ at level $\left ( \frac{\beta}{2} \right )^{p_i}$ and get a collection of pairwise disjoint dyadic intervals $\{I_j ^i\}_{j \geq 0}$ such that 
	\[
		E^{i} _{\beta /2} \subset \bigcup\limits_j I_j ^i \, . 
		\] By Lemma \ref{CZ_decomp_lemma} \ref{maximal_intervals},
		\[
			\langle |u_i|^{p_i} \rangle_{I_j ^i , \mu} ^{1/p_i} > \frac{\beta}{2} \, .
			\]
			Let $\{\dot{I_j}\}_j = \{ I_j ^1 \}_j \cup \{ I_j ^2 \}_j$. Note by maximality of the $I_{j} ^i$'s specified by Lemma \ref{CZ_decomp_lemma} \ref{maximal_intervals}, each $\dot{I_j}$ is contained in at most one other $\dot{I_k}$. So now let $\{I_j\}$ be the maximal intervals among $\{\dot{I_j}\}$ with respect to $\preceq$; maximality ensures the $\{I_j\}$ are pairwise almost disjoint and so \ref{CZ_disjoint_L12} holds. Note maximality also yields
	\[
	E_{\beta} \subset E_{\beta/2} ^1 \cup E_{\beta/2} ^2 \subset \left ( \bigcup\limits_{j} I_j ^1 \right ) \cup \left (\bigcup\limits_{j} I_j ^2 \right )\subset \bigcup\limits_j I_j \, , 
		\]
		i.e.\  \ref{smallset_contained_L12} holds. And \eqref{sum_intervals_CZ_mu} yields \eqref{sum_intervals_CZ_L12}, as 
	\[
		\sum\limits \mu(I_j) \leq \sum\limits_j \mu(I_j ^1) + \sum\limits_j \mu(I_j ^2) \lesssim_{p_1, p_2} \frac{1}{\beta^{p_1}} \|u_1\|_{L^{p_1} _{\mu} (\R)} ^{p_1} + \frac{1}{\beta^{p_2}} \|u_2\|_{L^{p_2} _{\mu} (\R)} ^{p_2}  = \frac{1}{\beta^{p_2}} (\beta^{p_2 -p_1} \|u_1\|_{L^{p_1} _{\mu} (\R)} ^{p_1}+ \|u_2\|_{L^{p_2} _{\mu} (\R)} ^{p_2} ) \, .
	\]

Also note by our choice of $I_j$ as the maximal intervals among $\{ I_j ^1 \}_j \cup \{ I_j ^2 \}_j$, then for any ancestor $K$ of $I_j$, it follows that 
	\[
		\langle |u_1|^{p_1} \rangle_{K , \mu} ^{1/p_1} \, , \, \langle |u_2|^{p_2} \rangle_{K , \mu} ^{1/p_2} \leq \frac{\beta}{2} \, .
		\]
		The triangle inequality followed by H\"older's inequality then imply \eqref{avg_w_close_1_L12}, i.e.\
		\[
			\langle |u| \rangle_{K , \mu} \leq \langle |u_1| \rangle_{K , \mu}  + \langle |u_2| \rangle_{K , \mu} \leq \langle |u_1|^{p_1} \rangle_{K , \mu}^{1/p_1}  + \langle |u_2|^{p_2} \rangle_{K , \mu} ^{1/p_2} \leq \beta \, .
			\]
\end{proof}

We also need the following additional lemma. Recall that for an interval $I$, we defined $D_I = \mathrm{dist}\left( I,0 \right)$.

\begin{Lem}[$\mu$ is flat away from $0$]\label{mu_flat}
  Let $d \mu =\langle \lambda \rangle^q d \lambda$ for some $q \geq 0$, and suppose $I$ is an interval such that $\mu (I) \leq L$, where $L$ is some constant. Then there exists $D = D(L,q) \geq 0$ sufficiently large so that if $D_I \geq D$, then
	\begin{equation}\label{flat_measure_1}
	  \langle \lambda \rangle ^q \sim_q \langle D_I \rangle^q \, , \quad \lambda \text{ in } I \text{ or either of its neighbors} \, ,
	\end{equation}  In particular,
	\begin{equation}\label{flat_measure_avg}
	  \langle f \rangle_{J, \mu} \sim_q \langle f \rangle_J
	\end{equation}
	for all $f \geq 0$, where $J$ may equal $I$, either of its neighbors, or its parent in the case that $I$ is dyadic.
\end{Lem}
	\begin{proof}
	  If $q=0$, the lemma is trivial. %\eqref{flat_measure_avg} is trivial, and \eqref{flat_measure_1} holds for e.g.\ $D = 100 L$. %then take $D$ large enough that $|L| \leq D$, so that $\langle D_I \rangle \leq \langle \lambda \rangle \leq \langle 2 D_I \rangle$.

	  For $q > 0$, choose $D$ so large that $\langle D \rangle^q \geq 100 L$. Thus if $D_I \geq D$, then
		\[
			100 L |I|\leq |I| \langle D \rangle^q \leq |I| \langle D_I \rangle^q \leq \mu(I)\leq L			\]
			and so $|I| < \frac{1}{100}$. Thus on $I$ or either of its neighbors, we have $\langle \lambda \rangle^q \sim_q \langle D_I \rangle^q$. Hence for $J=I$ or either of its neighbors, \eqref{flat_measure_avg} holds as
		\[
		  \langle f \rangle_{J , \mu} = \frac{1}{\mu(J)} \int\limits_J f d \mu = \frac{1}{\mu(J)} \int\limits_J f(\lambda) \langle \lambda \rangle^q d \lambda \sim_q \frac{1}{\langle D_I \rangle^q |J|} \int\limits_J f(\lambda) \langle D_I \rangle^q d \lambda = \langle f \rangle_J \, .
			\]
			If $I$ is dyadic, then since the parent of $I$ is the union of $I$ and one of its neighbors, it follows $\langle \lambda \rangle^q \sim_q \langle D_I \rangle^q $ on the parent of $I$. Then the same computations as done previously yield \eqref{flat_measure_avg} when $J$ is the parent of $I$. 
	\end{proof}

%%\ref{Reverse_Holder_A2}

\begin{proof}[Proof that $\|w^{1/\widetilde{p}} - w^{-1/\widetilde{p}'}\|_{L^p _{\mu}  (\R)}<\infty$ in Lemma \ref{lemma_mu_weight_Lp}]
  Let $\widetilde{p} \in (1, \infty)$ and $p \geq p_2$. Let's apply Lemma \ref{CZ_decomp_lemma_L12} with $u = w-1$ and $\beta = \alpha$, where $\alpha \in (0,\frac{1}{2})$, so that $E_{\alpha} \ddd \{\lambda ~:~ |w(\lambda)-1| > \alpha\}$ and $E ^i _{\alpha/2} \ddd \{ \lambda ~:~ |u_i (\lambda)| > \alpha /2 \}$. Then $E_{\alpha} \subset E_{\alpha/2} ^1 \cup E_{\alpha/2} ^2$.

We begin by splitting
	\begin{equation}\label{first_split_L12}
		\|w^{1/\widetilde{p}} - w^{-1/\widetilde{p}'}\|_{L^p _{\mu} (\R)} ^p = \|w^{1/\widetilde{p}} - w^{-1/\widetilde{p}'}\|_{L^p _{\mu} ((E_{\alpha/2} ^1 \cup E_{\alpha/2} ^2) ^c )} ^p + \|w^{1/\widetilde{p}} - w^{-1/\widetilde{p}'}\|_{L^p _{\mu} (E_{\alpha/2} ^1 \cup E_{\alpha/2} ^2)}^p  \ddd A + B\, .
	\end{equation}
	Since $(E_{\alpha/2} ^1 \cup E_{\alpha/2} ^2) ^c \subset E_{\alpha} ^c = \{|w-1| \leq \alpha \}$, we may estimate
\[
	A = \int\limits_{(E_{\alpha/2} ^1 \cup E_{\alpha/2} ^2) ^c} w^{-p/\widetilde{p}'} |w-1|^p \, d \mu \leq (1-\alpha)^{-p/\widetilde{p}'} \int\limits_{(E_{\alpha/2} ^1 )^c \cap (E_{\alpha/2} ^2)^c} |w-1|^p \, d\mu \,  .
	\]
	By the triangle inequality, $A$ is  
\[
  \lesssim_{p, \widetilde{p}} \int\limits_{(E_{\alpha/2} ^1 )^c \cap (E_{\alpha/2} ^2)^c} |u_1|^p \, d\mu + \int\limits_{(E_{\alpha/2} ^1 )^c \cap (E_{\alpha/2} ^2)^c} |u_2|^p \, d\mu \, \lesssim_{p_1, p_2} \alpha^{p-p_1}\int\limits_{(E_{\alpha/2} ^1 )^c} |u_1|^{p_1} \, d\mu + \alpha^{p-p_2} \int\limits_{ (E_{\alpha/2} ^2)^c } |u_2|^{p_2} \, d\mu 
	\]
	Thus, 
	\begin{equation}\label{bd_A_L12}
	  A \lesssim_{p, \widetilde{p}, p_1, p_2} \alpha^{p-p_2} (\alpha^{p_2 - p_1} \|u_1\|_{L^{p_1} _{\mu} (\R)} ^{p_1} + \|u_2\|_{L^{p_2} _{\mu} (\R)} ^{p_2}) \, . 
	\end{equation}

By the triangle inequality followed by Lemma \ref{CZ_decomp_lemma_L12} \ref{smallset_contained_L12},
	\begin{equation}\label{B_first_split_L12}
		B \lesssim_p \int\limits_{E_{\alpha/2} ^1 \cup E_{\alpha/2} ^2} w^{p/\widetilde{p}} d \mu + \int\limits_{E_{\alpha/2} ^1 \cup E_{\alpha/2} ^2} w^{-p/\widetilde{p}'} d \mu  \leq \left ( \sum\limits_{j} \int\limits_{I_j} w^{p/\widetilde{p}} d \mu \right )+ \left ( \sum\limits_{j} \int\limits_{I_j} w^{-p/\widetilde{p'}} d \mu \right ) \ddd B_1+B_2\, .
	\end{equation}

	Recall $w \in A_{p'} (\R)$ and so $w^{-p/p'} \in A_p(\R)$. By applying Lemma \ref{Reverse_Holder_A2} to both of these weights, we can choose $\epsilon (\gamma, p) \in (0, \frac{1}{2})$, with $\lim\limits_{\gamma \to 1} \epsilon(\gamma,p) = \min\{\frac{1}{p}, \frac{1}{p'}\}$, so that whenever $\widetilde{p}$ is chosen so that $|\frac{1}{\widetilde{p}} - \frac{1}{p}| < \epsilon (\gamma, p)$, we have 
	\begin{equation}\label{Reverse_Holder_integrability_Ap_weights}
	  \langle w^{p/\widetilde{p}} \rangle_{I} \lesssim_{\gamma, p} \langle w \rangle_{I} ^{p/\widetilde{p}} \, , \quad  \langle w^{-p/\widetilde{p'}} \rangle_{I} \lesssim_{\gamma, p} \langle w^{-p/p'} \rangle_{I}^{p'/\widetilde{p'}} \, , \quad \text{for all intervals } I \, .
	\end{equation}
	In particular, this means $w^{p/\widetilde{p}}$ and $w^{-p/\widetilde{p}'}$ are locally integrable.

Let us focus on $B_2$ first, as $B_1$ will be similar. Set $D_j \ddd \mathrm{dist} (I_j, 0)$, then write
	\[
		\sum\limits_j \int\limits_{I_j} w^{-p/\widetilde{p'}} d \mu = \sum\limits_{j \, : \, D_j < D} \int\limits_{I_j} w^{-p/\widetilde{p'}} d \mu + \sum\limits_{j \, : \, D_j \geq D} \int\limits_{I_j} w^{-p/\widetilde{p'}} d \mu \, ,
		\]
			where $D>0$ is a constant which will be chosen later.

			Since $|I_j| \leq \mu (I_j) \lesssim_{p_1, p_2} \frac{1}{\alpha^{p_2}} (\beta^{p_2 -p_1} \|u_1\|_{L^{p_1} _{\mu} (\R)} ^{p_1}+ \|u_2\|_{L^{p_2} _{\mu} (\R)} ^{p_2} )$ by \eqref{sum_intervals_CZ_L12}, then the length of each interval $I_j$ is bounded. Since $w^{-p/\widetilde{p}'}$ is locally integrable, then 
			\[
				\sum\limits_{j \, : \, D_j < D} \int\limits_{I_j} w^{-p/\widetilde{p}'} d \mu \leq \int\limits_{\widetilde{I}} w^{-p/\widetilde{p}'} \langle \lambda \rangle^q \, d \lambda < \infty \, ,
				\]
				where $\widetilde{I}$ is some sufficiently large interval centered at $0$. And so it suffices to show
				\[
					\sum\limits_{j \, : \, D_j \geq D} \int\limits_{I_j} w^{-p/\widetilde{p}'} d \mu =  \sum\limits_{j \, : \, D_j \geq D} \mu (I_j) \langle  w^{-p/\widetilde{p}'} \rangle_{I_j , \mu} < \infty \, .
					\]
					Since $\sum_j \mu (I_j) < \infty$, to show $B_2 < \infty$, it suffices to show $\langle  w^{-p/\widetilde{p}'} \rangle_{ I_j, \mu} \lesssim_{p, \gamma, q} 1$; similarly $B_1 < \infty$ will follow from showing $\langle  w^{p/\widetilde{p}} \rangle_{I_j , \mu} \lesssim_{p, \gamma} 1$ for a suitable choice of $D$.

					Since there exists some constant $C = C(p_1, p_2)$ such that 
					\[
	\mu (I_j) \leq C \frac{1}{\alpha^{p_2}} \left ( \alpha^{p_2 -p_1} \|u_1\|_{L^{p_1} _{\mu} (\R)} ^{p_1}+ \|u_2\|_{L^{p_2} _{\mu} (\R)} ^{p_2} \right ) \, ,
	\]
	we can take $D = D(C \frac{1}{\alpha^{p_2}} (\alpha^{p_2 -p_1} \|u_1\|_{L^{p_1} _{\mu} (\R)} ^{p_1}+ \|u_2\|_{L^{p_2} _{\mu} (\R)} ^{p_2} ),q)$ as in Lemma \ref{mu_flat} so that when $D_j \geq D$, we have $\langle \lambda \rangle^q \sim_q \langle D_j \rangle^q$ on $I_j$, its neighbors and its parent $K_j$.

	For such an interval $I_j$, recall \eqref{avg_w_close_1_L12} implies 
	\[
	|\langle w-1 \rangle_{K_j ,\mu}| \leq  \langle |w-1| \rangle_{K_j ,\mu} \leq \alpha \leq \frac{1}{2} \, ,
	\]
	where $K_j$ is the parent of $I_j$. And so $\langle \lambda \rangle^q \sim_q \langle D_j \rangle^q$ on $K_j$ and \eqref{flat_measure_avg} together imply 
			\[
				\langle w \rangle_{K_j} \sim \langle w \rangle_{K_j , \mu } = 1  + \langle w-1 \rangle_{K_j , \mu } \sim 1\, . 
				\]
				In particular, 
				\[
				  \langle w \rangle_{I_j , \mu} \sim_q \langle w \rangle_{I_j} \lesssim \langle w \rangle_{K_j}  \sim 1 \, 
					\]
				and so  \eqref{Reverse_Holder_integrability_Ap_weights} implies
				\[
				  \langle w^{p/\widetilde{p}} \rangle_{I_j , \mu} \sim_q \langle w^{p/\widetilde{p}} \rangle_{I_j} \lesssim_{\gamma, p} \langle w \rangle_{I_j} ^{p/\widetilde{p}} \sim_{q, p, \gamma} \langle w \rangle_{I_j , \mu} ^{p/\widetilde{p}} \lesssim_{p, \gamma, q} 1 \, .
					\]
				 Similarly, since $w^{-p/p'} \in A_p (\R)$, then 
		\[
		  \langle w^{-p/p'} \rangle_{ I_j , \mu} \sim_q \langle w^{-p/p'} \rangle_{I_j} \lesssim \langle w^{-p/p'} \rangle_{K_j} \leq [w^{-p/p'}]_{A_p (\R)} \langle w \rangle_{K_j} ^{-p'/p} \lesssim_{p, \gamma} 1 \, 
		\]
and therefore \eqref{Reverse_Holder_integrability_Ap_weights} implies
				\[
				  \langle w^{-p/\widetilde{p}'} \rangle_{ I_j , \mu} \, , \quad \langle w ^{p/\widetilde{p}}\rangle_{ I_j , \mu}   \lesssim_{p, \gamma} 1 \, .
					\]
\end{proof}

For the proof of \eqref{perturbative_est_weight}, we need to understand how $[w]_{A_{\infty} (\R)}$ affects $\log w$.

\begin{Lem}\label{avg_logw_vs_w-1_lemma}
  Suppose $w \in A_{\infty} (\R)$ and let $J$ be an interval. If $\langle |w-1| \rangle_J \leq \frac{1}{2}$, then 
	\[
		|\langle \log w \rangle_J| \lesssim \langle |w-1| \rangle_J + \log [w]_{A_{\infty} (\R)} \, .
		\]
\end{Lem}
\begin{proof}
	By Jensen's inequality, 
	\[
		\langle \log w \rangle_J \leq \log \, \langle w \rangle_J = \log (\langle w-1 \rangle_J + 1)\leq \log (\langle |w-1| \rangle_J + 1) \lesssim \langle |w-1| \rangle_J \, .
	\]
	But also since $[w]_{A_{\infty} (\R)} <\infty$, then
	\[
		\e{- \langle \log w \rangle_J} \leq \frac{[w]_{A_{\infty} (\R)}}{\langle w \rangle_J} \leq \frac{[w]_{A_{\infty} (\R)}}{1 - \langle |w-1| \rangle_J} \, ,
		\]
		 which implies
		\[
			\langle \log w \rangle_J \geq \log (1 - \langle |w-1| \rangle_J) - \log [w]_{A_{\infty} (\R)} \, .
			\]

	Combine the estimates of $\langle \log w \rangle_J$ from above and below to get
	\[
		|\langle \log w \rangle_J | \lesssim \langle |w-1| \rangle_J + |\log (1 - \langle |w-1| \rangle_J) | + \log [w]_{A_{\infty} (\R)} \lesssim  \langle |w-1| \rangle_J + \log [w]_{A_{\infty} (\R)} \, .
		\]
\end{proof}

Next, we recall a few facts about $\rm{BMO}$, the space of functions with bounded mean oscillation, and how it relates to $A_p (\R)$ weights. Recall that $f\in {\rm BMO}(\R^d)$ if
\[
\|f\|_{{\rm BMO}(\R^d)} \ddd \sup_{B}\,\langle|f-\langle f\rangle_B|\rangle_B<\infty\,,
\]
where $B$ denotes a ball in $\R^d$ (see, e.g., p.140 in \cite{stein}). Functions in $\rm BMO (\R^d)$ all satisfy the John-Nirenberg estimates below.

\begin{Thm}[John-Nirenberg, {\cite[p.144-146]{stein}}]\label{John-Nirenberg_Thm}
	Suppose $f \in \mathrm{BMO}(\R^d)$. Then
	\begin{enumerate}[label=(\alph*)]
		\item There exist positive absolute constants $c_1, c_2$ such that for each $\alpha > 0$ and every ball $B$,
			\[
				\frac{1}{|B|} |\{x \in B ~:~ |f(x) - \langle f \rangle_B | > \alpha\}| \leq c_1 \e{-c_2 \alpha /\|f\|_{\mathrm{BMO}}} \, .
				\]
			\item For any $p < \infty$, $f \in L^p_{loc} (\R^d)$ and
			\[
				\langle |f - \langle f \rangle_B|^p \rangle_B \lesssim_p \|f\|_{\mathrm{BMO} (\R^d)} ^p \, 
				\]
				for all balls $B$.
			\item %\label{exp_integral_JN}
			  If $0 \leq  \mu \|f\|_{\mathrm{BMO}(\R^d)} \lesssim 1$, then
			\[
				\langle \e{\mu|f-\langle f \rangle_B|} \rangle_B \lesssim 1 \, .
				\]
	\end{enumerate}
\end{Thm}

It is well known that if $w \in A_p (\R^d)$, then $\log w \in \mathrm{BMO} (\R^d)$. We also have the following well-known quantification  (see, e.g., \cite[Corollary 6]{korey}).
\begin{Lem}\label{sd_ll}
  If $w \in A_p (\R)$ for some $p \in (1,\infty]$, then $\|\log w\|_{\rm BMO} \leq \log (2 [w]_{A_p (\R^d)})$. If in addition $[w]_{A_p(\R^d)}=1+\tau, \tau\in [0,1]$, then
\[
\|\log w\|_{\rm BMO}\lesssim \sqrt\tau\,.
\]
\end{Lem}

\begin{proof}[Proof of \eqref{perturbative_est_weight}] By Lemma \ref{sd_ll}, when $\tau_0 (p) \leq \frac{1}{2}$ we have $\|\log w\|_{\mathrm{BMO}(\R)} \lesssim \tau^{1/2}$. Take $\tau_0 (p)$ small enough so that additionally $\|\log w\|_{\mathrm{BMO}(\R)} \leq \frac{1}{10}$.  

  Set $\alpha \ddd \tau^{1/2}$ and define $E_{\alpha}$, $E_{\alpha/2} ^i$ and $\{I_j\}$ as in the proof that $\|w^{1/\widetilde{p}} - w^{-1/\widetilde{p}'} \|_{L^p _{\mu} (\R)} <\infty$, i.e.\ by applying Lemma \ref{CZ_decomp_lemma_L12} with $u=|w-1|$ and $\beta = \alpha$. We begin by splitting $\|w^{1/p} - w^{-1/p'}\|_{L^p (\R)} ^p$ as in \eqref{first_split_L12}, with the same definitions of $A$ and $B$. By \eqref{bd_A_L12} and our choice of $\alpha$, it follows that 
	\[
	  A \lesssim_{p, p_1, p_2} \tau ^{(p-p_2)/2} (\tau^{\frac{p_2 - p_1}{2}} \|u_1\|_{L^{p_1} (\R) } ^{p_1}+ \|u_2\|_{L^{p_2} (\R)} ^{p_2}) \, .
		\]

	As for estimating $B$, we split slightly differently than in \eqref{B_first_split_L12}:
	\[
		B \lesssim_p  \|w^{1/p}  - 1\|_{L^p (E_{\alpha/2} ^1 \cup E_{\alpha/2} ^2)} ^p + \|w^{-1/p'}  - 1\|_{L^p (E_{\alpha/2} ^1 \cup E_{\alpha/2} ^2)} ^p  \ddd  B_1+B_2 \, . 
		\]
It suffices to show
	\begin{equation}\label{bd_B_perturbative_1}
	B_1 , B_2 \lesssim_{p, p_1, p_2} \tau ^{(p-p_2)/2} (\tau^{\frac{p_2 - p_1}{2}} \|u_1\|_{L^{p_1} (\R) } ^{p_1}+ \|u_2\|_{L^{p_2} (\R)} ^{p_2}) \, .		\end{equation}We prove this for $B_1$; the estimate for $B_2$ will follow similarly.

	By Lemma \ref{CZ_decomp_lemma_L12} \ref{smallset_contained_L12},
	\[
		B_1 \leq \sum\limits_{j} \|w^{1/p} - 1\|_{L^p (I_j)} ^p \, .
		\]
	Note that \eqref{bd_B_perturbative_1} will follow if we can show
	\begin{equation}\label{weight_perturbative_interval}
		\|w^{1/p} - 1\|_{L^p (I)} ^p \lesssim_p |I| \tau^{p/2}
	\end{equation}
      for all $I \in  \{I_j\}$: indeed, sum over $I = I_j$ and use \eqref{sum_intervals_CZ_L12} to get \eqref{bd_B_perturbative_1}. As such, fix $I \in \{I_j\}$; our goal is to show \eqref{weight_perturbative_interval}.

	Now note
	\begin{equation}\label{logw_estimate}
		|\langle \log w \rangle_I| \lesssim \tau^{1/2} \, .
	\end{equation}
	Indeed, if $K$ is the dyadic parent of $I$, then \eqref{avg_w_close_1_L12} guarantees $\langle |w-1| \rangle_I \lesssim \langle |w-1| \rangle_{K} \leq \alpha$. By Lemma \ref{avg_logw_vs_w-1_lemma}, we get \eqref{logw_estimate} for $\tau_0 (p)$ sufficiently small, which we will make use of repeatedly.

	Split the left side of \eqref{weight_perturbative_interval} into
	\[ %\begin{equation}\label{split_for_John_Nirenberg}
		\|w^{1/p} - 1\|_{L^p (I)} ^p \leq \|w^{1/p} - \e{ \langle \log w \rangle_I /p  }\|_{L^p (I)} ^p + \|\e{ \langle \log w \rangle_I  /p } - 1\|_{L^p (I)} ^p \ddd B_{1 1} + B_{1 2} \, .
		\]%\end{equation}

	%By \eqref{logw_estimate}, $B_{1 2}$ is
	By \eqref{logw_estimate} it follows that \[
	  B_{1 2} = |I| |\e{\langle \log w \rangle_I /p} -1|^p \lesssim_p |I| \tau^{p/2} \, .
			\]
		 Thus it suffices to show $B_{1 1} \lesssim_p |I| \tau^{p/2}$.

		But $B_{1 1}$ equals
		\[
			\e{\langle \log w \rangle_I /p} \|\e{ (\log w - \langle \log w \rangle_I)/p} - 1 \|_{L^p (I)} ^p \lesssim \|\e{f/p } -1\|_{L^p (I)} ^p \, ,
			\]
where $f \ddd \log w - \langle \log w \rangle_I$ and we applied \eqref{logw_estimate} in the inequality. Write
	\[
		\e{f/p} - 1 = \int\limits_{0}^{1/p} \e{s f} f \, ds \,
		\]
		and apply Minkowski's inequality, followed by the Cauchy-Schwarz inequality, to get
\begin{equation*}%\label{B11_intermediate_estimate_2}
	\|\e{f/p } -1\|_{L^p  (I)} \lesssim \|f \e{|f|/p}\|_{L^p  (I)} \leq \|f \|_{L^{2 p } (I)}   \|\e{|f|/p}\|_{L^{2 p } (I)}\, .
\end{equation*}
The John-Nirenberg Theorem \ref{John-Nirenberg_Thm} yields
\[
  \|f\|_{L^{2 p} (I)} \lesssim_{p} \|\log w\|_{\mathrm{BMO} (\R)} |I|^{\frac{1}{2p}} \, , \quad  \|\e{|f|/p}\|_{L^{2 p } (I)}\lesssim_{p } |I|^{\frac{1}{2p}}
\] for $\|\log w\|_{\mathrm{BMO} (\R)}$ small enough, which we can arrange by taking $\tau_0 (p)$ as small as necessary and applying Lemma \ref{sd_ll}. Thus
\begin{align*}
	\|\e{f/p } -1\|_{L^p  (I)} &\lesssim_{p} |I|^{1/p} \| \log w \|_{\mathrm{BMO}(\R)} \lesssim |I|^{1/p} \tau^{1/2}\, .
\end{align*}
	Hence $B_{1 1} \lesssim_{p} |I|^{1/p} \tau^{1/2}$, which completes the proof.
\end{proof}

\section{Inverting \texorpdfstring{$I-Q_{w,p}$}{I-Qwp}: proof of Lemma \ref{thm_inverse_bd_final_basic}}\label{section_I-Qwp_invertible}
Let us introduce some notation and definitions.

\beginntn Given $z = \C$, we will generally write $z = t + i y$, where $t, y$ are real-valued.

Given $K \subset \C$, we let $N(K)$ denote an open set containing $K$, although we allow for the particular set to change line to line.
\smallskip

\begindef Given some $p_* \in [1, \infty]$, define
\begin{equation}\label{p(z)_initial}
  \frac{1}{p(z)} \ddd \frac{z}{p_*}+\frac{1-z}{2}, \quad \frac{1}{p'(z)} \ddd 1-\frac{1}{p(z)}=\frac{1+z}{2}-\frac{z}{p_*}, \quad  z=t+iy \, : -1 \leq t \leq 1 \, . 
\end{equation}

	 Given $\epsilon \geq 0$, define the open strip 
	 \[
	   \Omega_{ < \epsilon} \ddd \left \{ z ~:~ \left |\frac{1}{p(t)} - \frac{1}{2} \right| < \epsilon \right \} \, .
	 \]
	 We define the closed strip $\Omega_{\leq \epsilon}$ similarly.

Given $I$ an interval, let 
	 \[
	   \Omega_{I} \ddd \{ z\in \C ~:~ t \in I\} \, .
	 \]  
	 \smallskip
\begin{Lem} \label{thm_inverse_bd_final} Let ${\bf \Lambda}>1$ be an arbitrary absolute constant, and let $p(z)$ is as in \eqref{p(z)_initial} for some $p_* \in [1, \infty]$. For $[w]_{A_2 (\R)} \leq \gamma$, consider the operator
\begin{equation*}%\label{Q_general}
  Q_{w, p(z), T} \ddd  w^{1/p(z)} T w^{-1/p(z)} - w^{-1/p'(z)} T w^{1/p' (z)} \, ,
\end{equation*}
	where $T$ is an operator satisfying \eqref{sd_01} for some $\cal{F}$ and is also self-adjoint as an operator on $L^2 (\R)$.

	Then there exists $\epsilon = \epsilon (\gamma, \cal{F}, {\bf \Lambda} ) \in (0,\frac{1}{2})$, with $\lim\limits_{\gamma \to 1} \epsilon (\gamma, \cal{F}, {\bf \Lambda}) = \frac{1}{2}$, such that on  $\Omega_{<\epsilon}$, $I-Q_{w,p(z)}$ has bounded inverse on $L^{p(t)} (\R)$ with operator bound
	\begin{equation}\label{inverse_bd_final}
	  \|(I-Q_{w, p(z), T})^{-1}\|_{p(t),p(t)} \leq 2 \bf{\Lambda} \, ,
	\end{equation}
	and $(I-Q_{w, p(z), T})^{-1}$ is analytic as a map $\Omega_{<\epsilon} \to \cal{L} (L^2 (\R))$. 
\end{Lem}
Lemma \ref{thm_inverse_bd_final_basic} follows from Lemma \ref{thm_inverse_bd_final} by taking $T = \proj{0}{r}{}$, $\mathbf{\Lambda} = 10$ and applying Lemma \ref{proj_calF_ind_r}. As such, this section is dedicated to the proof of Lemma \ref{thm_inverse_bd_final}. The proof strategy will proceed more or less as follows:
\begin{enumerate}
  \item Show $\|Q_{w, p(z), T}\|_{p(t), p(t)}$ uniformly bounded on the strip. 
  \item Show that if we have uniform bounds on $\|( I - \kappa  Q_{w, p(z), T})^{-1}\|_{p(t), p(t)}$ is uniformly bounded on the strip, then $( I - \kappa  Q_{w, p(z), T})^{-1}$ is weakly analytic. 
  \item Chop $( I - \kappa  Q_{w, p(z), T})^{-1}-I$ into small pieces and show we have uniform bounds for $I$ plus the small piece; using the two previous parts of the strategy, we have our function is weakly analytic. We maintain boundedness while adding the small pieces by applying analytic interpolation.
\end{enumerate}

Before we begin, we will need a few preliminary lemmas.
\begin{Prop}
Suppose $X$ is an Banach space and $H,V$ are linear bounded operators from $X$ to $X$. Then,
\begin{eqnarray*}
(I+H+V)^{-1}=(I+H)^{-1}-(I+H+V)^{-1}V(I+H)^{-1}, \\
(I+H+V)^{-1}=(I+H)^{-1}(I+V(I+H)^{-1})^{-1}\,,
\end{eqnarray*}
provided the operators involved are well-defined and bounded in $X$. Moreover, assuming $\|V\|\cdot \|(I+H)^{-1}\|<1$, we get
\begin{equation}\label{sd_33}
\|(I+H+V)^{-1}\|\le  \frac{\|(I+H)^{-1}\|}{1-\|V\|\cdot \|(I+H)^{-1}\|}\,.
\end{equation}
Finally, if $\|V\|<1$, then 
\begin{equation}\label{ssa1}
\|(I+V)^{-1}\|\le \frac{1}{1-\|V\|}\,.
\end{equation}
\end{Prop}
\noindent The proof of this proposition is a straightforward calculation.

We will also need continuity of weighted operators as proved in \cite[Theorem 1.2]{AlexisAD1}.

\begin{Thm}[{\cite[Theorem 1.2]{AlexisAD1}}]\label{t1} Suppose $p\in (1,\infty)$, $[w]_{A_p(\R^d)}<\infty$, $\|f\|_{\rm BMO}<\infty$, and $T$ satisfies \eqref{sd_01}. Consider $w_\delta =we^{\delta f}$. Then, there is $\delta_0(p,[w]_{A_p},\|f\|_{\rm BMO})>0$ such that 
\[
\|w_\delta^{1/p}Tw_\delta^{-1/p}-w^{1/p}Tw^{-1/p}\|_{p,p}<|\delta|C(p,[w]_{A_p},\|f\|_{\rm BMO},\cal{F})
\]
for all $\delta: |\delta|<\delta_0$.
\end{Thm} 

\begindef Suppose $\cal{D}$ is a set of linear functionals acting on a vector space $V$. If $U \subset \C$ is an open set, and $\{a(z)\}_{z \in U}$ is contained within $V$, then $a(z)$ is \textit{weakly analytic} with respect to $\cal{D}$ if $\ell (a (z))$ is analytic for all $\ell \in \cal{D}$.
\smallskip

We will generally be concerned with two cases:
\begin{enumerate}
  \item $V = \bigcup\limits_{p \geq  1} L^p (\R)$ for $p \in (1, \infty)$, and $\cal{D}= \dense$, the space of simple functions with compact support, where $g \in \dense$ acts on $f \in \bigcup\limits_{p \geq 1} L^p (\R)$ by $\langle f , g \rangle$.
  \item $V= \bigcup\limits_{p \geq 1} \cal{L}(L^p (\R))$ and $\cal{D} = \{\langle \cdot f , g \rangle \}_{f, g \, \in \dense}$.
\end{enumerate}

\begin{Lem}\label{Q_weak_analytic_bds}
  Let $T$ be an operator satisfying \eqref{sd_01} that is self-adjoint on $L^2 (\R)$. If $[w]_{A_2 (\R)} \leq \gamma$, then there exists $\epsilon_0 (\gamma) \in (0, \frac{1}{2})$, with $\lim\limits_{\gamma \to 1} \epsilon_0 (\gamma) = \frac{1}{2}$, such that for all $z = t+i y \in N (\Omega_{\leq \epsilon_0 (\gamma)})$, we have
     \begin{equation} \label{bd_muckenhoupt_2}
       [w]_{A_{p(t)} (\R)} , [w]_{A_{p' (t)} (\R)} \lesssim_{\gamma} 1 , \, 
\end{equation}
and
      \begin{equation}\label{bd_Q}
	\|Q_{w, p(z), T}\|_{p,p} \leq C_{\gamma , \cal{F}} \,  
\end{equation}
and $Q_{w, p(z), T}$ is weakly analytic with respect to $\{ \langle \cdot f, g \rangle_{f, g \in \dense}\}$ on $N (\Omega_{\leq \epsilon_0 (\gamma)})$.
\end{Lem}
\begin{proof}
  Initially define $\epsilon_0 (\gamma) \ddd \frac{1}{s(2, \gamma)}- \frac{1}{2}$ where $s$ is as in Lemma \ref{Reverse_Holder_A2}. Then $\epsilon_0 (\gamma) \in (0, \frac{1}{2})$, $\lim\limits_{\gamma \to 1} \epsilon_0 (\gamma) = \frac{1}{2}$, and \eqref{bd_muckenhoupt_2} holds for $z \in \Omega_{\leq \epsilon_0}$; by taking a slightly smaller $\epsilon_0(\gamma) \in (0, \frac{1}{2})$, still with $\lim\limits_{\gamma \to 1} \epsilon_0 (\gamma) = \frac{1}{2}$, then \eqref{bd_muckenhoupt_2} holds for $z \in N (\Omega_{\leq \epsilon_0 (\gamma)})$. 

  If $T$ satisfies \eqref{sd_01}, then 
\[
	\|w^{1/p} T w^{-1/p}\|_{p,p} \leq C_{\gamma , \cal{F}, p} \, ,
\]
for all $p$ satisfying $\left |\frac{1}{2} - \frac{1}{p} \right | \leq \epsilon_0 (\gamma)$. Since $\|w^{1/p} T w^{-1/p}\|_{p,p} = \|T\|_{L^p _w (\R), L^p _w (\R)}$, we can interpolate between the extremal values of $p$ given by $\left |\frac{1}{2} - \frac{1}{p} \right | = \epsilon_0 (\gamma)$ to in fact get 
\begin{equation*}%\label{bd_T_initial}
\|w^{1/p} T w^{-1/p}\|_{p,p} \leq C_{\gamma , \cal{F}} \, 
\end{equation*}
for all $p$ satisfying $\left |\frac{1}{2} - \frac{1}{p} \right | \leq \epsilon_0 (\gamma)$. In particular, 
\begin{equation*}%\label{bd_T_initial}
  \|w^{1/p(z)} T w^{-1/p(z)}\|_{p(t),p(t)} = \|w^{1/p(t)} T w^{-1/p(t)}\|_{p(t),p(t)} \leq C_{\gamma , \cal{F}} \, 
\end{equation*}
for all $z \in \Omega_{\leq \epsilon_0 (\gamma)}$. If we swap the role of $p_*$ with $p_* '$, this has the net effect of swapping the role $p(z)$ with $p'(z)$. Thus if $T$ is self-adjoint on $L^2 (\R)$, then swapping as described and taking adjoints yields 
\[
  \|w^{-1/p'(z)} T w^{1/p'(z)}\|_{p(t),p(t)}  = \|w^{1/p'(z)} T w^{-1/p'(z)}\|_{p'(t), p'(t)} \leq C_{\gamma , \cal{F}} \, 
	\]
	for $z \in \Omega_{ \leq \epsilon_0 (\gamma)}$. Combining all of this together yields \eqref{bd_Q} for $z \in \Omega_{\epsilon_0 (\gamma)}$; by slightly shrinking $\epsilon_0(\gamma) \in (0, 1/2)$, while still requiring $\lim\limits_{\gamma \to 1} \epsilon_0 (\gamma) = \frac{1}{2}$, we get \eqref{bd_Q} for $z \in N (\Omega_{\leq \epsilon_0 (\gamma)})$.

	By Proposition \ref{T_weakly_analytic_prop}, if we again shrinking $\epsilon_0(\gamma) \in (0, \frac{1}{2})$ we may assume without loss of generality that $Q_{w, p(z), T}$ is analytic as a map $N (\Omega_{\leq \epsilon_0 (\gamma)}) \to \cal{L} (L^2 (\R))$, and hence weakly analytic with respect to $\{ \langle \cdot f, g \rangle_{f, g \, \in \dense}\}$ on $N (\Omega_{\leq \epsilon_0 (\gamma)})$.
\end{proof}

\begin{Lem}\label{Inverse_analytic_lem}
  Let $[w]_{ A_2 (\R)} \leq \gamma$, $\epsilon_0 (\gamma)$ is as Lemma \ref{Q_weak_analytic_bds} and let $p(z)$ as in \eqref{p(z)_initial} be given for some $p_* \in [1, \infty]$. If
\[
  \sup\limits_{z \in \Omega_{\leq \epsilon_0 (\gamma)}} \|(I - \kappa Q_{w, p(z),T})^{-1}\|_{p(t),p(t)} < \infty \, , 
\]
then  $(I - \kappa Q_{w, p(z),T})^{-1}$ is weakly analytic with respect to $\{\langle \cdot f, g \rangle\}_{f, g \in \dense }$ on $N(\Omega_{\leq \epsilon_0 (\gamma)})$.
\end{Lem}
\begin{proof}
  Let $F(z) \ddd (I - \kappa Q_{w, p(z), T})^{-1}$. We show that for each $z_0 \in \Omega_{\leq \epsilon_0 (\gamma)}$, there exists a ball $B_{\delta} (z_0)$ on which $F(z)$ is analytic as a map $B_{\delta} (z_0) \to \cal{L} (L^{p(t_0)} (\R))$; this will clearly imply weak analyticity on $N (\Omega_{\leq \epsilon_0 (\gamma)})$. Fix $z_0$, define 
\[
  V(z,z_0) \ddd \kappa (Q_{w, p(z), T} - Q_{w, p(z_0), T})
\]and write 
\begin{multline}\label{borne_series_init}
F(z) = F(z_0) + (F(z_0)^{-1} - V(z,z_0))^{-1} - F(z_0) = F(z_0) +  ((I - F(z_0) V(z,z_0))^{-1} - I) F(z_0) \\
= (I - F(z_0) V(z,z_0))^{-1} F(z_0)\, .
\end{multline}
We claim that there exists a ball $B_{\delta} (z_0)$ on which $Q_{w,p(z),T}$, and hence $V(z,z_0)$, is analytic as a map $B_{\delta} (z_0) \to \cal{L} (L^{p(t_0)} (\R))$. Given the claim, we can then take $\delta$ small enough so that $\|V(z,z_{0})\|_{p(t_0),p(t_0)} < \frac{1}{2 \|F(z_0)\|_{p(t_0), p(t_0)}}$. Then \eqref{borne_series_init} yields
  \[
    F(z) = (I - F(z_0) V(z,z_0))^{-1} F(z_0) = \sum\limits_{k=0}^{\infty} (F(z_0) V(z,z_0) )^k F(z_0) \, , 
  \]
  where the sum converges uniformly in $\cal{L} (L^{p(t_0) (\R)})$. Since each partial sum is analytic in $z$ and the sum converges uniformly, we have $F(z)$ is analytic as a map $B_{\delta} (z_0) \to \cal{L} (L^{p(t_0)} (\R))$.

  We must now show there exists a ball $B_{\delta} (z_0)$ on which $Q_{w,p(z),T}$ is analytic as a map $B_{\delta} (z_0) \to \cal{L} (L^{p(t_0)} (\R))$. As per Lemma \ref{Q_weak_analytic_bds}, $Q_{w, p(z), T}$ is weakly analytic with respect to $\{ \langle \cdot f, g \rangle\}_{f, g \in \dense}$ on $N (\Omega_{\leq \epsilon_0 (\gamma)})$, and so by Proposition \ref{weak_to_strong_analytic} it suffices to show there exists $\delta> 0$ such that for all $z \in B_{\delta} (z_0)$, we have $w^{1/p(z)} T w^{-1/p(z)}, w^{-1/p'(z)} T w^{1/p(z)} \in \cal{L} (L^{p(t_0)} (\R))$. We show this for $w^{1/p(z)} T w^{-1/p(z)}$, as the proof for $w^{-1/p'(z)} T w^{1/p'(z)}$ will follow similarly. Write
  \begin{multline*}
    \|w^{1/p(z)} T w^{-1/p(z)}\|_{p(t_0), p(t_0)} = \|w^{1/p(t)} T w^{-1/p(t)}\|_{p(t_0), p(t_0)} \\
    \leq \|w^{1/p(t)} T w^{-1/p(t)}- w^{1/p(t_0)} T w^{-1/p(t_0)}\|_{p(t_0), p(t_0)} + \|w^{1/p(t_0)} T w^{-1/p(t_0)}\|_{p(t_0), p(t_0)} \, .
  \end{multline*}
  The last term is clearly finite. As for $\|w^{1/p(t)} T w^{-1/p(t)}- w^{1/p(t_0)} T w^{-1/p(t_0)}\|_{p(t_0), p(t_0)}$, by \eqref{bd_muckenhoupt_2} we have $w \in A_{p(t)}$. Thus we can apply the continuity of weighted operators Theorem \ref{t1} with $f = \log w$, to argue that for $|t-t_0|$ sufficiently small, we have $\|w^{1/p(t)} T w^{-1/p(t)}- w^{1/p(t_0)} T w^{-1/p(t_0)}\|_{p(t_0), p(t_0)} \leq 1$. This completes the proof.
\end{proof}

\begin{Prop}[\cite{AlexisAD1}]\label{ss1_extra} Let $[w]_{ A_2 (\R)} \leq \gamma$ and let $p(z)$ as in \eqref{p(z)_initial} be given for some $p_* \in [1, \infty]$ such that $\Omega_{[-1, 1]} \subset \Omega_{\leq \epsilon_0 (\gamma)}$, where $\epsilon_0 (\gamma)$ is as in Lemma \ref{Q_weak_analytic_bds}. Suppose 
\begin{equation*}
  \sup_{-1 \le \Re z\le 1}\|Q_{w,p(z),T}\|_{p(t),p(t)}<\infty\,,
 \end{equation*}
 where $t\ddd \Re z$. If there is a positive number ${\bf \Lambda} > 1$ such that
\[
  \|(I - \kappa Q_{w, p(z),T})^{-1}\|_{p(t),p(t)}\le 2{\bf \Lambda}
\]
for all $z \in \Omega_{[-1,1]}$, then there is a $t_*({\bf \Lambda})\in (0,1]$,  so that
\[
  \|(I - \kappa Q_{w, p(z),T})^{-1}\|_{p(t),p(t)}\le {\bf \Lambda}
\]
for all $z \in \Omega_{[-t_* , t_*]}$.
\end{Prop}

\begin{proof}
  We notice that $Q_{w,p(iy),T}$ is bounded and antisymmetric operator in Hilbert space $L^2(\R)$. Therefore, $\|(I - \kappa Q_{w, p(iy),T})^{-1}\|_{2,2}\le 1$. By Lemma \ref{Inverse_analytic_lem}, $(I - \kappa Q_{w, p(z),T})^{-1}$ is weakly analytic and continuous in the sense of Stein (p.209, \cite{bsh}). Applying Stein's interpolation theorem on the strips $\Omega_{[0,1]}$ and $\Omega_{[-1,0]}$, we get
\[
  \|(I - \kappa Q_{w, p(z),T})^{-1}\|_{p(t),p(t)}\le \exp\left(
  \frac{\sin(\pi |t|)}{2}\int_\R \frac{\log (2{\bf \Lambda)}}{\cosh (\pi y)+\cos(\pi |t|)}dy\right)=1+O(|t|), \quad t\to 0\, .
\]

\end{proof}
\noindent{\bf Remark.} We emphasize here that positive $t_*$ does not depend on $w$.

\begin{proof}[Proof of Lemma \ref{thm_inverse_bd_final}]
  The reader may also consult a similar, but simpler proof in \cite[Proof of Lemma 3.6]{AlexisAD1}.

  We just need to find $\epsilon(\gamma,\cal{F}, \mathbf{\Lambda})$ such that on $\Omega_{< \epsilon(\gamma,\cal{F}, \mathbf{\Lambda})}$, we have $(I-Q_{w,p(z),T})^{-1}$ is bounded on $L^{p(t)} (\R)$ and analytic as an $\cal{L} (L^2 (\R))$-valued function. 

  In \eqref{p(z)_initial}, we take parameter $p_*$ as follows: define $p_* ^{(1)}$ by $\frac{1}{p_* ^{(1)}} = \frac{1}{2} - \epsilon_0 (\gamma)$, where $\epsilon_0 (\gamma)$ is as in Lemma \ref{Q_weak_analytic_bds}, and set $p_1(z)\ddd p(z)$; note then $\Omega_{[-1,1]} \subset \Omega_{\leq \epsilon_0 (\gamma)}$. Consider $Q^{(j)}_{w,p(z),T}\ddd jQ_{w,p(z),T}/N$, $j=1,\ldots,N$, where $N$ is large and will be fixed later (it will depend on $\gamma$, $\cal{F}$, $\bf{\Lambda}$ only).

We take $N$ to satisfy
\begin{equation}\label{sd_44_extra}
	1-C_{\gamma, \cal{F}} {\bf \Lambda}/N>1/2\, ,
\end{equation}
	where $C_{\gamma, \cal{F}}$ is as in \eqref{bd_Q}. 
	Next, by \eqref{ssa1} and \eqref{bd_Q}, \[
	  \|(I- Q^{(1)}_{w, p(t+iy),T})^{-1} \|_{p(t),p(t)}\le \frac{1}{1-C_{\gamma, \mathcal{F}} /N}\le \frac{1}{1-C_{\gamma, \mathcal{F}} {\bf \Lambda}/N}\le 2\le  2{\bf \Lambda} \, ,
\]
	since ${\bf \Lambda}>1$.

	We continue with an inductive argument in which the bound for $\{Q^{(j)}_{w,p(z)}\}$ provides the bound for $\{Q^{(j+1)}_{w,p(z),T}\}$ when  $j=1,\ldots,N-1$. \smallskip

	{$\bullet$ \bf Base of induction: handling  $Q^{(1)}_{w,p(z)}$.} Apply Proposition \ref{ss1_extra} with $\kappa=1/N$ to get an absolute constant $t_*$ so that
\[
	\|(I- Q^{(1)}_{w, p(t+iy),T})^{-1} \|_{p(t),p(t)}\le {\bf \Lambda}
\]
for $t\in [- t_*, t_*]$ and $y\in \R$. Next, we use \eqref{sd_33} with $H=-Q^{(1)}_{w,p(t+iy)}$ and $V=-N^{-1}Q_{w,p(t+iy)}$. This gives
\begin{equation}\label{sa4_extra}
  \|(I- Q^{(2)}_{w, p(t+iy),T})^{-1} \|_{p(t),p(t)}\le \frac{{\bf \Lambda}}{1-C_{\gamma, \mathcal{F}} {\bf \Lambda}/N}\le 2{\bf \Lambda}, \quad t\in [- t_*, t_*]
\end{equation}
by \eqref{sd_44_extra}.

That finishes the first step. Next, we will explain how estimates on $Q^{(2)}_{w,p(z)}$ give bounds for $Q^{(3)}_{w,p(z)}$.

{$\bullet$ \bf Handling  $Q^{(2)}_{w,p(z),T}$.} In Proposition \ref{ss1_extra}, we now take  $\kappa=\kappa_2\ddd 2/N, p^{(2)}_*\ddd p_1(t_*)=p(t_*)$ (here $p(t_*)$ is obtained at the previous step) and compute new $p_2(z),p_2'(z)$ by \eqref{p(z)_initial}:
\begin{equation}\label{sad66_extra}
  \frac{1}{p_2(z)}\ddd \frac{z}{p ^{(2)} _{*} }+\frac{1-z}{2}=\frac{zt_*}{p_*}+\frac{1-zt_*}{2}=\frac{1}{p_1(zt_*)}=\frac{1}{p(zt_*)}\,.
\end{equation}
Therefore, when $z$ belongs to $-1 \leq \Re z \leq 1$, $zt^*$ belongs to $-t_* \leq \Re z \leq t_*$, $p_2(z)=p(zt_*)$ and $p_2 (z)$ still satisfies $\Omega_{[-1,1]} \subset \Omega_{\leq \epsilon_0 (\gamma)}$. In this domain, the estimate \eqref{sa4_extra} can be rewritten as
\[
   \|(I- Q^{(2)}_{w,p_2(t+iy),T})^{-1}\|_{p_2(t),p_2(t)}\le 2{\bf \Lambda}, \quad t\in [-1,1],\quad y\in \R\,,
\]
	where $p_2(z)$ is different from $p_1(z)=p(z)$ only by the choice of parameter  $p_*$ in \eqref{p(z)_initial} and is in fact a rescaling of the original $p(z)$ as follows from  \eqref{sad66_extra}. From Proposition \ref{ss1_extra}, we have
\[
	\|(I- Q^{(2)}_{w,p_2(t+iy),T})^{-1}\|_{p_2(t),p_2(t)}\le {\bf \Lambda}
\]
for $t\in [- t_*, t_*],y\in \R$. We use the perturbative bound \eqref{sd_33} one more time with $H=-Q^{(2)}_{w,p_2(t+iy),T}$ and $V=-N^{-1}Q_{w,p_2(t+iy),T}$ to get
\[
   \|(I- Q^{(3)}_{w,p_2(t+iy),T})^{-1}\|_{p_2(t),p_2(t)}\le 2{\bf \Lambda}
\]
for $t\in [- t_*, t_*],y\in \R$.

	{$\bullet$ \bf Induction in $j$ and  the bound for $Q^{(N)}_{w,p(z),T}$.} Next, we take $p_*^{(3)}\ddd p^{(2)}(t_*)$ and repeat the process in which the bound
\[
   \|(I- Q^{(j)}_{w,p_j(t+iy),T})^{-1}\|_{p_j(t),p_j(t)}\le 2{\bf \Lambda}, \quad z=t+iy\in \Omega_{[-1,1]} \, ,
\]
	implies
\[
  \|(I-Q^{(j+1)}_{w,p_{j+1}(t+iy),T})^{-1}\|_{p_{j+1}(t),p_{j+1}(t)}\le 2{\bf \Lambda} , \quad z=t+iy\in \Omega_{[-1,1]} \, .
\]
	  Notice that each time the new $p_j(z)$ is in fact a rescaling of the original $p(z)$ by $t_*^{j-1}$ as can be seen from a calculation analogous to  \eqref{sad66_extra}. In $N-1$ steps, we get
\begin{align*}
  \|(I-Q^{(N)}_{w,p_{N-1}(t+iy),T})^{-1}\|_{p_{N-1}(t),p_{N-1}(t)}\le  2{\bf \Lambda}\, , \quad z = t + i y \in \Omega_{[-t_*,t_*]}\, .
\end{align*}
Thus recalling that $p_{N-1}(z)=p(t_*^{N-2}z)$, one has
\[
  \|(I-Q^{(N)}_{w,p(t_*^{N-1} z),T})^{-1}\|_{p(t_*^{N-1} t),p(t_*^{N-1} t)}\le  2{\bf \Lambda}\,.
\]
Since $Q^{(N)}_{w,p(t_*^N z )}=Q_{w,p(t_*^N z ),T}$, we get \eqref{inverse_bd_final} with
\[
  \epsilon (\gamma, \mathbf{\Lambda}, \cal{F}) \ddd t_* ^{N-1} \epsilon_0 (\gamma).
\]
The estimates \eqref{sd_44_extra} implies that we can take
$
N\sim C_{\gamma, \cal{F}, \mathbf{\Lambda}}\,.
$

Thus, we showed there exists $\epsilon(\gamma, \cal{F}, \mathbf{\Lambda}) > 0$ such that $\|(I-Q_{w,p,T})^{-1}\|_{p,p} \leq 2 \mathbf{\Lambda}$ for all $|\frac{1}{p} - \frac{1}{2}| < \epsilon (\gamma, \cal{F}, \mathbf{\Lambda})$.  

In fact, we may assume without loss of generality that $\lim\limits_{\gamma \to 1} \epsilon(\gamma, \cal{F}, \mathbf{\Lambda}) = \frac{1}{2}$: indeed, it suffices to show that for any $\widetilde{\epsilon} \in (0,\frac{1}{2})$, we can choose $\gamma-1$ sufficiently small so that $\|(I-Q_{w,p(z),T})^{-1}\|_{p(t),p(t)} \leq 2 \mathbf{\Lambda}$ for all $z$ satisfying $|\frac{1}{p(t)} - \frac{1}{2}| < \widetilde{\epsilon}$. Note that if $\|Q_{w,p(z), T}\|_{p(t),p(t)} \leq \frac{1}{2}$, then by geometric sum
\[
  \|(I-Q_{w,p(z),T})^{-1}\|_{p(t),p(t)} \leq \sum\limits_{k=0}^{\infty} \|Q_{w,p(z)}\|_{p(t),p(t)} ^k \leq 2 \leq 2 \mathbf{\Lambda} \, .
	\]
Thus it suffices to show we can take $\gamma$ sufficiently small so that \[
  \|Q_{w,p(z), T}\|_{p(t),p(t)} \leq \frac{1}{2} \, 
	\]
	for $|\frac{1}{p(t)} - \frac{1}{2}| < \widetilde{\epsilon}$. In fact, by Stein's analytic interpolation it suffices to check this at the vertical lines such that $|\frac{1}{p(t)} - \frac{1}{2}| = \widetilde{\epsilon}$. By duality and the triangle inequality this in turn will follow from showing
\[
  \|w^{1/p(z)}T w^{-1/p(z)} -T \|_{p(t),p(t)} \leq \frac{1}{4} \, 
	\]
	for $z= t+iy$ such that $|\frac{1}{p(t)} - \frac{1}{2}| = \epsilon(\gamma)$. Fix one the values $t$ satisfying the last equality: if $\gamma -1$ is sufficiently small, then by Lemma \ref{sd_ll} we have $\|\log w\|_{\rm BMO} \lesssim \sqrt{\gamma -1}$, which can be made arbitrarily small. Apply Theorem \ref{t1} with $f= \log w$ to get 
	\[
		\|w^{1/p} T w^{-1/p} - T\|_{p,p} \lesssim_{p, \cal{F}} \sqrt{\gamma -1} \, .
		\]
		We can now choose $\gamma$ sufficiently small that this is at most $\frac{1}{4}$. Thus, we can choose $\epsilon(\gamma, \cal{F}, \mathbf{\Lambda})$ so that as $\gamma \to 1$, we get $\epsilon(\gamma, \cal{F}, \mathbf{\Lambda}) \to \frac{1}{2}$.

As for showing $(I-Q_{w,p (z),T})^{-1}$ is analytic as an $\cal{L} (L^2 (\R))$-valued function, by Lemma \ref{Inverse_analytic_lem} and Proposition \ref{weak_to_strong_analytic}, it suffices to show $(I-Q_{w,p(z),T})^{-1}$ is bounded on $L^2 (\R)$ for $z \in \Omega_{\leq \epsilon(\gamma, \cal{F}, \mathbf{\Lambda})}$. Fix one such $z$; we just showed previously that
\[
  \|(I-Q_{w,p(z)})^{-1}\|_{p(t), p(t)} \leq 2 \mathbf{\Lambda} \, .
	\]
	However the proof just as easily applies to weight $w^{-1}$, operator $-T$ and H\"older index $p'(z)$ (as opposed to $p(z)$) and so we get
	\[
	  \|(I-Q_{w,p(z), T})^{-1}\|_{p'(t), p'(t)} = \|(I - Q_{w^{-1}, p'(z), -T})^{-1} \|_{p'(t),p'(t)} \leq 2 \mathbf{\Lambda} \,  \quad \text{for } z\in \Omega_{< \epsilon(\gamma, \cal{F}, \mathbf{\Lambda})}.
	\]
	Interpolate between both estimates to get
\[
  \|(I - Q_{w, p(z), T})^{-1}\|_{2,2} \leq 2 \mathbf{\Lambda} \, .
	\]
This completes the proof.
\end{proof}

\appendix

\section{A detour through complex analysis: proof of Proposition \ref{solve_for_X_Steklov_fullp_statement}}\label{complex_analysis_section}

We recall the following well-known lemma, which we prove for the sake of completeness.
\begin{Prop} \label{weak_to_strong_analytic}
  Let $B$ be a Banach space, let $\cal{D} \subset B^*$ be a dense subset of the dual space and $U \subset \C$ an open set. If $a:U \to B$ is weakly analytic with respect to $\cal{D}$, then $a:U \to B$ is analytic.
\end{Prop}
\begin{proof}
	We adapt the proof of \cite[Theorem 8.20]{einsiedler2017functional}.

	Let $\ell \in \cal{D}$ so that $\ell \circ a(z)$ is analytic.

	Fix $\zeta \in U$ and let $|h|< \epsilon$ for $\epsilon$ sufficiently small. By the Cauchy integral formula,
	\[
		\ell \circ a (\zeta+h) = \frac{1}{2 \pi i} \oint\limits_{|z - \zeta| = \epsilon} \frac{\ell \circ a(z)}{z- (\zeta + h)} \, dz \, .
		\]

	Then \begin{equation} \label{eq:1}
		\ell \circ a (\zeta + h) - \ell \circ a (\zeta) = \frac{h}{2 \pi i} \oint\limits_{|z - \zeta| = \epsilon} \frac{ \ell \circ a(z)}{(z - (\zeta + h))(z - \zeta)} \, dz \, .
	\end{equation}

	For $h \neq h'$ with $0 < |h|, |h'| < \epsilon$, define the second order difference quotient
	\[
		x(h,h') \ddd \frac{1}{h - h'} \left ( \frac{a(\zeta+ h ) - a(\zeta)}{h} - \frac{a(\zeta + h') - a(\zeta)}{h'} \right ) \, .
		\]
	It suffices to show $x(h,h')$ is uniformly bounded in $h,h'$, for both sufficiently small. Indeed, uniform boundedness in $B$ would then yield
	\[
		h \mapsto \frac{a(\zeta + h) - a(\zeta)}{h}
		\]
		is Lipschitz for $h$ near $0$, and so has a limit as $h \to 0$.

	By \eqref{eq:1},
	\begin{align*}
		\ell (x(h,h')) 
		&= \frac{1}{2\pi i} \oint\limits_{|z- \zeta| = \epsilon} \frac{\ell \circ a(z)}{(z - \zeta)(z - (\zeta + h))(z - (\zeta + h'))} \, dz \, .
	\end{align*}
	For $|h|, |h'| < \epsilon/2$, the denominator in the integral is bounded uniformly away from $0$, and the numerator is bounded above by $\|\ell\|_{B^*} \sup\limits_{|z - \zeta| = \epsilon}  \|a(z)\|_{B}$, which is finite by the uniform boundedness principle. Whence
	\[
		\ell (x(h,h')) \leq M \|\ell\|_{B^*} 
	\]
	where $M$ does not depend on $\ell$. By duality
	\[
		\|x(h,h')\|_{B} \leq M \, .
	\]
\end{proof}

	 Fix $\chi$ a characteristic function of some finite interval $I$, and let $X_p \ddd w^{1/p} (P(r, \lambda; w) - \e{i \lambda r})$, like in Section \ref{weight_A2_section}.
 \begin{Prop}\label{weak_analytic_weight_function}
   Suppose $[w]_{A_2 (\R)} \leq \gamma$, $p_* \in [1, \infty]$ and  $p(z)$ is as in \eqref{p(z)_initial}. Then $z \mapsto \chi w^{1/p(z) - 1/2}$ is analytic as a map \[
     \Omega_{(-1,1)} \to L^2 (\R) \, .
			\]
 \end{Prop}
\begin{proof}
  Without loss of generality, assume $p_* = \infty$; the general case follows by a rescaling argument.

	We first compute
	\[
	  \|\chi w^{1/p(z) - 1/2}\|_{L^{q} (\R)} ^q = \int\limits_{I} w^{-qt/2} \, d \lambda \, . 
	\] Since $[w]_{A_2 (\R)} , \, [w^{-1}]_{A_2 (\R)} \leq \gamma$, then by Lemma \ref{Reverse_Holder_A2} there exists $q (\gamma) > 2$ such that for all $|t| < 1$, $w^{-qt/2} \in A_2 (\R)$. Since $A_2 (\R) \subset L^1 _{loc} (\R)$, then $\int\limits_{I} w^{-qt/2} \, d \lambda < \infty$ and so $\chi w^{1/p(z) - 1/2} \in L^q (\R)$ for some $q > 2$ by \eqref{p(z)_initial}. 

	Now let us show analyticity: write 
	\begin{align*}
		\chi w^{1/p(z) - 1/2} &= (\chi w^{1/p(z) - 1/p(z_0)}) (\chi w^{1/p(z_0) - 1/2}) \\
		&= \sum\limits_{k=0}^{\infty} \frac{(-1)^k}{k! \, 2^k} (\chi \log w)^k (z-z_0)^k  (\chi w^{1/p(z_0) - 1/2})\, 
	\end{align*}
	by Taylor expansion of $ w^{1/p(z) - 1/p(z_0)} = \e{-\frac{(z-z_0)}{2} \log w }$.
	We simply need to check the series converges uniformly in $L^2 (\R)$ norm for $|z-z_0|$ sufficiently small; by Stirling's approximation, it suffices to show 
	\[
		\|(\chi \log w)^k (\chi w^{1/p(z_0) - 1/2})\|_{L^2 (\R)} \lesssim_{\log w, I, \gamma } C(\log w, I, \gamma) ^k k^k \, . 
		\]
		 
		Apply H\"older's inequality to bound
		\[
			\|(\chi \log w)^k (\chi w^{1/p(z_0) - 1/2})\|_{L^2 (\R)} \lesssim \| \chi w^{1/p(z_0) - 1/2}\|_{L^q (\R)} \|(\chi \log w)^k\|_{L^{2 (q/2)'} (\R)} ^{1/2 - 1/q} \, ,
			\] where $(q/2)'$ is the dual exponent to $q/2$.

			Since we already showed $\chi w^{1/p(z) - 1/2} \in L^q (\R)$, then we will be done once we show $\| (\chi \log w)^k\|_{L^{2 (q/2)'} (\R)} \lesssim_{\log w, I, \gamma} C(\log w , I, \gamma)^k k^k$. If $p = 2 (q/2)'$, write
			\[
				\| (\chi \log w)^k\|_{L^{p} (\R)} = \| (\log w)^k\|_{L^{p} (I)} = \|\log w\|_{L^{p k} (I)} ^{k} \, .
				\]
				It suffices to show $\|\log w\|_{L^{p k} (I)} \lesssim_{\log w, I, \gamma} k$.

				Write 
				\[
				  \|\log w\|_{L^{p k} (I)} \leq  \|\log w - \langle \log w \rangle_{I} \|_{L^{p k} (I)} + \langle \log w \rangle_{I} |I|^{1/(pk)} 
				\]
				The last term is $\lesssim_{\log w , I, p} k$. For the first term, the John-Nirenberg Theorem \ref{John-Nirenberg_Thm} yields
				\[
				  \|\log w - \langle \log w \rangle_{I} \|_{L^{p k} (I)} \leq (c |I|   p k  \|\log w \|_{\rm BMO})^{1/(pk)} \Gamma(p k) ^{1/(pk)} \lesssim_{\log w, I, \gamma} \Gamma (p k) ^{1/(pk)}
					\]
					where $\Gamma$ function is the usual analytic extension of factorial. Finally, Stirling's formula yields
					\[
						\|\log w - \langle \log w \rangle_{I} \|_{L^{p k} (I)} \lesssim_{\log w, I, \gamma} \Gamma (p k) ^{1/pk} \lesssim pk \lesssim_p k \, ,						\] thereby completing the proof.
\end{proof}

\beginrmk To avoid repeating similar arguments, from now onwards if we write that a function or operator is analytic (or weakly analytic) thanks to a ``John-Nirenberg argument,'' the reader should understand this as meaning that analyticity follows from an argument similar to the one above where we showed $\chi w^{1/p(z) - 1/2}$ was analytic as a map $\{|\Re z| < 1\} \to L^2 (\R)$. \smallskip

\begin{Lem}\label{integrability_Ap_weights_analytic}
  Suppose $[w]_{A_2 (\R)} \leq \gamma$, $p_* \in [1, \infty]$ and  $p(z)$ is as in \eqref{p(z)_initial}. Then there exists $\epsilon(\gamma) \in (0, \frac{1}{2})$, with $\lim\limits_{\gamma \to 1} \epsilon(\gamma) = \frac{1}{2}$, such that $z \mapsto w^{1/p(z)} - w^{-1/p'(z)}$ is analytic as a map $\Omega_{ < \epsilon (\gamma)} \to L^2 (\R)$.
\end{Lem}
\begin{proof}
  We first note that $w^{1/p(z)} - w^{-1/p'(z)} \in L^2 (\R)$ for $z \in \Omega_{<\epsilon}$ for some $\epsilon > 0$. Indeed,
	\[
		\left | w^{1/p(z)} - w^{-1/p'(z)} \right | = \left |w^{-1/p'(z)} \right | \left | w-1 \right | = w^{-1/p(t)} \left |w-1 \right | = \left |w^{1/p(t)} - w^{-1/p'(t)} \right | \, .
		\]
		By Lemma \ref{lemma_mu_weight_Lp}, this belongs to $L^2 (\R)$ so long as $z \in \Omega_{< \epsilon (\gamma)}$, where $\epsilon(\gamma) \in (0,\frac{1}{2})$ with $\lim\limits_{\gamma \to 1} \epsilon(\gamma) = \frac{1}{2}$.

		By a John-Nirenberg argument, $w^{1/p(z)}$ and $w^{-1/p'(z)}$ are weakly analytic with respect to $\dense$ on $\Omega_{< \epsilon (\gamma)}$. By Proposition \ref{weak_to_strong_analytic}, this means $w^{1/p(z)} - w^{-1/p'(z)}$ is analytic as a map $\Omega_{< \epsilon (\gamma)} \to L^2 (\R)$.
\end{proof}

\begin{Prop}\label{T_weakly_analytic_prop}
  Suppose $[w]_{A_2 (\R)} \leq \gamma$, $T$ is an operator satisfying \eqref{sd_01} for some $\cal{F}$, and $p(z)$ is as in \eqref{p(z)_initial} for some $p_* \in [1, \infty]$. Then there exists $\epsilon(\gamma, \cal{F}) \in (0, \frac{1}{2}]$, with $\lim\limits_{\gamma \to 1} \epsilon(\gamma, \cal{F}) = \frac{1}{2}$, such that $w^{1/p(z)} T w^{-1/p(z)}$ and $w^{-1/p'(z)} T w^{1/p'(z)}$ are analytic as maps $\Omega_{< \epsilon (\gamma, \cal{F})} \to \cal{L}(L^2 (\R))$.
\end{Prop}
\begin{proof}
  We will prove the portion of the proposition involving $w^{1/p(z)} T w^{-1/p(z)}$; the result for $w^{-1/p'(z)} T w^{1/p'(z)}$ follows by then replacing $w$ by $w^{-1}$ and $p_*$ by $p_* '$.

  A John-Nirenberg argument shows that $w^{1/p(z)} T w^{-1/p(z)}$ is weakly analytic with respect to $\{\langle \cdot f, g \rangle \}_{f,g \, \in \dense}$ for all $z= t+iy$ where $w \in A_{p(t)}$. By Proposition \ref{weak_to_strong_analytic}, $w^{1/p(z)} T w^{-1/p(z)}$ is then analytic as a map $\Omega_{<\epsilon} \to \cal{L}(L^2 (\R))$ if we can show $w^{1/p(z)} T w^{-1/p(z)} \in \cal{L} (L^2 (\R))$ on $\Omega_{< \epsilon}$ for some $\epsilon > 0$. To see this latter part, notice that 
	\[
		\|w^{1/p(z)} T w^{-1/p(z)}\|_{2,2} =  \|w^{1/p(t)} T w^{-1/p(t)}\|_{2,2} = \|(w^{2/p(t)})^{1/2} T (w^{2/p(t)})^{-1/2}\|_{2,2} \, .
	      \] By Lemma \ref{Reverse_Holder_A2}, there exists $\epsilon(\gamma, \cal{F})$ such that when $t + 0i \in \Omega_{< \epsilon(\gamma, \cal{F})}$, we have $w^{2/p(t)} \in A_2 (\R)$. By \eqref{sd_01}, the operator is then indeed bounded on $L^2 (\R)$. 
\end{proof}

The above Proposition \ref{T_weakly_analytic_prop}, combined with Lemma \ref{proj_calF_ind_r}, yield the following.

\begin{Cor}\label{projections_analytic_nearL2}
  Suppose $[w]_{A_2 (\R)} \leq \gamma$ and $p(z)$ is as in \eqref{p(z)_initial} for some $p_* \in [1, \infty]$. Then there exists $\epsilon(\gamma)\in (0, \frac{1}{2})$, with $\lim\limits_{\gamma \to 1} \epsilon(\gamma) = \frac{1}{2}$, such that $w^{1/p(z)} \proj{0}{r}{} w^{-1/p(z)}, w^{-1/p'(z)} \proj{0}{r}{} w^{1/p'(z)}$ are analytic as maps $\Omega_{<\epsilon (\gamma)} \to \cal{L}(L^2 (\R))$.
\end{Cor}

\beginrmk In what follows, we will need to show some $\epsilon(\gamma) \in (0,\frac{1}{2})$ exists, such that $\lim\limits_{\gamma \to 1} \epsilon (\gamma) = \frac{1}{2}$. In our reasoning, we will consider various other such functions $\epsilon (\gamma)$. We consider only finitely many such functions, so by taking the minimum of all of them we will assume without loss of generality that all the $\epsilon(\gamma)$'s are the same.\smallskip

\begin{Cor}\label{Krein_Steklov_RHS_analytic}
  Suppose $[w]_{A_2 (\R)} \leq \gamma$ and $p(z)$ is as in \eqref{p(z)_initial} for some $p_* \in [1, \infty]$. Then there exists $ \epsilon(\gamma) \in (0,\frac{1}{2})$ with $\lim\limits_{\gamma \to 1} \epsilon(\gamma) = \frac{1}{2}$ such that 	\[
		\chi (I-Q_{w, p(z)})^{-1} w^{-1/p' (z)} \P_{[0,r]} w^{1/p'(z)} (w^{1/p (z)} - w^{-1/p' (z)}) \e{i \lambda r}
	      \] is analytic as a map $\Omega_{< \epsilon (\gamma)} \to L^2 (\R)$.
\end{Cor}
\begin{proof}
  By Corollary \ref{projections_analytic_nearL2}, Lemma \ref{integrability_Ap_weights_analytic} and Lemma \ref{thm_inverse_bd_final} with $T= \proj{0}{r}{}$ and e.g. $\mathbf{\Lambda} = 10$, we can choose $\epsilon (\gamma) \in (0, \frac{1}{2})$, with $\lim\limits_{\gamma \to 1} \epsilon (\gamma) = \frac{1}{2}$, such that $w^{-1/p' (z)} \P_{[0,r]} w^{1/p'(z)}$, $w^{1/p (z)} - w^{-1/p' (z)}$ and $(I-Q_{w, p(z)})^{-1}$ are analytic as elements or operators on $L^2 (\R)$, for $\Omega_{< \epsilon (\gamma)}$. A difference quotient computation then yields the Corollary statement. 
\end{proof}

\begin{proof}[Proof of Proposition \ref{solve_for_X_Steklov_fullp_statement}]
  We begin with \eqref{functional_eqn_A2}. By applying Lemma \ref{proj_calF_ind_r} and Lemma \ref{thm_inverse_bd_final} with $T = \proj{0}{r}{}$ and e.g.$\ \mathbf{\Lambda} = 10$, we can invert $I-Q_{w,p}$ to get \eqref{solve_for_X_Steklov_fullp} for all $p \in [2, \infty)$ satisfying $|\frac{1}{2} - \frac{1}{p} | < \epsilon (\gamma)$, where $\epsilon (\gamma) \in (0,\frac{1}{2})$ and $\lim\limits_{\gamma \to 1} \epsilon (\gamma) = \frac{1}{2}$. So for all $p\geq 2$ satisfying $|\frac{1}{2} - \frac{1}{p} | < \epsilon (\gamma)$, we have
	\begin{equation}\label{test_X}
	\langle \chi X_p, f \rangle = - \langle \chi (I-Q_{w,p})^{-1} w^{-1/p'} \P_{[0,r]} w^{1/p'} (w^{1/p} - w^{-1/p'}) \e{i \lambda r} , f \rangle
	\end{equation}
	for all $f \in \dense$.

	Now replace $p$ by $p(z)$ with $p_* = \infty$. Then by Corollary \ref{Krein_Steklov_RHS_analytic}, we have  
	\[
		\langle \chi (I-Q_{w,p(z)})^{-1} w^{-1/p'(z)} \P_{[0,r]} w^{1/p'(z)} (w^{1/p (z)} - w^{-1/p' (z)}) \e{i \lambda r} , f \rangle
		\]
		is analytic on $\Omega_{< \epsilon (\gamma)}$. Recall $X_2 \in L^2 (\R)$ by Proposition \ref{regularity_P}. Thus, if we write $\langle \chi X_{p(z)}, f \rangle = \langle \chi w^{1/p(z) - 1/2} X_2, f \rangle$, then by Proposition \ref{weak_analytic_weight_function} and that $X_2 \in L^2 (\R)$, we may assume $\langle \chi X_{p(z)}, f \rangle$ is analytic on the same region.

		Thus \eqref{test_X} shows two analytic functions are equals when $z$ is in the interval $\Omega_{< \epsilon (\gamma)} \cap \R$. Thus we must conclude that they are in fact equal on the strip $\Omega_{< \epsilon (\gamma)}$, i.e.\
\[
	\langle \chi X_{p(z)}, f \rangle = - \langle \chi (I-Q_{w,p(z)})^{-1} w^{-1/p'(z)} \P_{[0,r]} w^{1/p' (z)} (w^{1/p (z)} - w^{-1/p' (z)}) \e{i \lambda r} , f \rangle
	\]
	on $\Omega_{< \epsilon (\gamma)}$ for all $f \in \dense$. Since $\chi$ is the indicator of an arbitrary finite interval, then by duality and density of $\dense$ in $L^p (\R)$,
	\[
		X_{p(z)} = -(I-Q_{w,p(z)})^{-1} w^{-1/p'(z)} \P_{[0,r]} w^{1/p' (z)} (w^{1/p (z)} - w^{-1/p' (z)}) \e{i \lambda r} \,
		\]
		for $z \in \Omega_{< \epsilon (\gamma)}$. Taking $z = t$ yields \eqref{solve_for_X_Steklov_fullp} for $|\frac{1}{p} - \frac{1}{2}| < \epsilon (\gamma)$. 
\end{proof}

\beginrmk Both sides of \eqref{solve_for_X_Steklov_fullp} always makes sense as elements of $L^2 (\R)$ for  $|\frac{1}{p} - \frac{1}{2}| < \epsilon (\gamma)$. \smallskip

\bibliographystyle{plain}
\bibliography{bibfile}

\def\cprime{$'$} \def\cprime{$'$} \def\cprime{$'$}
\begin{thebibliography}{10}

\bibitem{AlexisAD1}
M.~Alexis, A.~Aptekarev, and S.~Denisov.
\newblock {Continuity of Weighted Operators, Muckenhoupt {$A_p$} Weights, and
  Steklov Problem for Orthogonal Polynomials}.
\newblock {\em International Mathematics Research Notices}, 2022(8):5935--5972,
  10 2020.

\bibitem{bsh}
C.~Bennett and R.~Sharpley.
\newblock {\em Interpolation of operators}, volume 129 of {\em Pure and Applied
  Mathematics}.
\newblock Academic Press, Inc., Boston, MA, 1988.

\bibitem{denisov_krein}
S.~Denisov.
\newblock {Continuous analogs of polynomials orthogonal on the unit circle and
  {K}reĭn systems}.
\newblock {\em International Mathematics Research Surveys}, 01 2006.

\bibitem{denik1}
S.~Denisov.
\newblock On the growth of polynomials orthogonal on the unit circle with a
  weight {$w$} that satisfies {$w,w^{-1}\in L^\infty(\Bbb T)$}.
\newblock {\em Mat. Sb.}, 209(7):71--105, 2018.

\bibitem{denik2}
S.~Denisov and K.~Rush.
\newblock Orthogonal polynomials on the circle for the weight {$w$} satisfying
  conditions {$w, w^{-1}\in{\rm BMO}$}.
\newblock {\em Constr. Approx.}, 46(2):285--303, 2017.

\bibitem{einsiedler2017functional}
M.~Einsiedler and T.~Ward.
\newblock {\em Functional Analysis, Spectral Theory, and Applications}.
\newblock Graduate Texts in Mathematics. Springer International Publishing,
  2017.

\bibitem{korey}
M.~B. Korey.
\newblock Ideal weights: asymptotically optimal versions of doubling, absolute
  continuity, and bounded mean oscillation.
\newblock {\em J. Fourier Anal. Appl.}, 4(4-5):491--519, 1998.

\bibitem{Krein1954}
M.~G. Krein.
\newblock Continuous analogues of propositions on polynomials orthogonal on the
  unit circle.
\newblock {\em Dokl. Akad. Nauk SSSR (N.S.)}, 105:637--640, 1955.

\bibitem{NazLer}
A.~K. Lerner and F.~Nazarov.
\newblock Intuitive dyadic calculus: the basics.
\newblock {\em Expo. Math.}, 37(3):225--265, 2019.

\bibitem{Pichorides1972}
S.~Pichorides.
\newblock On the best values of the constants in the theorem of {M}. {R}iesz,
  {Z}ygmund and {K}olmogorov.
\newblock {\em Studia Mathematica}, 44(2):165--179, 1972.

\bibitem{poltoratski2021pointwise}
Alexei Poltoratski.
\newblock Pointwise convergence of the non-linear fourier transform.
\newblock arXiv:2103.13349, 2021.

\bibitem{Simonbook}
B.~Simon.
\newblock {\em Orthogonal Polynomials on the Unit Circle. {P}art 1: {C}lassical
  {T}heory}.
\newblock Colloquium Publications. American Mathematical Society, 2004.

\bibitem{stein}
E.~M. Stein.
\newblock {\em Harmonic analysis: real-variable methods, orthogonality, and
  oscillatory integrals}, volume~43 of {\em Princeton Mathematical Series}.
\newblock Princeton University Press, Princeton, NJ, 1993.
\newblock With the assistance of Timothy S. Murphy, Monographs in Harmonic
  Analysis, III.

\bibitem{stein_singular_ints}
Elias~M. Stein.
\newblock {\em Singular Integrals and Differentiability Properties of Functions
  (PMS-30)}.
\newblock Princeton University Press, 2016.

\bibitem{Vasyunin}
V.~I. Vasyunin.
\newblock The exact constant in the inverse {H}\"{o}lder inequality for
  {M}uckenhoupt weights.
\newblock {\em Algebra i Analiz}, 15(1):73--117, 2003.

\end{thebibliography}

\end{document}